\documentclass[12pt]{article}
\usepackage[margin=0.75in]{geometry}
\usepackage{amsmath,amssymb,amsthm}
\usepackage{mathtools}
\usepackage{booktabs}
\usepackage{graphicx}
\usepackage{paralist}
\usepackage{placeins}
\usepackage[colorlinks=true]{hyperref}
\hypersetup{urlcolor=blue, citecolor=red}
\allowdisplaybreaks
\newtheorem{theorem}{Theorem}[section]
\newtheorem{corollary}[theorem]{Corollary}

\newtheorem{proposition}[theorem]{Proposition}
\newtheorem{conjecture}[theorem]{Conjecture}

\theoremstyle{definition}
\newtheorem{definition}[theorem]{Definition}
\newtheorem{remark}[theorem]{Remark}

\title{What remains when integrability fails? Mathematical structure in the K$(m,m)$ compacton equations}
\author{James M.\ Hyman\thanks{Department of Mathematics, Tulane University, New Orleans, LA 70118, USA.
Email: \href{mailto:mhyman@tulane.edu}{\texttt{mhyman@tulane.edu}}.}}
\date{Published in
\emph{Discrete and Continuous Dynamical Systems, Series A},\\
Peter D.\ Lax Centennial Memorial Volume,
\href{https://doi.org/10.3934/dcds.2026.105536}{doi:10.3934/dcds.2026.105536}.\\[4pt]
{\small Author's version, with the corrections made at proof.}}

\begin{document}

\maketitle
\begin{center}
\textit{Dedicated to the memory of Peter D.\ Lax (1926--2025),\\
my teacher and friend, who introduced me to soliton theory\\
and whose search for organizing principles shaped a generation of mathematics}
\end{center}

\bigskip

\begin{abstract}
Peter Lax's 1968 discovery that certain nonlinear wave equations
possess operator pairs (Lax pairs) opened the theory of
completely integrable systems, yielding exact solutions, infinite
conservation hierarchies, and perfectly elastic soliton collisions.
When a nonlinear dispersive equation is provably
non-integrable (admitting no Lax pair and only finitely many
conservation laws), what mathematical structure governs its dynamics?

The K(m,m) compacton equations, $u_t + (u^m)_x + (u^m)_{xxx} = 0$
for $m \geq 2$, provide a testing ground: they support traveling
waves with compact support, yet fail every known
integrability test.  We prove universal algebraic formulas
determining the shape, width, and speed of every K(m,m)
compacton.  We develop a constraint-deficit mechanism explaining how
the five independent conservation laws constrain emergent
amplitudes and suppress the fundamental Fourier mode of the collision
debris, while higher radiation modes remain unconstrained.
The mechanism's conclusions rest on explicit hypotheses
supported by numerical evidence.
A phase-plane analysis reveals the compacton as a singular limit
of smooth solitary waves through simultaneous regularity loss
and support compactification.
These results delineate intermediate territory between integrable
order and generic nonlinear dynamics, organized by scaling
universality and by the constraint deficit of a finite
conservation structure.
\end{abstract}



\section{Introduction}
\label{sec:introduction}

When does a nonlinear dispersive PDE possess systematic mathematical
structure? For three decades following Lax's 1968 discovery of
operator-pair compatibility conditions for the Korteweg--de Vries
equation~\cite{lax1968integrals,korteweg1895change}, the answer was
clear: when the equation is completely integrable. The inverse
scattering transform (IST) gave integrable systems a unified theory
producing exact multi-soliton solutions, infinite conservation
hierarchies, and algebraically determined collision
outcomes~\cite{gardner1967method,ablowitz1981solitons}. Equations
outside the integrable class were treated as generic nonlinear
problems requiring ad hoc numerical investigation.
Often implicitly, the field operated with a dichotomy: integrable
equations possess mathematical structure, and everything else does
not. This paper presents evidence that the dichotomy is incomplete.

Our evidence comes from the K(m,n) family of nonlinear dispersive
equations,
\begin{equation}\label{eq:kmn}
u_t + (u^m)_x + (u^n)_{xxx} = 0,
\end{equation}
introduced in 1993 \cite{rosenau1993compactons}.
Rosenau was the first to conceive the K(m,n) equations and to unfold
their special properties, and he went on to make major advances in
understanding them~\cite{olver1996tri,rosenau1994nonlinear,rosenau1997nonanalytic,%
rosenau1998class,rosenau2000compact,rosenau2018compactons}.
The case K(2,1) recovers KdV, and K(3,1) corresponds to the modified
KdV equation after scaling; these are the only integrable members
of the family with a standard polynomial
nonlinearity~\cite{heredero2020nonexistence}.
(Exotic integrable cases with negative exponents, such as
K$(-2,-2)$ and K$(-1/2,-1/2)$, arise in the tri-Hamiltonian duality
of Olver and Rosenau~\cite{olver1996tri} but lie outside the scope
of this paper.)
Equal convective and dispersive exponents ($m = n$) with $m > 1$ are
sufficient for equation~\eqref{eq:kmn} to admit traveling
wave solutions with genuinely compact support: \emph{compactons},
which vanish identically beyond a finite boundary rather than decaying
exponentially.  The mechanism enabling compact support is the
amplitude-dependent dispersion in the term~$(u^m)_{xxx}$, which
degenerates where $u \to 0$ and cannot sustain exponential tails.
The exponent~$m$ need not be an integer for compacton existence;
for non-integer~$m$, the absolute-value formulation
$|u|^{m-1}u$ extends the equation to solutions that may change
sign~\cite{garralon2023robustness}.
All K(m,m) equations with real $m > 1$ are provably
non-integrable; on the diagonal $m = n$, compactification and
non-integrability are coextensive.

What is surprising is what these non-integrable equations possess.
K(m,m) compactons satisfy universal algebraic formulas connecting
the three compacton parameters (the edge regularity
exponent~$\alpha$, the width parameter~$\ell$, and the
speed~$c$) to the single
nonlinearity exponent~$m$.  Specifically, K(m,m) compactons
take the form
\[
u(x,t) = A\cos^{\alpha}\!\!\left(\frac{x - ct}{\ell}\right),
\qquad |x - ct| \leq \tfrac{\pi\ell}{2},
\]
vanishing identically outside this interval, where $\alpha$
governs the rate of approach to zero at the support
boundary.  We prove (Theorem~\ref{thm:universal})
that for all $m > 1$:
\begin{equation}\label{eq:formulas-preview}
\alpha = \frac{2}{m-1}, \qquad
\ell = \frac{2m}{m-1}, \qquad
c = \frac{m+1}{2m}\,A^{m-1}.
\end{equation}
These formulas were identified through pattern recognition across
individual cases~$m = 2,\ldots,7$~\cite{rosenau1993compactons,%
rosenau1994nonlinear,cooper2001compacton,rosenau2018compactons},
but their status had remained that of empirical regularities.
The proof establishes them as algebraic consequences of the
cosine-power ansatz and the traveling wave ODE, valid for all
real~$m > 1$.

These universal formulas are one component of a richer mathematical
picture. The K(m,m) equations possess five known independent
conservation laws (exactly four local~\cite{vodova2013complete}
and one nonlocal~\cite{rosenau2018compactons}), a count lying strictly between
the two conservation laws (mass and energy) that generic Hamiltonian
PDEs possess and the infinite hierarchy of integrable systems.
The equations admit single Hamiltonian structure but not the
bi-Hamiltonian pair required for Magri's recursion
operator~\cite{magri1978simple}. Compacton collisions are
quasi-elastic~\cite{rosenau1993compactons,rosenau2018compactons}:
in the collisions computed here the emerging waves keep their
amplitudes to within $0.5\%$, and a small-amplitude debris field
carries the remainder (\S\ref{subsec:quasi-elastic}). This is
not residual integrability. It is a different kind of mathematical
order, governed by a finite conservation structure rather than an
infinite one.

This paper makes three contributions.

\emph{First,} we synthesize the universal formulas, conservation
structure, Hamiltonian geometry, and collision mechanics of K(m,m)
compactons into a coherent mathematical framework. The individual
components have been developed
separately~\cite{rosenau1993compactons,rosenau1994nonlinear,%
vodova2013complete,cooper2001compacton,%
rosenau2018compactons}; bringing them together reveals
their interconnections and the pattern they form.

\emph{Second,} we prove results that sharpen the
framework. Theorem~\ref{thm:universal} provides a unified derivation
of the universal formulas for all~$m > 1$; the individual formulas
appear in the literature~\cite{rosenau1993compactons,rosenau2018compactons},
but their simultaneous emergence from a single traveling-wave
analysis appears to be new.  This derivation yields consequences
that had not been previously identified: a
trigonometric orthogonality property, monotonicity of the mass and
of the translation generator along the compacton family
(Corollary~\ref{cor:VK}), and the exclusivity of the cosine-power
ansatz to the K(m,m) subfamily.
Theorem~\ref{thm:constraint-deficit},
under hypotheses (H1)--(H3) supported by numerical evidence,
identifies a mechanism behind imperfect collision elasticity
in the K(m,m) family:
the five conservation laws constrain emergent amplitudes and
force the fundamental Fourier mode of collision debris to
vanish, but leave all higher radiation modes unconstrained.
Numerical spectral analysis across the family confirms this
prediction: post-collision debris exhibits a characteristic
notch at wavenumber $k = 1$ for all tested values of~$m$.
This constraint-deficit framework is introduced here;
it explains both the
algebraic order (why universal formulas exist) and the
dynamical imperfection (why K(m,m) collisions are near-elastic
but not exactly elastic).

\emph{Third,} we propose a conservation--elasticity framework
(Conjecture~\ref{conj:monotonicity}) relating collision
elasticity to both the conservation law count and the
interaction geometry of solitary waves.  Quantitative comparison
of collisions in K(m,m) and in the Benjamin--Bona--Mahony (BBM)
equation~\cite{benjamin1972model}, using identical methodology, shows
that BBM solitary waves ($N = 3$ known conservation laws, smooth
exponential tails) collide about as elastically as K(m,m) compactons
($N = 5$ conservation laws, compact support) by the amplitude measure,
while shedding more debris energy.  Whether interaction geometry
modulates the role of the conservation law count therefore remains
open.
We also pose eight open problems
defining a research program for investigating mathematical
structure in K(m,m) and in other non-integrable
systems (Appendix~\ref{app:open-problems}).

An arc runs through this paper, from Lax's integrable vision through
its failure for K(m,m) to what replaces it.  The author had the privilege
of being Peter Lax's student, colleague, and friend, and this
contribution continues a conversation about structure in nonlinear
waves that began in his office at the Courant Institute.  Lax's contributions to
nonlinear wave theory extended well beyond the operator-pair
construction~\cite{lax2006hyperbolic}: his broader program sought
organizing principles for nonlinear PDEs wherever such principles
might exist.
For KdV and its integrable relatives, the IST remains the organizing
principle. For K(m,m), a different organizing principle operates:
scaling universality and finite conservation structure produce
systematic, predictable behavior without recourse to
spectral theory.

Figure~\ref{fig:roadmap} provides a conceptual roadmap for the
paper's argument.

\begin{figure}[t]
\centering
\includegraphics[width=\textwidth]{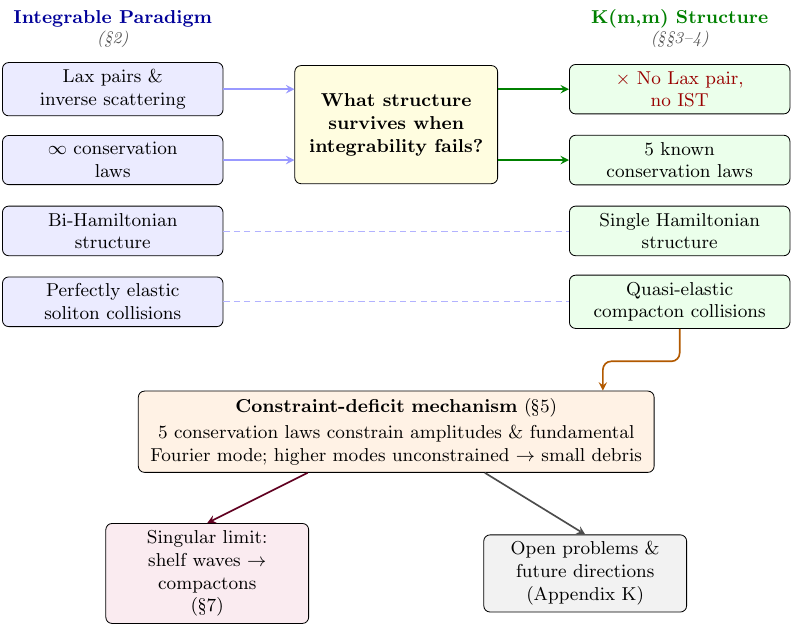}
\caption{Conceptual roadmap.  The integrable paradigm
(left column) provides four structural features that serve as
reference points.  K(m,m) equations lack the first two
(\S\ref{sec:nonintegrability}) but retain analogues of the
remaining two (\S\ref{sec:structure-beyond}).
The constraint-deficit mechanism
(\S\ref{sec:collision-dynamics}) links the finite conservation
count to the quasi-elastic collision dynamics.
The singular limit (\S\ref{sec:shelf-waves}) reveals the
compacton as a degeneration of smooth shelf waves, traveling waves
on a constant background $u = a > 0$.}
\label{fig:roadmap}
\end{figure}

Section~\ref{sec:integrable-paradigm} reviews the integrable
paradigm, presenting Lax pairs, the IST, and the four structural
features of integrable systems that will serve as reference
points throughout.  Section~\ref{sec:nonintegrability} introduces
the K(m,m) compacton family and establishes non-integrability
through four independent lines of evidence.
Section~\ref{sec:structure-beyond} presents the universal formulas
with their proof, the regularity classification, the finite
conservation hierarchy, and Hamiltonian geometry.
Section~\ref{sec:collision-dynamics} develops the collision theory,
constraint-deficit mechanism, and BBM comparison, concluding with
broader implications of the constraint-deficit framework for
non-integrable dispersive PDEs.
Section~\ref{sec:emergence} examines compacton emergence from
localized initial data, demonstrating that the universal speed
formula is recovered to within one percent in dynamics that begin
far from any compacton profile, and identifying a resolution limit
controlled by the boundary regularity exponent $\alpha = 2/(m-1)$.
Section~\ref{sec:shelf-waves} reveals a phase-plane bifurcation
governed by a critical speed and identifies the compacton as a
singular limit of smooth waves, providing a precise connection to Lax's 1975
periodic KdV theory~\cite{lax1975periodic}.
Eight concrete open problems are posed in
Appendix~\ref{app:open-problems}, with extended versions in the
Supplement.
Detailed derivations and extended numerical results appear in
Appendices~A--K below; computational data, simulation code,
and reproducibility materials are collected in the accompanying
computational supplement.
Readers familiar with the integrable paradigm may proceed from
Section~\ref{sec:introduction} directly to
Section~\ref{sec:nonintegrability}; the contributions of this paper
begin in Section~\ref{sec:structure-beyond}.

\paragraph{Notation.}
Throughout, $m$ and $n$ denote the nonlinearity exponents in the
K(m,n) family~\eqref{eq:kmn}.  The compacton existence results
(Theorem~\ref{thm:universal}) and non-integrability
(\S\ref{sec:nonintegrability}) hold for all real $m > 1$; the
conservation law, Hamiltonian, and collision analyses
(\S\S\ref{subsec:conservation-laws}--\ref{subsec:quasi-elastic})
are established for integer $m \geq 2$.
For traveling waves, $\xi = x - ct$ is the co-moving
coordinate, $c > 0$ the propagation speed, and $A > 0$ the
amplitude.  The compacton profile parameters are the edge
exponent~$\alpha$, the width parameter~$\ell$ (so that the support
is $|\xi| \leq \pi\ell/2$), and the speed
coefficient $k(m) = c/A^{m-1}$.
Conservation integrals are labeled $I_0, \ldots, I_4$ following
the ordering in~\S\ref{subsec:conservation-laws}.

\section{The integrable paradigm}
\label{sec:integrable-paradigm}

Few developments in twentieth-century mathematics have been as
surprising as the discovery that certain nonlinear PDEs possess
hidden linear structure.  The story begins with a numerical
accident: Fermi, Pasta, Ulam, and Tsingou observed unexpected
recurrence in a nonlinear lattice~\cite{fermi1955studies}.
Zabusky and Kruskal traced the recurrence to particle-like
waves that survive mutual collisions
intact~\cite{zabusky1965interaction}.  Gardner, Greene, Kruskal,
and Miura then showed that the Korteweg--de Vries equation could
be solved exactly by an inverse scattering
method~\cite{gardner1967method}, and Lax distilled the mechanism
to its algebraic essence: an operator pair whose compatibility
condition is equivalent to the
PDE~\cite{lax1968integrals}.  This section reviews the
integrable framework that emerged from these discoveries,
emphasizing the specific features that will serve as reference
points when we examine the K(m,m) family in
Sections~\ref{sec:nonintegrability}
and~\ref{sec:structure-beyond}.

\subsection{Lax pairs and the KdV archetype}
\label{subsec:lax-pairs}

The KdV equation, written in K(2,1) normalization as
\begin{equation}
  u_t + (u^2)_x + u_{xxx} = 0,
  \label{eq:kdv}
\end{equation}
balances nonlinear convection $(u^2)_x = 2uu_x$ against linear
dispersion~$u_{xxx}$.  Its soliton solutions
\begin{equation}
  u(x,t) = 6\kappa^2 \operatorname{sech}^2\bigl(\kappa(x - 4\kappa^2 t)\bigr)
  \label{eq:kdv-soliton}
\end{equation}
decay exponentially as $|x| \to \infty$, propagate at speed $c = 4\kappa^2$
proportional to the amplitude $A = 6\kappa^2$, and interact elastically:
two solitons emerge from a collision with their shapes and speeds intact,
displaced only by a phase shift~\cite{zabusky1965interaction}.

These properties are not coincidental. They follow from a deep
structural fact discovered by Lax~\cite{lax1968integrals}: KdV
admits a \emph{Lax pair} representation.  Associated with
equation~\eqref{eq:kdv} are two linear operators,
\begin{equation}
  L = -\partial_x^2 + \tfrac{1}{3}u, \qquad
  M = -4\partial_x^3 + 2u\,\partial_x + u_x,
  \label{eq:lax-pair}
\end{equation}
whose compatibility condition $L_t = [M, L]$ is equivalent to the
nonlinear PDE.  The operator $L$ is a Schr\"odinger eigenvalue problem
(with potential $u/3$);
$M$ governs the time evolution of the eigenfunctions.  Together they
convert a nonlinear initial-value problem into a sequence of linear
operations: decompose initial data into spectral components of~$L$,
evolve these components linearly through~$M$, and reconstruct the
solution at any later time through the inverse scattering
transform~\cite{gardner1967method,ablowitz1981solitons}.

The geometric content of this construction deserves emphasis.
The Lax equation $L_t = [M,L]$ asserts that the spectrum
of~$L$ is preserved under the time evolution defined by~$M$.
KdV dynamics are \emph{isospectral deformations}: the potential
in~$L$ changes, but its eigenvalues remain fixed.  Each eigenvalue
corresponds to a soliton, and spectral invariance explains why
solitons retain their identities through interactions.  The phase
shift experienced during collision is encoded in the scattering
matrix, which evolves linearly in the IST
framework~\cite{ablowitz1981solitons}.

This is not merely a technique for solving KdV.  It reveals that
the equation possesses an organizing principle, namely spectral
theory, governing its entire solution space.  Multi-soliton
solutions follow from $N$-fold Darboux transformations or
Hirota's bilinear method~\cite{hirota1971exact},
both deriving from the same Lax pair structure.  The organizing
principle is \emph{algebraic}: the spectral data form a complete
set of coordinates in which KdV becomes a linear flow.

\subsection{What integrability buys}
\label{subsec:integrability-features}

The Lax pair structure ramifies into four concrete mathematical
consequences, each of which plays a distinct role in the theory of
integrable dispersive equations.

\paragraph{An infinite conservation hierarchy.}
Lax~\cite{lax1968integrals} showed that the isospectral property
generates an infinite sequence of conserved quantities
$I_0, I_1, I_2, \ldots$ for KdV (in this section only; from
\S\ref{subsec:conservation-laws} on, $I_0, \ldots, I_4$ denote the
K(m,m) laws). The first three have direct
physical interpretations: mass $I_0 = \int u\,dx$, momentum
$I_1 = \int u^2\,dx$, and energy
$I_2 = \int (\tfrac{1}{3}u^3 - \tfrac{1}{2}u_x^2)\,dx$.  Higher
conservation laws involve progressively complex differential
expressions, but the sequence continues without termination.

These infinitely many constraints overdetermine the dynamics in a
precise sense.  For a system with $N$ solitons, $N$ spectral
parameters characterize the solution completely.  The infinite
conservation laws provide far more than $N$ constraints, encoding
not only the discrete spectrum (soliton parameters) but also the
continuous spectrum (radiation).  Complete integrability means that
the conservation laws form a maximal involutive set under the
Poisson bracket, in the sense of Liouville--Arnold
theory~\cite{arnold1989mathematical}.

\paragraph{Perfectly elastic collisions.}
When two KdV solitons collide, both emerge with exactly their
pre-collision amplitudes, speeds, and shapes.  The only trace of
the interaction is a phase shift: each soliton is displaced from
where it would have been without the encounter.  No radiation is
generated, no energy is lost to dispersive tails, and no amplitude
change occurs.  This perfect elasticity holds for any number of
simultaneously interacting solitons and is a direct consequence of
the IST: spectral data are preserved under the nonlinear evolution,
so the asymptotic soliton parameters cannot
change~\cite{ablowitz1981solitons}.

\paragraph{Bi-Hamiltonian structure.}
Magri~\cite{magri1978simple} and Fokas and
Fuchssteiner~\cite{fokas1981symplectic} discovered that KdV admits
not one but two compatible Hamiltonian formulations.  Two Poisson
operators
\begin{equation}
  \mathcal{J}_1 = -\partial_x, \qquad
  \mathcal{J}_2 = -\partial_x^3
    - \tfrac{2}{3}(u\,\partial_x + \partial_x\, u)
  \label{eq:bi-ham}
\end{equation}
generate the same dynamics through different Hamiltonians:
$u_t = \mathcal{J}_1\,\delta H_1/\delta u
     = \mathcal{J}_2\,\delta H_2/\delta u$,
where $H_1 = \int(\tfrac{1}{3}u^3 - \tfrac{1}{2}u_x^2)\,dx$ and
$H_2 = \tfrac{1}{2}\!\int u^2\,dx$.
The recursion operator
$\mathcal{R} = \mathcal{J}_2\,\mathcal{J}_1^{-1}$
produces the infinite conservation hierarchy automatically through
the Lenard scheme~\cite{magri1978simple,fokas1981symplectic}, connecting conservation laws at successive
differential orders.  Bi-Hamiltonian structure is not merely a
reformulation; it provides the \emph{mechanism} generating the
infinite hierarchy.

\paragraph{Exact multi-soliton formulas.}
The IST yields closed-form $N$-soliton solutions through
Hirota's bilinear method~\cite{hirota1971exact} or the
Gel'fand--Levitan--Marchenko integral
equation~\cite{ablowitz1981solitons}, expressing the complete
$N$-body interaction algebraically in terms of spectral parameters.

Together, these four features constitute the \emph{integrable
paradigm}: spectral theory organizes the entire solution space
through conservation laws, elastic collisions, bi-Hamiltonian
geometry, and exact formulas.

\subsection{The prevailing dichotomy}
\label{subsec:dichotomy}

The integrable paradigm has shaped how the nonlinear dispersive wave
community organizes its subject.  Textbooks tend to divide their material
into integrable equations (treated analytically) and non-integrable
equations (treated
numerically)~\cite{ablowitz1981solitons,drazin1989solitons}.
An implicit assumption pervades the field: equations either possess
the full integrable apparatus or they lack systematic analytical
theory altogether.  Equations admitting Lax pairs receive extensive
analytical attention; those failing the Painlev\'e test or the
bi-Hamiltonian criterion are often relegated to numerical simulation.

The K(m,m) compacton family challenges this dichotomy
fundamentally.  The challenge is sharpened by recent work of
Anco and Gandarias~\cite{anco2024weak}, who developed a rigorous
weak (distributional) formulation for compacton solutions of
nonlinearly dispersive KdV-type equations.  Their framework
establishes that compactons (which may fail to be classical
solutions at the support boundary) satisfy the PDE in a
precise distributional sense.  This places the analytical theory
of non-integrable compacton equations on firmer ground:
the structures developed in the following sections rest on a
well-defined solution concept rather than purely formal
manipulations.  As the next section demonstrates, these equations
fail every standard test of integrability, yet the mathematical
structure they exhibit cannot be dismissed as coincidence.

\section{The K(m,m) family: non-integrability established}
\label{sec:nonintegrability}

This section surveys established results on K(m,m) non-integrability:
the compacton construction~\cite{rosenau1993compactons}, the
integrability classification of Hern\'andez Heredero et
al.~\cite{heredero2020nonexistence}, and Vodov\'a's conservation law
analysis~\cite{vodova2013complete}.
The contribution here is not new theorems but a unified presentation
connecting these results to the paper's main theme: what mathematical
structure persists when integrability fails.

The generalized KdV family
\begin{equation}
  u_t + (u^m)_x + (u^n)_{xxx} = 0, \qquad m, n > 1,
  \label{eq:kmn-general}
\end{equation}
introduced in~\cite{rosenau1993compactons},
provides an ideal test case.  This two-parameter extension of KdV
allows both the convective and
dispersive nonlinearities to be independently specified.  When
$m = 2$ and $n = 1$, equation~\eqref{eq:kmn-general} reduces to
the classical KdV.  When $m = n$, a qualitatively new phenomenon
appears: solitary waves with
genuinely compact support.

\subsection{Compactons: a new kind of localization}
\label{subsec:compactons}

Consider the K(m,m) equation
\begin{equation}
  u_t + (u^m)_x + (u^m)_{xxx} = 0, \qquad m > 1.
  \label{eq:kmm}
\end{equation}
The critical distinction from KdV lies in the dispersive term
$(u^m)_{xxx}$.  In KdV, the linear dispersion $u_{xxx}$ operates
at all amplitudes: even where $u$ is negligibly small, dispersive
spreading persists.  In K(m,m), the dispersive strength scales as
$u^{m-1}$, vanishing wherever the solution vanishes.  When
dispersion shuts off at $u = 0$, nothing prevents the solution
from remaining identically zero beyond a finite boundary.  Compact
support arises not from artificial truncation or imposed boundary
conditions, but from the intrinsic structure of the equation.

This mechanism produces \emph{compactons}: solitary waves that
vanish identically outside a finite
interval~\cite{rosenau1993compactons}.  The term, coined by
Rosenau, reflects the compact
support that distinguishes these waves from classical solitons.
Amplitude-dependent dispersion and the resulting compact support
also arises in the dynamics of granular
chains~\cite{nesterenko1983propagation,nesterenko2001dynamics},
giving the K(m,n) family physical as well as mathematical interest.
Subsequent investigations by
Rosenau~\cite{rosenau1997nonanalytic,rosenau1998class,rosenau2000compact}
explored the analytical properties of these nonanalytic solutions
and their role within broader families of nonlinear dispersive
interactions.

To make the compact support mechanism concrete, consider the
traveling wave reduction.  Substituting $u(x,t) = U(x - ct)$ into
equation~\eqref{eq:kmm} and integrating once yields the ODE
\begin{equation}
  -cU + U^m + (U^m)'' = 0,
  \label{eq:tw-ode}
\end{equation}
where primes denote derivatives with respect to
$\xi = x - ct$.  At any point where $U = 0$, the nonlinear
terms vanish: the equation becomes $0 + 0 + 0 = 0$, satisfied
trivially.  A solution that reaches $U = 0$ at some finite
$\xi_0$ can continue as $U \equiv 0$ for all
$|\xi| > \xi_0$ without violating the equation.  No exponential
tail is forced, because no linear mechanism drives the solution
away from zero.  The general ansatz
$U(\xi) = A\cos^\alpha(\xi/\ell)$ for $|\xi| \leq \pi\ell/2$
captures this structure, with the three parameters $\alpha$,
$\ell$, and $A$ determined by the ODE.

\paragraph{K(2,2).}
The simplest case, $u_t + (u^2)_x + (u^2)_{xxx} = 0$, admits
traveling wave solutions of the form
\begin{equation}
  u(x,t) = \frac{4c}{3}\cos^2\!\biggl(\frac{x - ct}{4}\biggr)
  \quad\text{for } |x - ct| \leq 2\pi,
  \label{eq:k22-solution}
\end{equation}
and $u \equiv 0$ elsewhere.  The amplitude is $A = 4c/3$ and the
speed--amplitude relation is $c = (3/4)A$, linear in amplitude.
The solution vanishes quadratically at its edges: near the boundary
$\xi_0 = 2\pi$, the profile behaves as
$u \sim B(\xi_0 - \xi)^2$, maintaining continuous first
derivatives at the junction with the exterior zero state.
The support has fixed width $4\pi$, independent of
amplitude~\cite{rosenau1993compactons}.

\paragraph{K(3,3).}
Cubic nonlinearity changes the edge character.  The equation
$u_t + (u^3)_x + (u^3)_{xxx} = 0$ admits
\begin{equation}
  u(x,t) = A\cos\!\biggl(\frac{x - ct}{3}\biggr)
  \quad\text{for } |x - ct| \leq \frac{3\pi}{2},
  \label{eq:k33-solution}
\end{equation}
with $c = (2/3)A^2$.  The edge exponent drops to
$\alpha = 1$: the solution vanishes linearly, producing cusp-like
boundaries where the first derivative jumps
discontinuously~\cite{rosenau1994nonlinear}.  Despite this
reduced regularity, the compacton propagates without change of
form and retains its compact support indefinitely.  The support
width is $3\pi$, narrower than $4\pi$ for K(2,2), again
independent of amplitude.  The quadratic speed--amplitude
relation $c = (2/3)A^2$ means that K(3,3) compactons of
different amplitudes separate more rapidly than their K(2,2)
counterparts, where the relation is linear.

\paragraph{K(4,4) and beyond.}
The pattern continues with increasing nonlinearity.
K(4,4) admits solutions $u = A\cos^{2/3}(\xi/\ell)$ with
fractional edge exponent $\alpha = 2/3$ and support width
$(8/3)\pi$.  Derivatives diverge at the boundary, requiring
weak solution formulations for mathematical
rigor~\cite{rosenau2018compactons}.  Subsequent cases K(5,5)
through K(7,7) exhibit progressively sharper edges (with
$\alpha = 1/2, 2/5, 1/3$ respectively) while maintaining
the same compact support structure.  Cooper, Shepard, and
Sodano~\cite{cooper1993solitary} studied these generalized solitary
wave solutions within a broader class of KdV-type equations,
confirming the systematic parameter relationships across the family.

Figure~\ref{fig:compacton-profiles} displays representative
compacton profiles for K(2,2), K(3,3), and K(4,4), each with unit
amplitude.  The regularity transition is visible at the support
boundary: K(2,2) joins the zero exterior smoothly with vanishing
slope, K(3,3) meets it at a cusp with nonzero slope, and K(4,4)
arrives with diverging derivative.  All three profiles have the
cosine-power form $U(\xi) = \cos^{\alpha}(\xi/\ell)$ from the
traveling wave ansatz, with the amplitude-independent widths
$W = \pi\ell$ decreasing as $m$ increases.

\begin{figure}[ht]
\centering
\includegraphics{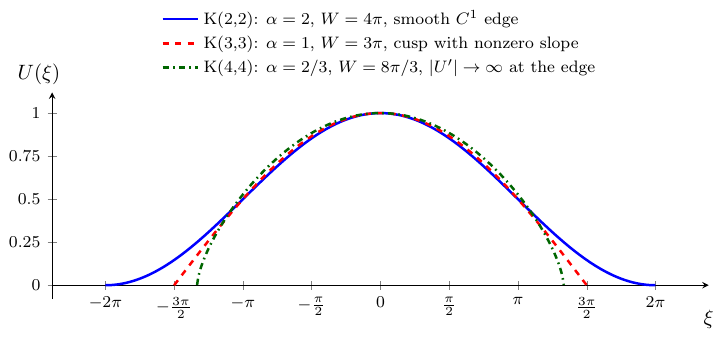}
\caption{Compacton profiles $U(\xi) = \cos^{\alpha}(\xi/\ell)$
for K(2,2), K(3,3), and K(4,4), each with unit amplitude $A = 1$.
All three vanish identically outside their support intervals.
Boundary regularity degrades with increasing~$m$:
K(2,2) joins the exterior smoothly ($\alpha = 2$, $U' = 0$ at the edge),
K(3,3) forms a cusp ($\alpha = 1$, $U' \neq 0$),
and K(4,4) has divergent slope ($\alpha = 2/3$, $|U'| \to \infty$).
The support width $W = \pi\ell = 2\pi m/(m-1)$
decreases monotonically with~$m$.}
\label{fig:compacton-profiles}
\end{figure}

Three features distinguish compactons from KdV solitons.
First, compactons vanish identically beyond a finite boundary
determined by the equation structure, rather than decaying
exponentially over infinite extent.  Second, their spatial width
is amplitude-independent: all K(m,m) compactons of a given~$m$
occupy the same interval regardless of how tall they are.  Third,
edge regularity depends on~$m$, transitioning from smooth
($m = 2$) through cusped ($m = 3$) to singular-derivative
($m \geq 4$).

\subsection{The Hern\'andez Heredero--Euler--Reyes classification}
\label{subsec:her-classification}

Given these striking traveling wave solutions, a natural question
arises: might K(m,m) equations be integrable despite their
departure from KdV?  Could some hidden Lax pair or recursion
operator exist, restoring the full integrable apparatus?

The definitive answer is no.  Hern\'andez Heredero, Euler, Euler,
and Reyes~\cite{heredero2020nonexistence} resolved the
integrability question for the entire K(m,n) family through a
systematic analysis of formal symmetry structure.

Their classification is exhaustive: integrable K(m,n) equations
exist only when the dispersive exponent takes one of three values,
$n \in \{1,\; -2,\; -\tfrac{1}{2}\}$, with a finite list of
admissible convective exponents~$m$ in each case.  These include
the familiar K(2,1)~=~KdV and K(3,1)~$\sim$~mKdV
(related through a Miura-type
transformation~\cite{miura1968korteweg}) on the $n = 1$
line, together with exotic negative-exponent cases such as
K($-2$,$-2$), K($-\tfrac{1}{2}$,$-\tfrac{1}{2}$), and
K($-1$,$-2$) that are connected to KdV through
nonlocal differential substitutions.  We extract the consequence
relevant to the present work.

\begin{theorem}[Non-integrability of compacton equations
  {\cite{heredero2020nonexistence}}]
\label{thm:her-classification}
No K(m,n) equation~\eqref{eq:kmn-general} with $m > 1$ and
$n > 1$ possesses a recursion operator.  In particular, no such
equation is completely integrable.
\end{theorem}

\begin{proof}
The Hern\'andez Heredero--Euler--Reyes classification proves that
K(m,n) admits a recursion operator only when
$n \in \{1, -2, -\frac{1}{2}\}$.  Since $n > 1$ excludes all three
values, the result follows immediately.
\end{proof}

The proof of the full classification proceeds through compatibility
conditions for formal recursion operators acting on symmetry spaces.
For integrable equations, these conditions hold at every
differential order.  For non-integrable cases, contradictions
emerge at finite order, terminating the symmetry hierarchy.

The critical consequence for the present work is this: \emph{every
K(m,m) equation with real $m > 1$ is non-integrable.}  On the
diagonal $m = n$ the only integrable members are $m = -2$ and
$m = -\frac{1}{2}$, both with negative exponents and no compacton
solutions, together with the degenerate linear case $m = 1$.
Because the classification treats $m$ and $n$ as real parameters
(note the non-integer integrable case $m = -\frac{1}{2}$), the
conclusion holds for all real $m > 1$, not merely for integers.
The matched nonlinearity, with dispersion power equal to convection
power, yields compactons and also destroys integrability: no K(m,m) equation supports both
compact solitary waves and the algebraic machinery of complete
integrability.

All integrable K(m,n) equations are connected to KdV through
differential substitutions, establishing K(2,1) as the archetypal
progenitor of the integrable subfamily.  K(m,m) equations with
$m > 1$, by contrast,
stand isolated: no Miura-type transformation connects them to any
integrable system~\cite{heredero2020nonexistence}.  This isolation
is not merely the absence of a known transformation.  The
Hern\'andez Heredero--Euler--Reyes analysis proves that no such transformation
can exist, because the underlying symmetry structure is
incompatible with infinite-dimensional recursion.

\subsection{Four proofs of non-integrability}
\label{subsec:four-proofs}

The classification theorem settles the question through formal
symmetry methods.  But K(m,m) non-integrability is overdetermined:
four independent lines of evidence converge on the same conclusion.

\paragraph{(a) Finite conservation hierarchy.}
Vodov\'a~\cite{vodova2013complete} proved structurally that K(m,m)
equations with $m \neq -2, -1/2, 0, 1$ admit exactly four local
conservation laws: mass, energy ($\int u^{m+1}/(m+1)\,dx$), and
two trigonometric functionals.  Her proof establishes that any
conserved density depends at most on $x$, $t$, and~$u$ (no
derivatives), reducing the classification to a finite problem.
A fifth conserved quantity (the nonlocal momentum $I_4$ of Rosenau and
Zilburg~\cite{rosenau2018compactons}) exists
but falls outside the scope of local classification.  Explicit
forms of all five appear in Section~\ref{subsec:conservation-laws}.
Vodov\'a's structural proof applies to all real $m \neq -2, -1/2, 0, 1$,
not only integer exponents.

Integrable equations possess infinite conservation hierarchies
generated by recursion operators.  The abrupt termination at five
laws, with no structural pattern suggesting
continuation, provides direct evidence of non-integrability
through conservation law counting.

\paragraph{(b) Absence of Lax pairs.}
Systematic searches for differential operators $L$ and $M$
satisfying $L_t = [M, L]$ with K(m,m) dynamics find no valid
candidates.  The Hern\'andez Heredero--Euler--Reyes
classification~\cite{heredero2020nonexistence} establishes through
formal symmetry analysis that no recursion operator exists for
K(m,m) with $m > 1$, which precludes the existence of Lax pairs.
Applied to the K(m,m) family, the analysis shows that
compatibility conditions for recursion operators fail at finite
differential order.  Each candidate either fails the
compatibility condition or produces dynamics inconsistent with K(m,m).
Without a Lax pair, the inverse scattering transform cannot be
constructed, and the spectral theory that organizes integrable
dynamics is unavailable.

\paragraph{(c) Painlev\'e test failure.}
The Painlev\'e property requires that all movable singularities of
solutions in the complex plane be poles.  Dominant balance analysis
of the K(m,m) traveling wave ODE reveals movable branch points for
every $m > 1$: the leading singularity exponent $-2/(m-1)$ is
non-integer except at $m = 2$ and $m = 3$, and even those cases
fail at higher-order resonance conditions.  The detailed
expansion is given in the Supplement, \S S6.

\paragraph{(d) No bi-Hamiltonian structure.}
K(m,m) equations admit a single Hamiltonian formulation through
the third-order Poisson operator $\mathfrak{D} = -(D_x^3 + D_x)$, with $D_x = \partial/\partial x$
(Section~\ref{subsec:hamiltonian-geometry}).  The bi-Hamiltonian
criterion asks whether a compatible second Poisson operator exists,
such that a recursion operator can be formed from the pair.
Compatibility would generate the infinite
conservation hierarchy through the Lenard recursion
scheme~\cite{magri1978simple,fokas1981symplectic}.

No such second operator exists.  Vodov\'a's
proof~\cite{vodova2013complete} that K(m,m) possesses no
generalized symmetries of order greater than~$3$ implies that no
recursion operator exists, which in turn precludes a compatible
bi-Hamiltonian pair.  The single Hamiltonian structure persists,
but its extension to a compatible pair is impossible.

These four independent failures reinforce each other and the
Hern\'andez Heredero--Euler--Reyes classification.  Among them, the
classification theorem and the Painlev\'e analysis cover all
real $m > 1$, while Vodov\'a's structural proof covers all
real $m \neq -2, -1/2, 0, 1$
(Theorem~\ref{thm:finite-hierarchy}).  For integer exponents, all four arguments
converge; for non-integer $m \in (1,2) \cup (2,\infty)$, the
classification and Painlev\'e test suffice to establish
non-integrability independently.

The point is
not that a Lax pair has merely \emph{not been found}, or that
no one has yet succeeded in extending the conservation hierarchy
beyond five laws.  
The Hern\'andez Heredero--Euler--Reyes classification
proves, through finite-order obstruction arguments, that the
required algebraic structures \emph{cannot exist} for any
real $m > 1$.  K(m,m) non-integrability is a theorem, not a
conjecture, and it holds throughout the compacton-bearing
parameter range.

\medskip

A reader familiar with the integrable paradigm of
Section~\ref{sec:integrable-paradigm} now faces a clear
question.  K(m,m) compactons are elegant mathematical objects with
explicit traveling wave solutions, systematic parameter
relationships, and a finite but nontrivial conservation structure.
Yet they are provably non-integrable: no Lax pairs, no inverse
scattering, no infinite conservation hierarchy, no bi-Hamiltonian
geometry.  The standard theoretical apparatus is unavailable.

What mathematical structure, if any, remains?

\section{Structure beyond integrability}
\label{sec:structure-beyond}

Having established what K(m,m) compactons lack
(\S\ref{sec:nonintegrability}), we turn to what they possess.
The non-integrability results might suggest that little remains
to say about these equations from an analytical perspective.
The opposite is the case.
Four subsections develop the evidence for systematic mathematical
organization in the absence of complete integrability: universal
algebraic formulas (\S\ref{subsec:universal-formulas}),
regularity classification (\S\ref{subsec:regularity}),
finite conservation hierarchy (\S\ref{subsec:conservation-laws}),
and Hamiltonian geometry (\S\ref{subsec:hamiltonian-geometry}).
Detailed derivations supporting these topics are collected in
Appendices~\ref{app:weak-solutions}--\ref{app:hamiltonian-detail}.
The collision theory that converts these structural elements into
dynamical predictions follows in \S\ref{subsec:quasi-elastic}.

\subsection{Three universal formulas}
\label{subsec:universal-formulas}

Despite the non-integrability established in
\S\ref{sec:nonintegrability}, the K(m,m) family harbors a
striking regularity: three algebraic relationships connect every
compacton's shape, width, and speed to the single parameter~$m$.
The $K(m,m)$ compacton was first derived
in~\cite{rosenau1993compactons}; subsequent work verified individual
cases and Rosenau and Zilburg~\cite{rosenau2018compactons}
presented the general K(m,m) formula within their comprehensive
summary of the $\mathcal{C}_N(m,a,b)$ equation family, whose
dispersion term $\frac{1}{b}[u^a(u^b)_{xx}]_x$ reduces to that of
K(m,n) at $a = 0$, $b = n$ (equation~(2.35a) of that survey; see
Appendix~\ref{app:extensions}).  However, the existing treatments derive the formulas
for specific values of~$m$ or state the general result without a
self-contained proof.  The theorem below provides a unified derivation
for all $m > 1$ simultaneously, extracts the structural mechanism
behind each formula, and establishes consequences in
Corollaries~\ref{cor:width}--\ref{cor:speed-bounds} and
Proposition~\ref{prop:trig-vanishing}
(\S\ref{subsec:quasi-elastic}).

The compacton solutions take the form
\begin{equation}\label{eq:compacton-ansatz}
U(\xi) = A\cos^{\alpha}\!\left(\frac{\xi}{\ell}\right),
\qquad |\xi| \leq \frac{\pi\ell}{2},
\end{equation}
with $U(\xi) = 0$ outside this interval. Here $A > 0$ is the amplitude,
$\alpha$ governs the rate of approach to zero at the support boundary,
and $\ell$ controls the spatial width. The K(2,2) compacton was identified
in~\cite{rosenau1993compactons}, which also noted that expansion of the
discrete chain with monomial average leads to K(m,m) compactons ``of
the form $\cos^{2/(m-1)}$.'' Rosenau~\cite{rosenau1994nonlinear}
derived the $\mathcal{C}_1(2,1,1)$ family explicitly.  The 2018
survey~\cite{rosenau2018compactons} lists the general K(n,n) traveling
compacton in equation~(2.35a), expressed in the notation of the
$\mathcal{C}_N(m,a,b)$ framework.

The three relationships connecting the compacton
parameters $(\alpha, \ell, c)$ to the nonlinearity exponent~$m$,
namely the edge exponent $\alpha = 2/(m-1)$, the width parameter $\ell = 2m/(m-1)$,
and the speed--amplitude relation $c = [(m+1)/(2m)]\,A^{m-1}$, are
implicit in the existing formula.  The following theorem extracts all
three simultaneously for arbitrary $m > 1$ and identifies the
distinct mechanism behind each.

\begin{theorem}[Universal Formulas for K(m,m) Compactons]\label{thm:universal}
For $m > 1$, the K(m,m) equation
\begin{equation}\tag{\ref{eq:kmm}}
u_t + (u^m)_x + (u^m)_{xxx} = 0
\end{equation}
admits traveling wave solutions $u(x,t) = U(x - ct)$ of the
form~\eqref{eq:compacton-ansatz}, where $A > 0$ is an arbitrary
amplitude parameter and the three characteristic parameters are
determined uniquely by~$m$:
\begin{alignat}{2}
\alpha &= \frac{2}{m-1} &\qquad& \text{(edge exponent),}
  \label{eq:alpha}\\[4pt]
\ell &= \frac{2m}{m-1} &\qquad& \text{(width parameter),}
  \label{eq:ell}\\[4pt]
c &= \frac{m+1}{2m}\,A^{m-1} &\qquad& \text{(speed--amplitude relation).}
  \label{eq:speed}
\end{alignat}
\end{theorem}

\begin{proof}
The argument extracts three independent constraints from the
traveling wave ODE: with $\theta = \xi/\ell$, matching the dominant
powers of $\cos\theta$
determines~$\alpha$, balancing the coefficients of the leading
power determines~$\ell$, and the remaining balance
determines~$c$.

Substituting $u(x,t) = U(\xi)$ with $\xi = x - ct$
into~\eqref{eq:kmm} and integrating once yields the
traveling wave ODE
\begin{equation}\tag{\ref{eq:tw-ode}}
-cU + U^m + (U^m)'' = 0.
\end{equation}
The integration constant vanishes by requiring that both $U$ and
the flux $U^m + (U^m)''$ tend to zero at the support boundary
$|\xi| \to \pi\ell/2$.  For $m \geq 4$, the derivative $U'$
diverges at the support boundary (since $\alpha < 1$), but the
product $(U^m)' = mU^{m-1}U'$ remains bounded because the
prefactor $U^{m-1}$ decays fast enough to compensate.  In
particular, $(U^m)' \sim (\pi/2 - \theta)^{m\alpha - 1} \to 0$ as
$\theta \to \pi/2$ whenever $m\alpha > 1$, which is satisfied
for all $m > 1$ since $m\alpha = 2m/(m-1) > 2$.  The ODE is
therefore analyzed on the open interval
$\theta \in (-\pi/2, \pi/2)$, where $U > 0$ and all terms are
smooth.

Set $\theta = \xi/\ell$ so that $d/d\xi = \ell^{-1}\,d/d\theta$,
and write $U = A\cos^{\alpha}\!\theta$ and
$U^m = A^m\cos^{\gamma}\!\theta$, where $\gamma = m\alpha$.
The second $\xi$-derivative of $U^m$ is
$(U^m)'' = \ell^{-2}(U^m)_{\theta\theta}$.
Computing $(A^m\cos^{\gamma}\!\theta)_{\theta\theta}$ by differentiating
twice and applying $\sin^2\!\theta = 1 - \cos^2\!\theta$ gives
\begin{equation}\label{eq:Um-second}
(U^m)'' = \frac{A^m}{\ell^2}\Bigl[
  \gamma(\gamma-1)\cos^{\gamma-2}\!\theta
  - \gamma^2\cos^{\gamma}\!\theta\Bigr].
\end{equation}

Substituting into~\eqref{eq:tw-ode}:
\begin{equation}\label{eq:full-sub}
-cA\cos^{\alpha}\!\theta + A^m\cos^{\gamma}\!\theta
+ \frac{A^m}{\ell^2}\Bigl[
  \gamma(\gamma-1)\cos^{\gamma-2}\!\theta
  - \gamma^2\cos^{\gamma}\!\theta\Bigr] = 0.
\end{equation}

This equation involves three powers of $\cos\theta$:
$\cos^{\alpha}\!\theta$, $\cos^{\gamma}\!\theta$,
and $\cos^{\gamma-2}\!\theta$. For the equation to hold
identically on $(-\pi/2,\,\pi/2)$, coefficients of linearly
independent powers must vanish separately. Since $m > 1$
implies $\gamma > \alpha$, the power $\cos^{\alpha}\!\theta$
is independent of $\cos^{\gamma}\!\theta$.
The term $-cU = -cA\cos^{\alpha}\!\theta$ can balance with the
$\cos^{\gamma-2}\!\theta$ term from the dispersion only if
$\alpha = \gamma - 2 = m\alpha - 2$. Solving:
\[
\alpha(1 - m) = -2 \qquad\Longrightarrow\qquad
\alpha = \frac{2}{m-1},
\]
establishing~\eqref{eq:alpha}. With this matching,
equation~\eqref{eq:full-sub} involves exactly two independent
powers. Setting the coefficient of $\cos^{\gamma}\!\theta$ to zero:
\[
A^m\Bigl[1 - \frac{\gamma^2}{\ell^2}\Bigr] = 0
\qquad\Longrightarrow\qquad
\ell = \gamma = \frac{2m}{m-1},
\]
establishing~\eqref{eq:ell}. Setting the coefficient of
$\cos^{\gamma-2}\!\theta$ to zero:
\[
-cA + \frac{A^m}{\ell^2}\,\gamma(\gamma-1) = 0
\qquad\Longrightarrow\qquad
c = A^{m-1}\,\frac{\gamma-1}{\gamma}
  = \frac{m+1}{2m}\,A^{m-1},
\]
establishing~\eqref{eq:speed}.
After these three conditions are imposed,
equation~\eqref{eq:full-sub} reduces to $0 = 0$ on
$(-\pi/2,\,\pi/2)$.
At the support boundary $\theta = \pi/2$, the solution
satisfies $U(\xi_0) = 0$ and the flux condition
$(U^m)'(\xi_0^-) = 0$, matching the exterior solution
$U \equiv 0$.
\end{proof}

\medskip

Each formula in Theorem~\ref{thm:universal} has a distinct structural
origin.  \emph{Power matching} determines the edge exponent:
$\alpha = 2/(m-1)$ is the unique value for which the translation
term $-cU$ and the dispersive term $(U^m)''$ contribute to the
same power of $\cos\theta$.  \emph{Coefficient cancellation}
determines the width: forcing the nonlinear term $U^m$ and the
leading part of the dispersion to cancel at order
$\cos^{\gamma}\!\theta$ requires $\ell = \gamma = 2m/(m-1)$.
Finally, \emph{amplitude--speed balance} determines the
propagation speed: the remaining terms at order
$\cos^{\gamma-2}\!\theta$ yield $c = [(m+1)/(2m)]\,A^{m-1}$.
Together, these constraints exhaust the information available
from the cosine-power ansatz.

\subsubsection*{Verification and immediate consequences}

Supplement Table~S1 confirms the theorem across
$m = 2$ through $m = 7$, including the limiting behavior as
$m \to \infty$.

Several consequences follow directly from the universal formulas.

\begin{corollary}[Width Independence]\label{cor:width}
The compacton support width $W = \pi\ell = 2\pi m/(m-1)$ depends
only on~$m$ and is independent of amplitude~$A$.
\end{corollary}

This property has no analog for KdV solitons, whose width scales
as~$A^{-1/2}$. A tall K(m,m) compacton and a short one occupy the
same spatial extent; they travel at different speeds but maintain
identical profiles up to vertical rescaling.

\begin{corollary}[Product Identity]\label{cor:product}
The edge exponent and width parameter satisfy
\begin{equation}\label{eq:product}
\alpha\cdot\ell = \frac{4m}{(m-1)^2},
\end{equation}
connecting boundary regularity to spatial extent through a single
algebraic relation.
\end{corollary}

This product decreases monotonically from $8$ at
$m = 2$ toward $0$ as $m \to \infty$. Since $\alpha$ governs boundary
regularity and $\ell$ governs width, sharper
edges necessarily accompany narrower supports within the K(m,m) family.

\begin{corollary}[Speed Coefficient Bounds]\label{cor:speed-bounds}
The speed coefficient $k(m) = (m+1)/(2m)$ decreases monotonically
from $k(2) = 3/4$ toward $k(\infty) = 1/2$ as $m \to \infty$.
\end{corollary}

Higher nonlinearity reduces the speed gained per unit amplitude
increase. In the limit $m \to \infty$, the speed--amplitude
relation approaches $c = A^{m-1}/2$, where the growing exponent
compensates for the halved coefficient.

\begin{corollary}[Monotonicity along the compacton family]\label{cor:VK}
For all $m > 1$ and $A > 0$, the speed--amplitude relation satisfies
$dc/dA > 0$:
\begin{equation}\label{eq:VK}
\frac{dc}{dA} = \frac{(m+1)(m-1)}{2m}\,A^{m-2} > 0.
\end{equation}
Equivalently, the mass $I_0 = C_m A$, with
$C_m = \ell\,B\bigl(\frac{\alpha+1}{2}, \frac12\bigr)$
from~\eqref{eq:mass-evaluated}, and the translation generator
$P = \tfrac12 I_4 = (m+1)\,I_1/(2c)$, which is proportional to
$c^{2/(m-1)}$, are strictly increasing functions of~$c$.
\end{corollary}

The expression for $P$ follows from $(D_x^2+1)U^m = cU$, which gives
$(D_x^2+1)^{-1}U = U^m/c$.  The slope condition of Vakhitov and
Kolokolov~\cite{vakhitov1973stationary} and of the
Grillakis--Shatah--Strauss framework~\cite{grillakis1987stability}
concerns $dP/dc$, the derivative of the conserved quantity that
generates the symmetry, here translation; the mass $I_0$ is a
Casimir and does not enter it.  Even with $P$ the framework yields
no stability conclusion for K(m,m).  Because $\mathcal{H}$ has no
quadratic part, the second variation of $\mathcal{H} - cP$ at a
compacton reduces to $-c\,(D_x^2+1)^{-1}$ on perturbations supported
away from it, and is negative on slowly varying ones, so it has
infinitely many negative directions where the theory requires one.
Stability of compactons therefore remains open
(Problem~3 of Appendix~\ref{app:open-problems}).

\subsubsection*{Structural contrast with integrable systems}

In the absence of spectral theory, the universal formulas establish
an organizing principle different from, but comparable in kind to,
the inverse scattering transform for KdV.
For integrable equations, the IST yields exact multi-soliton formulas
from spectral data: soliton parameters are encoded in discrete
eigenvalues and norming constants of the Lax
operator~\cite{lax1968integrals,ablowitz1981solitons}.
For K(m,m), no spectral theory exists: there are no Lax pairs,
no recursion operators, no scattering data.
What remains is \emph{scaling universality}: three closed-form
relationships predicting all compacton properties from the single
parameter~$m$. The methods are different, the machinery is different,
but the output is the same kind of mathematical knowledge: algebraic
formulas that govern a family of nonlinear wave solutions.

A notable structural consequence is the
identity $\ell - \alpha = 2$, which holds for all $m > 1$.
This identity implies that the trigonometric conservation
integrals $\int U\sin(x)\,dx$ and $\int U\cos(x)\,dx$ vanish
identically for every K(m,m) compacton, a property noted
empirically for $K(2,2)$ in~\cite{rosenau1993compactons} and proved in
full generality in Proposition~\ref{prop:trig-vanishing}
and Appendix~\ref{app:jacobian}.  This vanishing has significant
consequences for the collision constraint analysis in
\S\ref{subsec:quasi-elastic}.

The proof makes no use of $m$ being an integer. For any real
$m > 1$, the formulas~\eqref{eq:alpha}--\eqref{eq:speed} define a
compacton profile that formally satisfies the traveling wave ODE.
We emphasize that Theorem~\ref{thm:universal} establishes the
\emph{existence} of cosine-power compactons; the question of whether other compact-support traveling
waves exist with different profiles remains open
(Problem~1 in Appendix~\ref{app:open-problems}).
The continuous dependence on~$m$ opens the possibility of studying
K(m,m) equations along continuous paths in parameter space, with
regularity transitions occurring at $m = 3$ (where $\alpha$ crosses
unity) serving as a codimension-one bifurcation. The regularity
classification that follows from Theorem~\ref{thm:universal} is
developed in the next subsection.

\subsection{Regularity classification}
\label{subsec:regularity}

The edge exponent $\alpha = 2/(m-1)$ governs more than the rate of
approach to zero.  It determines the entire regularity landscape of
K(m,m) compactons, classifying each member of the family into one of
three qualitatively distinct regimes.  This classification follows
directly from Theorem~\ref{thm:universal}: once $\alpha$ is known as a function of~$m$, the
regularity character at the support boundary is determined by
comparing~$\alpha$ with unity.

\subsubsection*{Three regularity regimes}

Near the support boundary $\xi_0 = \pi\ell/2$, the compacton
profile satisfies $U(\xi) \sim B(\xi_0 - \xi)^{\alpha}$ with
derivative $U'(\xi) \sim -\alpha B(\xi_0 - \xi)^{\alpha - 1}$,
so the classification reduces to the sign of~$\alpha - 1$.
For $m < 3$ ($\alpha > 1$), the derivative vanishes at the boundary
and the compacton is at least $C^1$; the traveling wave
ODE~\eqref{eq:tw-ode} holds pointwise.  At $m = 3$ ($\alpha = 1$),
the derivative reaches a nonzero limit, producing a cusp where
$U'$ jumps discontinuously; K(3,3) with $U(\xi) = A\cos(\xi/3)$
is the canonical example.  For $m > 3$ ($\alpha < 1$), the derivative
diverges:
\begin{equation}\label{eq:deriv-divergence}
U'(\xi) \sim -\alpha B (\xi_0 - \xi)^{\alpha - 1}
\to -\infty \quad \text{as } \xi \to \xi_0^-.
\end{equation}
Solutions in this regime belong to $C^{0,\alpha}$
(H\"older continuous with exponent~$\alpha$) but not to~$C^1$,
and a weak formulation becomes
necessary~\cite{rosenau2018compactons}.
Table~\ref{tab:regularity} catalogs the classification across the
family.

\subsubsection*{The cusp transition as a codimension-one event}

As~$m$ increases through~$3$, the edge exponent $\alpha = 2/(m-1)$
crosses unity and the qualitative character of boundary regularity
changes: the derivative transitions from vanishing ($m < 3$) through
a finite jump ($m = 3$) to divergence ($m > 3$).  This is a
codimension-one event in parameter space.  The transition has
analytical significance: in variational stability
theory~\cite{grillakis1987stability,grillakis1990stability},
regularity at solitary wave boundaries determines the function
spaces in which an analysis can be posed, and the cusp at $m = 3$
is exactly the point where this determination changes for K(m,m).

\begin{table}[ht]
\centering
\renewcommand{\arraystretch}{1.3}
\begin{tabular}{@{}cccll@{}}
\toprule
$m$ & $\alpha$ & Boundary behavior of $U'$
  & Regularity & Regime \\
\midrule
2   & 2    & $U'(\xi_0^-) = 0$, vanishes quadratically
  & $C^1$       & Smooth \\
3   & 1    & $U'(\xi_0^-) \neq 0$, jump discontinuity
  & $C^0$       & Cusp \\
4   & 2/3  & $|U'| \to \infty$ as $(\xi_0\!-\!\xi)^{-1/3}$
  & $C^{0,2/3}$ & Fractional \\
5   & 1/2  & $|U'| \to \infty$ as $(\xi_0\!-\!\xi)^{-1/2}$
  & $C^{0,1/2}$ & Fractional \\
6   & 2/5  & $|U'| \to \infty$ as $(\xi_0\!-\!\xi)^{-3/5}$
  & $C^{0,2/5}$ & Fractional \\
7   & 1/3  & $|U'| \to \infty$ as $(\xi_0\!-\!\xi)^{-2/3}$
  & $C^{0,1/3}$ & Fractional \\
\midrule
$\infty$ & $0^+$ & $|U'| \to \infty$, sharpening toward 
  & $C^{0,0^+}$ & square-wave limit \\
\bottomrule
\end{tabular}
\caption{Regularity classification of K(m,m) compactons at the
support boundary.  The edge exponent $\alpha = 2/(m-1)$ from
Theorem~\ref{thm:universal} determines each entry.}
\label{tab:regularity}
\end{table}

\subsubsection*{Weak solutions and distributional formulation}

For $m > 3$, the classical strong form of the traveling wave ODE
breaks down at the support boundary where $U'$ diverges.  A weak
(distributional) formulation replaces pointwise validity with
integral conditions: a bounded measurable function $u$ is a weak
solution if
\begin{equation}\label{eq:weak-formulation}
\int_0^T \!\!\int_{\mathbb{R}}
  \bigl[u\,\varphi_t + u^m\,\varphi_x + u^m\,\varphi_{xxx}\bigr]
  \,dx\,dt = 0
\end{equation}
for all test functions $\varphi \in C_c^{\infty}(\mathbb{R}
\times (0,T))$.
All five conservation laws extend to this weak setting, a fact
essential for the constraint analysis in~\S\ref{subsec:quasi-elastic};
the detailed argument appears in
Appendix~\ref{app:weak-solutions}.

\subsubsection*{Limiting behavior and connections}

As $m \to \infty$, the edge exponent $\alpha \to 0^+$ and the
support width $W = \pi\ell \to 2\pi$.  The compacton profile
approaches a limiting shape square-wave pulse.   
The K(m,m) family provides
a one-parameter path from smooth K(2,2) compactons through the
cusp at K(3,3) toward square-wave-like structures, with regularity
degrading monotonically along this path.  Whether the $m \to \infty$
limit yields a well-defined singular solution, which remains an open
question (Appendix~\ref{app:open-problems}).

\subsection{Finite conservation laws}
\label{subsec:conservation-laws}

Integrable equations carry infinite hierarchies of conservation laws,
generated through recursion operators from bi-Hamiltonian
structure~\cite{magri1978simple,fokas1981symplectic}.  These
hierarchies encode the algebraic backbone of
integrability: each new conservation law constrains the dynamics
further, and taken together they determine the solution completely.
What happens when that hierarchy terminates?

For K(m,m) equations, Vodov\'a~\cite{vodova2013complete} answered
this question through an exhaustive computational classification.

\begin{theorem}[Finite Conservation Hierarchy
  {\cite{vodova2013complete}}]
\label{thm:finite-hierarchy}
For $m \neq -2, -1/2, 0, 1$, the K(m,m) equation~\eqref{eq:kmm}
admits exactly four functionally independent local conservation laws.
In addition, one nonlocal conservation law
exists~\cite{rosenau2018compactons}, bringing the
known count of independent conserved quantities to five.  Whether
additional nonlocal laws exist remains open
(see the discussion following this theorem).
\end{theorem}

Vodov\'a's proof~\cite{vodova2013complete} is structural rather than
computational.  The K(m,m) equation is Hamiltonian with respect to
the third-order Poisson operator
$\mathfrak{D} = -(D_x^3 + D_x)$
(see~\S\ref{subsec:hamiltonian-geometry}).  Any conserved density
$\rho$ yields a cosymmetry $\gamma = \delta\rho/\delta u$, and
$\mathfrak{D}$ maps cosymmetries to symmetries.  Vodov\'a first
establishes (Proposition~2 in~\cite{vodova2013complete}) that no
generalized symmetries of order greater than~$3$ exist.  If $\gamma$
had differential order $K > 0$, then $\mathfrak{D}(\gamma)$ would be
a symmetry of order $3 + K > 3$, contradicting the symmetry bound.
Therefore every cosymmetry depends at most on $x$, $t$, and~$u$, and
consequently every conserved density $\rho$ depends at most on
$x$, $t$, and~$u$.  Enumerating such densities is then a finite
problem, yielding exactly four solutions.

The fifth conserved quantity is the nonlocal functional $I_4$ defined
in~\eqref{eq:momentum} below, whose density involves the integral
operator $(D_x^2 + 1)^{-1}$ and therefore falls outside the scope of
Vodov\'a's local classification.  Independence from $I_0$--$I_3$ follows
because $I_4$ is a nonlocal bilinear functional of~$u$, whereas each
local law has a density depending on~$u$ alone (pointwise in~$x$).  Its conservation follows from the
factorization $\mathfrak{D} = -D_x(D_x^2 + 1)$; the proof is given
below.  Vodov\'a's argument establishes maximality
for local conservation laws through a rigorous symmetry-order bound,
but the question of whether additional nonlocal conserved quantities
exist beyond~$I_4$ remains open; a formal impossibility proof for
higher-order nonlocal densities has not been obtained.

\begin{remark}[Distributional framework]
\label{rem:distributional}
Anco and Gandarias~\cite{anco2024weak} subsequently developed a
rigorous distributional framework for these conservation laws, showing
that weak compacton solutions satisfy the associated continuity
equations in a well-defined distributional sense even at the support
boundary where classical derivatives may not exist.  This framework
places the formal conservation law identities on firm analytical
footing for the class of compacton solutions with finite regularity
at the support edges.
\end{remark}

\subsubsection*{The five conservation laws}

Four of the five laws are local; the fifth is nonlocal, coupling
distant spatial points through an integral operator.  Full
derivations appear in Appendix~\ref{app:conservation-detail}.

\paragraph{Mass.}
\begin{equation}\label{eq:mass}
I_0 = \int_{-\infty}^{\infty} u \, dx.
\end{equation}
Mass conservation follows directly from the divergence structure of
K(m,m); in the Hamiltonian formulation of
\S\ref{subsec:hamiltonian-geometry} it is a Casimir of $\mathfrak{D}$
(a functional whose variational derivative lies in the kernel of
$\mathfrak{D}$, so that it is conserved by every flow of the form
$u_t = \mathfrak{D}(f)$),
while spatial translation invariance corresponds through Noether's
theorem to the nonlocal momentum
$I_4$~\cite{noether1918invariante,olver1993applications,rosenau2018compactons}.

\paragraph{Energy.}
\begin{equation}\label{eq:energy}
I_1 = \int_{-\infty}^{\infty} \frac{u^{m+1}}{m+1} \, dx.
\end{equation}
Energy is the Hamiltonian for the K(m,m) evolution
(\S\ref{subsec:hamiltonian-geometry}), associated with
time-translation invariance through Noether's theorem.

\paragraph{Trigonometric laws.}
\begin{align}
I_2 &= \int_{-\infty}^{\infty} u \sin(x) \, dx,
  \label{eq:trig-sin} \\[3pt]
I_3 &= \int_{-\infty}^{\infty} u \cos(x) \, dx.
  \label{eq:trig-cos}
\end{align}
These project the solution onto its fundamental Fourier mode at
wavenumber $k = 1$: the complex combination $I_3 - iI_2 =
\hat{u}(1,t)$, where $\hat{u}(k,t) = \int u\,e^{-ikx}\,dx$.  Conservation follows from exact cancellation
between the convection and dispersion terms at this specific
wavenumber (Appendix~\ref{app:conservation-detail}).
No higher Fourier modes yield conservation laws.

\paragraph{Nonlocal momentum.}
\begin{equation}\label{eq:momentum}
I_4 = \int_{-\infty}^{\infty} u(x) \,
  (D_x^2 + 1)^{-1}[u](x) \, dx,
\end{equation}
where $(D_x^2 + 1)^{-1}$ denotes the pseudoinverse of
$D_x^2 + 1$: on the Fourier side, it acts as the multiplier
$1/(1-k^2)$, which is well-defined on functions satisfying
$\hat{u}(\pm 1) = 0$, precisely the condition guaranteed by
conservation of $I_2$ and $I_3$ for compacton
solutions.  Rosenau and Zilburg~\cite[Lemma~1]{rosenau2018compactons}
obtained $I_4$ from a Lagrangian formulation of K(n,n), as the
momentum associated with translation invariance.  It falls outside
Vodov\'a's local classification and instead arises from the factorization of the Hamiltonian
operator $\mathfrak{D} = -D_x(D_x^2 + 1)$.  Conservation
of~$I_4$ requires $m = n$ and thus reflects the specific
algebraic structure of K(m,m); the proof is given in
Appendix~\ref{app:conservation-detail}.

\subsubsection*{Structural contrast with integrable hierarchies}

A direct comparison highlights what the finite hierarchy means
(Table~\ref{tab:conservation-comparison}).

\begin{table}[ht]
\centering
\small
\setlength{\tabcolsep}{4pt}
\renewcommand{\arraystretch}{1.3}
\begin{tabular}{@{}lcc@{}}
\toprule
Feature & KdV & K(m,m) \\
\midrule
Conservation law count & $\infty$ & $5$ (known) \\
Generation mechanism & Lenard recursion via
  $\mathcal{R} = \mathcal{J}_2\mathcal{J}_1^{-1}$ & None \\
Highest-order density & Unbounded & $I_4$ (nonlocal) \\
Algebraic closure & $\{I_j, I_k\} = 0$ for all $j, k$
  & $\{I_j, I_k\} = 0$ (three Casimirs) \\
Collision implication & Perfect elasticity
  & Quasi-elastic with debris \\
\bottomrule
\end{tabular}
\caption{Conservation law structure: KdV versus K(m,m).  Here
$I_j$ and $I_k$ denote any two conserved functionals of the equation
in question (for K(m,m), $j, k \in \{0, \ldots, 4\}$), and
$\{\cdot, \cdot\}$ is its Poisson bracket.}
\label{tab:conservation-comparison}
\end{table}

In KdV, the recursion operator $\mathcal{R}
= \mathcal{J}_2\mathcal{J}_1^{-1}$ generates conservation laws
indefinitely.  Given the $n$-th conserved density, applying
$\mathcal{R}$ produces the $(n+1)$-st.  The resulting infinite
family is in involution: all Poisson brackets $\{I_j, I_k\}$
vanish, providing the complete integrability in the sense of
Liouville~\cite{arnold1989mathematical}.  The inverse scattering
transform~\cite{gardner1967method,lax1968integrals} translates
this algebraic structure into explicit multi-soliton solutions.

For K(m,m), no recursion operator exists
(Theorem~\ref{thm:her-classification}), and the hierarchy
terminates.  Theorem~\ref{thm:finite-hierarchy} quantifies the
termination.  The five known laws are
in involution under the Poisson bracket~\eqref{eq:poisson-bracket}
of \S\ref{subsec:hamiltonian-geometry}: three are Casimirs, and,
with $\mathcal{H} = I_1$ the Hamiltonian~\eqref{eq:hamiltonian},
$\{\mathcal{H}, I_4\} = -2\int u^m u_x\,dx = 0$.  The two
non-Casimir functionals generate the K(m,m) time flow,
$X_{\mathcal{H}} = \mathfrak{D}\,\delta\mathcal{H}/\delta u$, and
spatial translation, $X_P = -u_x$ for $P = \tfrac12 I_4$; these
commute, and a traveling wave is the relative equilibrium on which
$X_{\mathcal{H}} = c\,X_P$.  Beyond this pair the known laws furnish no further
independent commuting flows.  That shortage, not any failure of the
laws to commute, marks non-integrability.

\subsubsection*{Origin and structure of the five laws}

The Hamiltonian operator $\mathfrak{D} = -(D_x^3 + D_x)$
accounts for all five known conservation laws through its algebraic
structure.  Its formal kernel $\ker(\mathfrak{D})$,
spanned by $\{1, \cos x, \sin x\}$, generates three Casimir
functionals: mass $I_0$ and the trigonometric pair $I_2$, $I_3$.
(Although $\{1, \cos x, \sin x\} \notin L^2(\mathbb{R})$, their
action as Casimir densities is rigorously defined through the
$L^2$ duality pairing with compacton solutions in
$L^1_c(\mathbb{R})$, the integrable functions of compact support.)
These are conserved for \emph{any} evolution of the form
$u_t = \mathfrak{D}(f)$, regardless of the choice of~$f$.
The fourth local law $I_1$ (energy) is the Hamiltonian itself,
and the nonlocal $I_4$ arises from the factorization
$\mathfrak{D} = -D_x(D_x^2 + 1)$.  The detailed Casimir analysis
and evaluated conservation integrals for individual compactons
appear in Appendix~\ref{app:conservation-detail}.

These five quantities provide the known conservation constraints
governing K(m,m) dynamics.  Their role in controlling collision
outcomes is the subject of~\S\ref{subsec:quasi-elastic}.

\subsection{Hamiltonian geometry without bi-Hamiltonian structure}
\label{subsec:hamiltonian-geometry}

The K(m,m) equation arises naturally as a Hamiltonian system.
Rosenau and Zilburg~\cite[\S 7.1]{rosenau2018compactons} gave Hamiltonian
and Lagrangian formulations of the K(n,n) and K(n+1,n) equations; the
formulation below follows theirs for K(n,n).
Define the functional
\begin{equation}\label{eq:hamiltonian}
\mathcal{H}[u] = \int_{-\infty}^{\infty}
  \frac{u^{m+1}}{m+1} \, dx,
\end{equation}
and the third-order Poisson operator
\begin{equation}\label{eq:poisson-op}
\mathfrak{D} = -(D_x^3 + D_x).
\end{equation}
The K(m,m) equation~\eqref{eq:kmm} is then equivalent to
\begin{equation}\label{eq:ham-evolution}
u_t = \mathfrak{D} \, \frac{\delta \mathcal{H}}{\delta u},
\end{equation}
since $\delta\mathcal{H}/\delta u = u^m$ and
$\mathfrak{D}(u^m) = -(u^m)_{xxx} - (u^m)_x$.
The operator $\mathfrak{D}$ is skew-adjoint and
satisfies the Jacobi identity, making it a valid Poisson
operator~\cite{olver1993applications,vodova2013complete}.
The Hamiltonian $\mathcal{H}$ is itself conserved: the
antisymmetry of $\mathfrak{D}$ ensures
$d\mathcal{H}/dt
= \langle \delta\mathcal{H}/\delta u,\,
\mathfrak{D}\,\delta\mathcal{H}/\delta u \rangle = 0$.
This identifies $\mathcal{H}$ with the energy conservation
law $I_1$ of \S\ref{subsec:conservation-laws}.

This formulation induces a Poisson bracket on the space of
functionals:
\begin{equation}\label{eq:poisson-bracket}
\{F, G\} = \int \frac{\delta F}{\delta u}\,
  \mathfrak{D} \!\left[\frac{\delta G}{\delta u}\right] dx,
\end{equation}
satisfying antisymmetry, bilinearity, and the Jacobi identity.  Any
functional with $\{F, \mathcal{H}\} = 0$ is conserved.  The five
conservation laws from Theorem~\ref{thm:finite-hierarchy} all
satisfy this condition, and the bracket provides a unified algebraic
framework relating them.

A notable feature of the K(m,m) Hamiltonian structure is
the factorization $\mathfrak{D} = -D_x(D_x^2 + 1)$, whose
three-dimensional kernel generates the three Casimirs; the origin of
each conservation law is traced in \S\ref{subsec:conservation-laws}.

\subsubsection*{The bi-Hamiltonian obstruction}

KdV possesses two compatible Hamiltonian structures whose
recursion operator $\mathcal{R}
= \mathcal{J}_2\,\mathcal{J}_1^{-1}$ generates the infinite
conservation hierarchy through the Lenard
scheme~\cite{magri1978simple,fokas1981symplectic}.
Compact support does not by itself rule out such a pair: Olver and
Rosenau~\cite{olver1996tri} built integrable bi-Hamiltonian equations
with compacton solutions from compatible Hamiltonian triples.
K(m,m) admits no such extension.  Systematic searches for a
second Poisson operator compatible with $\mathfrak{D}$ fail for
all $m \geq 2$, a consequence of Vodov\'a's proof that no
generalized symmetries of order greater than~$3$
exist~\cite{vodova2013complete}: a compatible second operator
would induce such symmetries through recursion.  The nonlinear
dispersion $(u^m)_{xxx}$ is the root cause.

Without the compatible pair, three consequences follow: no
recursion operator operates (explaining the hierarchy termination),
the conserved functionals, though in involution, generate only two
commuting flows, time evolution and translation (three are
Casimirs), and the reduction to action-angle variables
is unavailable~\cite{arnold1989mathematical}.
The detailed obstruction analysis appears in
Appendix~\ref{app:hamiltonian-detail}.

\subsubsection*{What persists from the integrable world}

Despite the failure of bi-Hamiltonian structure, K(m,m) retains
substantial geometric machinery from the Hamiltonian framework.

\emph{Variational principles.}  The traveling wave solutions of
Theorem~\ref{thm:universal} arise from the Hamiltonian structure.
Substituting the traveling-wave ansatz $u(x,t) = U(x - ct)$ into
$u_t = \mathfrak{D}(\delta\mathcal{H}/\delta u)$ and integrating
once yields $(D_x^2 + 1)(U^m) = cU$, the traveling wave ODE
whose phase-plane analysis is developed in
\S\ref{sec:shelf-waves}.

\emph{Stability theory.}  The standard route to orbital stability
uses $\mathcal{H} - cP$ as a Lyapunov functional and examines its
second variation at the traveling
wave~\cite{benjamin1972stability,grillakis1987stability}.  For
K(m,m) this route is blocked for every $m > 1$: the second variation
has infinitely many negative directions (see the remark after
Corollary~\ref{cor:VK}).

\emph{Conservation law origins.}  Each of the five laws has a
geometric origin in $\mathfrak{D}$ (\S\ref{subsec:conservation-laws}),
and the local hierarchy is finite, reflecting the finite-dimensional
symmetry algebra established by
Vodov\'a~\cite{vodova2013complete}.

The situation can be summarized concisely.  K(m,m) is Hamiltonian
but not bi-Hamiltonian.  It possesses a Poisson structure
(sufficient for constructing structure-preserving numerical
integrators), variational principles, and a Noether
correspondence, but lacks the recursion mechanism that converts
these ingredients into complete integrability.  The five
conservation laws constrain the dynamics strongly but not
completely.  The gap between strong constraint and complete
determination manifests as quasi-elastic collision behavior:
emerging compactons retain most of their pre-collision identities,
while small but nonzero debris fields absorb the remaining
degrees of freedom (\S\ref{subsec:quasi-elastic}).

\paragraph{Variational characterization of traveling waves.}
Formally, for perturbations in the domain of $(D_x^2+1)^{-1}$
(\S\ref{subsec:conservation-laws}), set $P := \tfrac12 I_4$, so that
$\delta P/\delta u = (D_x^2+1)^{-1}u$
and $\mathfrak{D}\,\delta P/\delta u = -u_x$: $P$ generates
translations.  For a traveling wave $u = U(x - ct)$ one has
$u_t = -cU' = \mathfrak{D}\,\delta\mathcal{H}/\delta u
= c\,\mathfrak{D}\,\delta P/\delta u$, hence
\[
\frac{\delta\mathcal{H}}{\delta u} = c\,\frac{\delta P}{\delta u} + \kappa,
\qquad \kappa \in \ker\mathfrak{D} = \operatorname{span}\{1, \cos x, \sin x\}.
\]
Applying $(D_x^2+1)$ recovers the once-integrated traveling-wave
ODE $-cU + U^m + (U^m)'' = C_0$, with the integration constant $C_0$
equal to the constant Casimir component of~$\kappa$;
\eqref{eq:tw-ode} is the case $C_0 = 0$, and~\eqref{eq:shelf-tw-ode}
is its form on a positive background $u = a$
(\S\ref{sec:shelf-waves}), for which $\mathcal{H}$ and $P$ are
renormalized by subtracting their values at $u \equiv a$ and varied
with respect to $f = U - a$.

\section{Collision dynamics and the constraint-deficit mechanism}
\label{sec:collision-dynamics}

The universal formulas, conservation hierarchy, and Hamiltonian
geometry of the preceding section describe the static structure of
K(m,m) compactons.  This section turns to dynamics: what happens
when compactons collide?  The constraint-deficit mechanism
(Theorem~\ref{thm:constraint-deficit}) provides a precise answer,
and comparison with solitary waves of the Benjamin--Bona--Mahony
(BBM) equation~\cite{benjamin1972model} (\S\ref{subsec:landscape})
reveals that interaction geometry plays an independent role alongside conservation law count.
A concluding subsection positions the constraint-deficit framework
as a potential tool for analyzing near-integrability in other
dispersive PDEs.

\subsection{Quasi-elastic collisions and the constraint-deficit mechanism}
\label{subsec:quasi-elastic}

\begin{figure}[htbp]
\centering
\includegraphics[width=\textwidth]{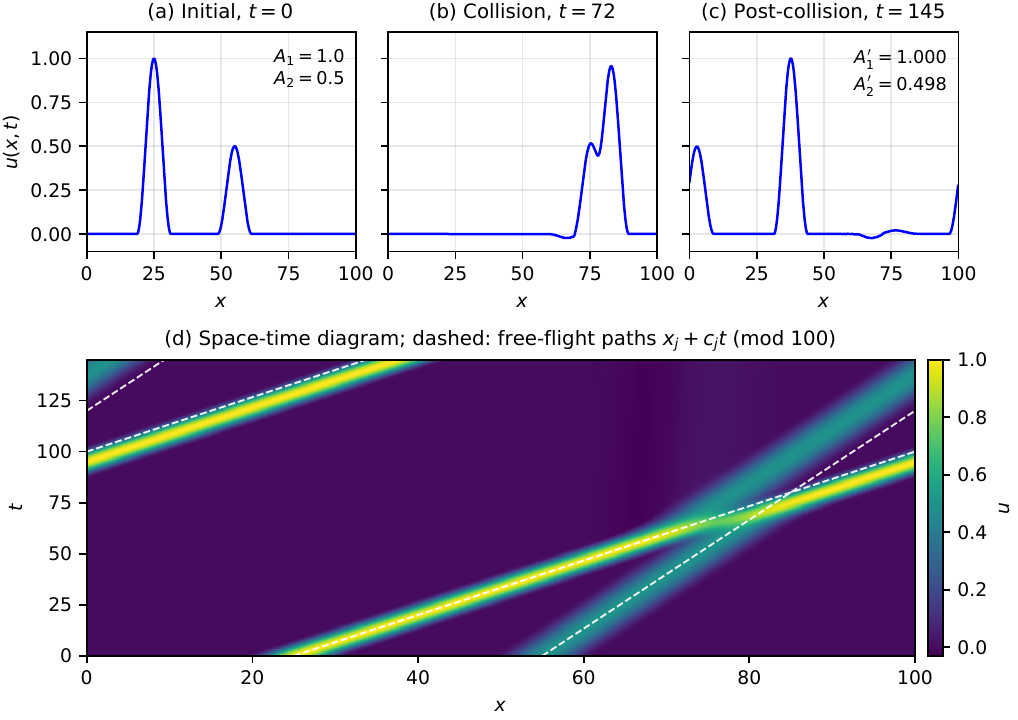}
\caption{Quasi-elastic collision of two K(2,2) compactons with
amplitude ratio $A_2/A_1 = 0.5$ on the periodic domain $[0, 100)$.
(a)~Initial configuration at $t=0$ with $A_1 = 1.0$, $A_2 = 0.5$.
(b)~During collision at $t = 72$.
(c)~Post-collision at $t = 145$.  Both compactons have left through
the right boundary and re-entered at the left, and the smaller one
straddles the boundary.  Primes denote post-collision amplitudes,
$A_1' = 1.000$ and $A_2' = 0.498$; the changes
$|\Delta A_j|/A_j < 0.5\%$ demonstrate quasi-elastic behavior.
(d)~Space-time diagram of $u(x,t)$.  Dashed lines are the
free-flight paths $x_j + c_j t$ (mod 100) with
$c_j = \tfrac34 A_j$.  After the collision the faster compacton runs
ahead of its path by $\delta_1 \approx 3.9$ and the slower one falls
behind by $|\delta_2| \approx 6.6$.
Numerical parameters: pseudospectral method (PSM) with $N_x = 512$ grid points,
joint low-pass (LP) filtering of the flux (Variant~C, cutoff
$k_{\max}/q$ with $q = 2$), biharmonic ramp dissipation with
coefficient $\delta = 0.15$ (distinct from the phase shifts
$\delta_j$), fixed-step fourth-order Runge--Kutta (RK4) time
stepping; see Appendix~\ref{app:numerical-method}.}
\label{fig:collision-k22}
\end{figure}

Two K(m,m) compactons with amplitudes $A_1 > A_2 > 0$ and speeds
$c_1 > c_2$ will eventually collide when the faster wave overtakes
the slower (Figure~\ref{fig:collision-k22}). The post-collision state consists of two principal
compactons with amplitudes $A'_1, A'_2$ (primes mark post-collision
values), asymptotic phase shifts $\delta_1, \delta_2$
(Figure~\ref{fig:collision-k22}(d)), and a residual debris field~$w(x,t)$
of small amplitude spreading dispersively behind the emergent
waves~\cite{rosenau1993compactons,rosenau2018compactons}.
The numerically resolved debris $w = u - u_{\mathrm{compactons}}$
is continuous on the periodic computational domain, hence bounded
and square-integrable, ensuring that its Fourier analysis is well-defined.

Numerical simulations reveal a consistent pattern across the K(m,m)
family (Figure~\ref{fig:collision-k22}): the principal compactons
emerge with amplitudes close to their pre-collision values, while
the debris field carries a small fraction of the total
energy~\cite{rosenau1993compactons,ismail1998numerical}. To quantify
this, define the elasticity coefficient
\begin{equation}\label{eq:elasticity}
\mathcal{E} = \frac{A'_1 + A'_2}{A_1 + A_2},
\end{equation}
measuring the fraction of total amplitude retained by the principal
waves. For K(2,2) with $A_1 = 1.0$, $A_2 = 0.8$, our simulations
(pseudospectral method (PSM), $N_x = 512$ grid points, domain
$[0, 200]$, dissipation coefficient $\delta = 0.15$)
with the principal waves recovered by least-squares fit of the exact
profile, give $A'_1 \approx 1.0008$, $A'_2 \approx 0.7987$, yielding
$\mathcal{E} \approx 0.9997$; these values are unchanged at
$N_x = 1024$.  The debris field $w$ carries approximately $0.03\%$
of the squared $L^2$ norm $\int u^2\,dx$; the principal waves retain all but
$6\times10^{-5}$ of their Hamiltonian energy according to the
fitted amplitudes; at the $10^{-4}$ precision of the fit this is an
upper bound on the true deficit.

\paragraph{Alternative elasticity measures.}
Several natural alternatives to~\eqref{eq:elasticity} exist,
including energy-based and $L^2$-based measures.  For K(m,m)
compactons, all standard measures are related by monotone
transformations of the amplitude ratio, so the within-family
ordering is invariant under the choice of measure; see
Appendix~\ref{app:collision-detail} for the detailed comparison.

These values are consistent with
earlier simulations~\cite{rosenau1993compactons,ismail1998numerical}.
For K(3,3) with $A_1 = 1.5$, $A_2 = 1.0$, the same method gives
$A'_1 \approx 1.5001$, $A'_2 \approx 0.9990$ and
$\mathcal{E} \approx 0.9996$, with debris again near $0.03\%$ of the
squared $L^2$ norm and a fitted-amplitude Hamiltonian deficit of at most
$5\times10^{-4}$, again limited by the fit precision.  Recent high-resolution simulations by
Garral\'on-L\'opez, Rus, and
Villatoro~\cite{garralon2023compacton} have extended these
observations to compacton-anticompacton collisions, where
the interaction geometry is more violent.  Their results
reveal a striking odd-even asymmetry in the nonlinearity
exponent, which they denote~$p$ (our~$m$): for odd $p$ (= 3, 5, 7), collisions remain
robust even with minimal hyperviscosity regularization,
while even $p$ (= 2, 4, 6) produces blow-up unless the
amplitude ratio lies below a threshold.  The same group
established that the absolute-value variant $|K|(p,p)$ admits
compactons for non-integer $p$, extending the parameter space
beyond the integer cases treated
here~\cite{garralon2023robustness}, and developed an adiabatic
perturbation theory quantifying how hyperviscosity affects
compacton evolution~\cite{garralon2024numerical}.  These
numerical findings complement and extend the analytical
framework below.

\paragraph{Numerical method.}
The compact support of K$(m,n)$ compactons poses challenges for standard 
spectral methods, which suffer from Gibbs oscillations at the support 
boundary.  We employ a spectrally-filtered pseudospectral method that
combines joint low-pass filtering of both nonlinear flux terms
(Variant~C, with cutoff wavenumber $k_{\max}/q$ and $q = 2$) with a biharmonic high-pass dissipation operator
concentrated near the compacton boundary by a smooth spatial ramp.
The dissipation coefficient $\delta = 0.15$, scaled by the
instantaneous maximum wave speed $c_{\max} = \tfrac{m+1}{2m}\max|u|^{m-1}$,
is universal across all values of~$m$ tested.
Full details, including the boundary-layer analysis motivating
the $c_{\max}$ scaling, are given in 
Appendix~\ref{app:numerical-method} and in~\cite{hyman2026filters}.
Code implementing these simulations
is available at \texttt{https://github.com/machyman/kmm-compactons}.

These collisions
are neither perfectly elastic nor strongly inelastic. The question
is why.

\paragraph{A shift of perspective.}
Peter Lax spent much of his career analyzing approaches to
eliminate numerical oscillations behind steep fronts, a
recurring challenge in the numerical approximation of shock
waves.  Late in his career, he shifted focus to
\emph{understanding} these oscillations and, with Dave
Levermore, developed the theory of the zero-dispersion limit of
the KdV equation~\cite{lax1983small}, revealing that the
oscillations encode deep information about the underlying
dispersive structure.  In the same spirit, after spending much
of my career focused on elastic soliton collisions, I now turn
to analyzing the small but persistent inelasticity of compacton
collisions.  The constraint-deficit mechanism
(Theorem~\ref{thm:constraint-deficit}) provides a concrete,
quantitative explanation for why these collisions are
near-elastic but not perfectly so, and what the residual debris
reveals about the gap between finite and infinite conservation
hierarchies.

\subsubsection*{Why collisions cannot be perfectly elastic}

Three hypotheses underlie the analysis below.  All are supported by
available numerical evidence but have not yet been proved rigorously.
We state them explicitly so that the logical dependence of the
subsequent results is transparent.

\begin{enumerate}
\item[\textbf{(H1)}] \emph{Two-compacton emergence.}
After a binary collision, the post-collision state consists of
exactly two principal compactons plus a dispersive debris field of
small amplitude.  This behavior has been observed in extensive numerical
studies spanning three decades, including
Rosenau--Hyman~\cite{rosenau1993compactons},
Ismail--Taha~\cite{ismail1998numerical},
Cooper--Mihaila--Saxena~\cite{cooper2001compacton,mihaila2010stability},
and Garral\'on--Rus--Villatoro~\cite{garralon2023compacton,garralon2024numerical}.
No simulation has produced a third compacton or a
non-compacton localized structure; the recent high-resolution
studies~\cite{garralon2023compacton,garralon2024numerical} tracked
collision dynamics over thousands of time units with no
deviation from this pattern.

\item[\textbf{(H2)}] \emph{Profile preservation.}
The emergent compactons take the cosine-power form established in
Theorem~\ref{thm:universal}, with possibly modified amplitudes and
phase shifts.  This hypothesis follows from
Proposition~\ref{prop:uniqueness}: within the class of positive,
symmetric, single-hump traveling waves with compact support, the
cosine-power profile is the unique solution of the traveling wave
ODE\@.  The only additional assumption is that emergent localized
structures are positive and symmetric, a property observed in all
reported numerical simulations.

\item[\textbf{(H3)}] \emph{Conservation law completeness.}
The four local conservation laws classified by
Vodov\'a~\cite{vodova2013complete} together with the nonlocal
functional $I_4$ exhaust the independent conserved quantities.
Vodov\'a's structural proof establishes local completeness
rigorously for all real $m \neq -2, -1/2, 0, 1$; the assumption
here is that no additional nonlocal conserved quantities exist
beyond~$I_4$.
\end{enumerate}

\noindent
Under these hypotheses, the answer to why collisions are not
perfectly elastic lies in the constraint structure of the five
conservation laws.  The following theorem, conditional on hypotheses
(H1)--(H3) above, shows that perfect
elasticity imposes incompatible demands.

\begin{theorem}[Constraint deficit (assuming H1--H3)]\label{thm:constraint-deficit}
For a two-compacton collision in K(m,m), $m \geq 2$:
\begin{enumerate}
\item[\emph{(a)}] The post-collision debris field $w(x,t)$ has
  infinitely many degrees of freedom, parameterized by its Fourier
  spectrum $\{\hat{w}(k)\}_{k \in \mathbb{R}}$.
\item[\emph{(b)}] The five conservation laws impose five scalar
  constraints on the post-collision state. Two of these (mass $I_0$
  and energy $I_1$) constrain the emergent compacton amplitudes:
  given the debris mass $M_w$ and energy $E_w$, they locally
  determine $(A'_1, A'_2)$ near $(A_1, A_2)$ by the implicit
  function theorem.  Two (the trigonometric laws $I_2$, $I_3$)
  constrain the
  fundamental Fourier mode of the debris field; and one ($I_4$)
  couples the debris to the compacton configuration.
\item[\emph{(c)}] In an integrable system, the infinite conservation
  hierarchy constrains every radiation mode, forcing zero debris.
  In K(m,m), the five laws leave all debris Fourier modes
  $\hat{w}(k)$ with $k \neq 1$ unconstrained.  This gap between
  finitely many radiation constraints and infinitely many radiation
  degrees of freedom is the mechanism by which K(m,m)
  collisions generate debris.  Whether this mechanism extends
  to other non-integrable dispersive systems remains to be established.
\end{enumerate}
\end{theorem}

\begin{proof}[Argument assuming H1--H3]
Part~(a): By Theorem~\ref{thm:universal}, each post-collision
compacton is determined by a single amplitude parameter $A'_i$ and
a phase shift $\delta_i$.  The debris field $w(x,t)$ adds infinitely
many degrees of freedom through its Fourier spectrum.

Part~(b): The five conservation
laws~\eqref{eq:mass}--\eqref{eq:momentum} provide five independent
scalar equations relating pre- and post-collision quantities.  We
partition these by what they constrain.

\emph{Amplitude constraints.}  Mass conservation gives
$C_m(A'_1 + A'_2) + M_w = C_m(A_1 + A_2)$ and energy conservation
gives $D_m(A'^{\,m+1}_1 + A'^{\,m+1}_2) + E_w = D_m(A_1^{m+1}
+ A_2^{m+1})$, where $C_m, D_m$ are the mass and energy integral
coefficients of~\eqref{eq:mass-evaluated}--\eqref{eq:energy-evaluated}.  In the small-debris regime ($M_w, E_w \ll 1$),
these two equations in two unknowns $(A'_1, A'_2)$ generically
have a unique solution close to $(A_1, A_2)$.

\emph{Radiation constraints.}  A structural property of K(m,m)
compactons is that the trigonometric integrals $I_2 = \int u\sin
x\,dx$ and $I_3 = \int u\cos x\,dx$ vanish identically for every
compacton, regardless of amplitude or position
(Proposition~\ref{prop:trig-vanishing} below).  Conservation of
$I_2$ and $I_3$ therefore constrains the debris alone:
$I_2[w] = 0$ and $I_3[w] = 0$, forcing the fundamental Fourier
mode ($k = 1$) of the debris to vanish.

\emph{Nonlocal constraint.}  The nonlocal momentum $I_4$ couples
the compacton configuration to the debris through the
pseudoinverse $(D_x^2 + 1)^{-1}$.

Part~(c): In an integrable system possessing infinitely many
conservation laws of successively higher order, each law
constraining a progressively higher moment or Fourier component
of the solution (a heuristic correspondence, not a proved
structural statement for general integrable systems).
The infinite collection leaves no free radiation modes, which is
the mechanism underlying the zero-debris
theorem~\cite{ablowitz1981solitons}
(Proposition~\ref{prop:integrable-endpoint}).  In K(m,m), the
hierarchy
terminates at five.  The debris Fourier modes at wavenumbers
$k \neq 1$ are unconstrained by any conservation law.  The
nonlinear collision dynamics naturally excites these modes, and
nothing in the conservation structure forces the excitation to
vanish.  This is the constraint deficit: a finite number of
radiation constraints acting on an infinite-dimensional radiation
field.
\end{proof}

\begin{remark}[Conditional status and paths to proof]
\label{rem:conditional-status}
Theorem~\ref{thm:constraint-deficit} rests on hypotheses (H1)--(H3),
which have different levels of support.
Hypothesis~(H3) rests on Vodov\'a's structural
classification~\cite{vodova2013complete}, which rigorously establishes
that exactly four local conservation laws exist.  The assumption that
no additional nonlocal conserved quantities exist beyond $I_4$ has
strong support from the absence of a recursion operator but lacks a
formal proof.
Hypothesis~(H2) follows from Proposition~\ref{prop:uniqueness}: the
cosine-power compacton is the unique positive, symmetric, single-hump
compact traveling wave, so any emergent traveling wave with these
properties must take this form.
Hypothesis~(H1) has overwhelming numerical support (three decades of
simulations with no counterexamples) but no analytical proof.  The
main open question is whether collision dynamics necessarily produce
traveling-wave structures; orbital stability of K(m,m) compactons
(Problem~3 in Appendix~\ref{app:open-problems}) would resolve this.
Despite the conditional status, the theorem provides the first coherent
explanation of why K(m,m) collisions produce small but nonzero
debris.
\end{remark}

\begin{proposition}[Trigonometric orthogonality]\label{prop:trig-vanishing}
For every K(m,m) compacton $U(\xi) = A\cos^\alpha\!(\xi/\ell)$
with $\alpha = 2/(m-1)$ and $\ell = 2m/(m-1)$, the trigonometric
conservation integrals vanish:
\[
\int_{-\infty}^{\infty} U(x - x_0)\sin(x)\,dx \;=\; 0, \qquad
\int_{-\infty}^{\infty} U(x - x_0)\cos(x)\,dx \;=\; 0
\]
for all $x_0 \in \mathbb{R}$ and all $A > 0$.
\end{proposition}

\begin{proof}
Substituting $\theta = (x - x_0)/\ell$ and expanding the
trigonometric factor gives
$\int U\sin(x)\,dx
= A\ell\sin(x_0)\,B_\alpha + A\ell\cos(x_0)\,S_\alpha$,
where
$B_\alpha = \int_{-\pi/2}^{\pi/2}\cos^\alpha(\theta)
\cos(\ell\theta)\,d\theta$ and $S_\alpha$ is the corresponding
sine integral, which vanishes by parity.  The coefficient
$B_\alpha$ is evaluated using the standard Fourier cosine formula
\[
\int_0^{\pi/2} \cos^a(t)\cos(nt)\,dt
\;=\; \frac{\pi\,\Gamma(a+1)}{2^{a+1}\,
\Gamma\!\bigl(\tfrac{a+n+2}{2}\bigr)\,
\Gamma\!\bigl(\tfrac{a-n+2}{2}\bigr)}.
\]
For K(m,m) compactons, $\ell = \alpha + 2$, so the second Gamma
argument becomes $\Gamma\bigl(\frac{\alpha - \alpha - 2 + 2}{2}
\bigr) = \Gamma(0)$, which has a pole.  Since $1/\Gamma(0) = 0$,
the integral $B_\alpha = 0$ for every $\alpha > 0$.
The identity $\ell - \alpha = 2$ is a direct consequence of
the K(m,m) parameter relations and holds for all $m > 1$.  The
identical argument applies to $\int U\cos(x)\,dx$.  Full details
appear in Appendix~\ref{app:jacobian}.
\end{proof}

\begin{remark}[Structural interpretation of
  Proposition~\ref{prop:trig-vanishing}]
\label{rem:trig-structure}
A resonance between the compacton geometry and the trigonometric
test functions underlies the vanishing $B_\alpha = 0$.  The
compacton support width $\pi\ell$ and the Fourier projection at
wavenumber $k = 1$ are related through the identity
$\ell = \alpha + 2$, which ensures exact cancellation of the
position-dependent contributions.  This orthogonality transforms
the role of $I_2$ and $I_3$ in the collision problem: rather than
constraining phase shifts, they constrain the debris field directly,
requiring its $k = 1$ Fourier component to vanish.  The constraint
deficit therefore operates entirely on the radiation spectrum,
where finitely many constraints (two, from $I_2$ and $I_3$) act on
infinitely many debris modes.
\end{remark}

\paragraph{Constraint partition.}
Table~\ref{tab:constraint-partition} summarizes how the five
conservation laws partition between compacton parameters and debris
degrees of freedom.  The two amplitude-determining constraints
($I_0$, $I_1$) act on a two-dimensional space, generically forcing
amplitude preservation.  The two radiation constraints ($I_2$, $I_3$)
act on a single Fourier mode of the debris.  The nonlocal law $I_4$
provides a coupling constraint.  In total, five scalar conditions
govern an infinite-dimensional post-collision state.

\begin{table}[!ht]
\centering
\small
\setlength{\tabcolsep}{4pt}
\renewcommand{\arraystretch}{1.2}
\begin{tabular}{@{}llll@{}}
\toprule
Law & What it constrains & Target space & Degrees of freedom \\
\midrule
$I_0$ (mass) & $A'_1 + A'_2$ & Compacton amplitudes & $2$ \\
$I_1$ (energy) & $A'^{\,m+1}_1 + A'^{\,m+1}_2$ &
  Compacton amplitudes & (same $2$) \\
$I_2$ (trig-sin) & $\hat{w}(1)\sin$-component & Debris $k\!=\!1$ mode & $2$ \\
$I_3$ (trig-cos) & $\hat{w}(1)\cos$-component &
  Debris $k\!=\!1$ mode & (same $2$) \\
$I_4$ (nonlocal) & Compacton--debris coupling & Mixed & $\infty$ \\
\midrule
\multicolumn{2}{@{}l}{Unconstrained debris modes ($k \neq 1$)}
  & & $\infty$ \\
\bottomrule
\end{tabular}
\caption{Partition of conservation law constraints in a
two-compacton collision.  The five laws constrain the emergent
amplitudes and the fundamental Fourier mode of the debris, leaving
all higher radiation modes unconstrained.  In an integrable system,
every radiation mode would be constrained by a corresponding
conservation law.}
\label{tab:constraint-partition}
\end{table}

Part~(c) is the crux of the constraint deficit argument. In integrable
systems, the infinite conservation hierarchy constrains every
Fourier mode of the solution: the $n$-th conservation law, which
involves spatial derivatives of order~$n$, restricts the $n$-th
spectral component.  The IST translates this algebraic control into
perfect elasticity~\cite{ablowitz1981solitons}. K(m,m) has no such
mechanism.  The debris field acquires energy through nonlinear mode
coupling during the collision, and only two of its infinitely many
spectral components are constrained to vanish afterward.

\FloatBarrier
\paragraph{Numerical confirmation.}
Two concrete predictions emerge from the constraint analysis, and
both are confirmed by high-resolution spectral analysis of the
post-collision debris field (Figure~\ref{fig:debris-spectrum}).
First, the debris Fourier spectrum exhibits a notch at $k = 1$:
the fundamental mode is forced to vanish by conservation of $I_2$
and $I_3$, and the measured $|\hat{w}(1)|$ is suppressed by nearly
two orders of magnitude below the spectral peak for all tested
values of~$m$.  Second, the debris energy resides predominantly in
modes $k \geq 2$, with the spectral peak shifting to higher
wavenumbers as $m$ increases (from $k_{\mathrm{peak}} \approx 0.38$
for $m = 2$ to $k_{\mathrm{peak}} \approx 3.6$ for $m = 5$), consistent
with the narrowing compacton support exciting shorter-wavelength modes
during the collision.  Grid convergence tests from $N_x = 256$
through $N_x = 1024$ confirm that the notch deepens with refinement
while the debris spectrum stabilizes, establishing the notch as
a physical feature rather than a numerical artifact.
Table~\ref{tab:multi-m-results} summarizes the quantitative results.

\begin{figure}[!htbp]
\centering
\includegraphics[width=\textwidth]{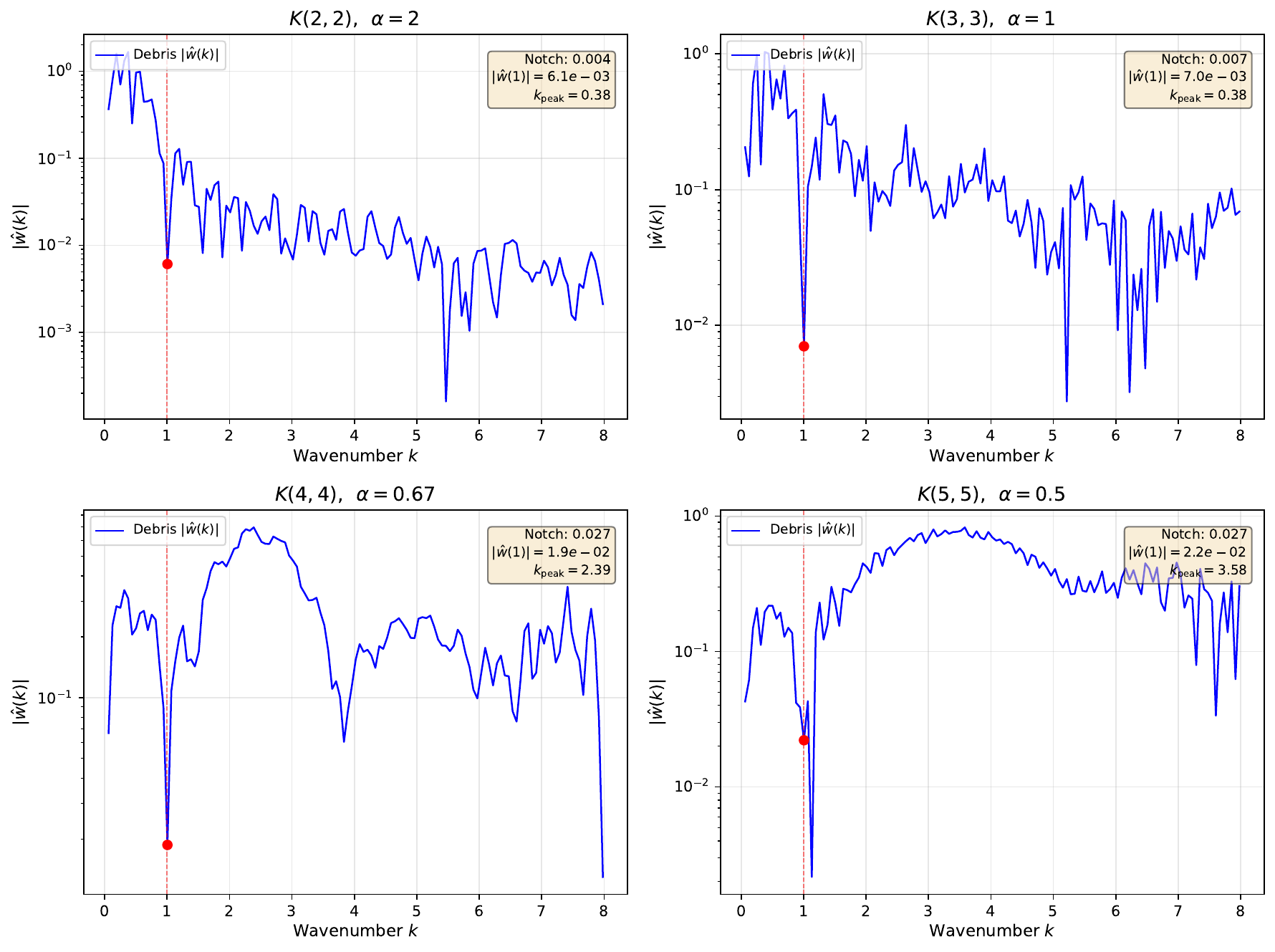}
\caption{Below-cutoff Fourier spectra of post-collision debris fields
for K($m$,$m$) with $m = 2, 3, 4, 5$ and amplitude ratio $A_2/A_1 = 0.5$.
Each panel shows the debris spectrum $|\hat{w}(k)|$ (solid) restricted to
wavenumbers below the low-pass filter cutoff
$\kappa_{\mathrm{cut}} = 2\pi\lceil N_x/(2q)\rceil/L$, above which the
filter removes the advection term
(Appendix~\ref{app:collision-detail}); the dashed vertical line marks
$k = 1$.
A spectral notch near $k = 1$ is visible in all four cases, confirming
the prediction of Theorem~\ref{thm:debris-bound}(c) across the family.
Notch ratios $|\hat{w}(1)|/\max|\hat{w}|$:
0.004 ($m\!=\!2$), 0.007 ($m\!=\!3$), 0.027 ($m\!=\!4$), 0.027 ($m\!=\!5$).
The spectral peak shifts to higher wavenumbers with increasing~$m$,
consistent with the narrowing compacton support exciting
shorter-wavelength debris modes.
Numerical parameters: pseudospectral method (PSM) with $N_x = 512$ grid points on a periodic domain of length $L = 100$,
exponential low-pass (LP) filter with cutoff parameter $q = 2$, biharmonic ramp dissipation with $\delta = 0.15$.
For $m \geq 3$ the debris signal-to-noise ratio exceeds
$15{:}1$, providing unambiguous debris detection
(see Appendix~\ref{app:numerical-method} and~\cite{hyman2026filters}).}
\label{fig:debris-spectrum}
\end{figure}

\begin{table}[!ht]
\centering
\renewcommand{\arraystretch}{1.15}
\begin{tabular}{@{}cccccc@{}}
\toprule
$m$ & $\alpha$ & Notch ratio & $|\hat{w}(1)|$ & $k_{\mathrm{peak}}$
  & $E_{\mathrm{debris}}/E_{\mathrm{tot}}$ \\
\midrule
$2$ & $2$    & $0.004$ & $6.1 \times 10^{-3}$ & $0.38$ & $3.7 \times 10^{-5}$ \\
$3$ & $1$    & $0.007$ & $7.0 \times 10^{-3}$ & $0.38$ & $5.2 \times 10^{-7}$ \\
$4$ & $2/3$  & $0.027$ & $1.9 \times 10^{-2}$ & $2.39$ & $7.4 \times 10^{-7}$ \\
$5$ & $1/2$  & $0.027$ & $2.2 \times 10^{-2}$ & $3.58$ & $1.5 \times 10^{-6}$ \\
\bottomrule
\end{tabular}
\caption{Quantitative debris spectral analysis for K($m$,$m$) collisions
with $A_1 = 1$, $A_2 = 0.5$, $N_x = 512$ (PSM).  The notch ratio
$|\hat{w}(1)|/\max|\hat{w}|$ measures the depth of the spectral
suppression at $k = 1$ predicted by Theorem~\ref{thm:debris-bound}(c);
values below $0.03$ for all cases confirm significant suppression.
The spectral peak wavenumber $k_{\mathrm{peak}}$ increases with~$m$,
reflecting shorter-wavelength debris from narrower compacton
profiles.  The debris energy fraction 
$E_{\mathrm{debris}}/E_{\mathrm{tot}}$ remains below $4 \times 10^{-5}$
for all~$m$, confirming that the post-collision radiation is a
minor perturbation.}
\label{tab:multi-m-results}
\end{table}

The multi-parameter validation of Figure~\ref{fig:debris-spectrum}
confirms that the $k = 1$ notch is a universal feature of K($m$,$m$)
collisions: it persists across all tested values of~$m$, deepening
with grid refinement.  For $m \ge 3$ the debris energy fraction grows with~$m$
(the $m = 2$ case, with its $C^1$ boundary, sits well over an order
of magnitude above), while the notch depth remains consistently deep, and the spectral
peak shifts to higher wavenumbers as compacton support narrows.

\clearpage
\subsubsection*{What the conservation laws do constrain}

While the five laws cannot enforce perfect elasticity, they impose
strong constraints on what the debris field looks like.

\begin{theorem}[Conservation Law Constraints on Debris]\label{thm:debris-bound}
For a two-compacton collision in K(m,m) with pre-collision amplitudes
$A_1, A_2$ and post-collision amplitudes $A'_1, A'_2$, and with the
post-collision compactons and the debris field $w$ of disjoint
support (so that the nonlinear energy is additive),
the debris field $w$ satisfies:
\begin{enumerate}
\item[\emph{(a)}] \textbf{Debris energy.}
$E_w = D_m[(A_1^{m+1} + A_2^{m+1}) - (A'^{\,m+1}_1 + A'^{\,m+1}_2)]$,
where $D_m$ is the energy integral coefficient.
\item[\emph{(b)}] \textbf{Debris mass.}
$M_w = C_m[(A_1 + A_2) - (A'_1 + A'_2)]$,
where $C_m$ is the mass integral coefficient.
\item[\emph{(c)}] \textbf{Fourier mode constraint.}
The Fourier transform of the debris at wavenumber $k = 1$ vanishes:
$\hat{w}(1) = \hat{w}(-1) = 0$.
All other wavenumbers are unconstrained by conservation laws.
\end{enumerate}
\end{theorem}

\begin{proof}
Part~(a) follows from energy conservation $E_{\mathrm{tot}}
= E'_1 + E'_2 + E_w$ with $E[U] = D_m A^{m+1}$. Part~(b) is
analogous using $M[U] = C_m A$. For~(c), conservation of $I_2$
requires $\int w \sin(x)\,dx = I_2^{\mathrm{pre}}
 - I_2[\text{comp.}\ 1'] - I_2[\text{comp.}\ 2']$.
By Proposition~\ref{prop:trig-vanishing}, the compacton contributions
vanish, and $I_2^{\mathrm{pre}} = 0$ as well (since the pre-collision
state consists of compactons only).  Hence $\int w\sin(x)\,dx = 0$.
Since $\int w(x)\sin(x)\,dx = -\tfrac{1}{2i}[\hat{w}(1) -
\hat{w}(-1)]$ (where $\hat{w}$ is the Fourier transform on
$\mathbb{R}$), and symmetrically $\int w\cos(x)\,dx = 0$ from $I_3$
conservation gives $\hat{w}(1) + \hat{w}(-1) = 0$, together these
yield $\hat{w}(\pm 1) = 0$.  The conservation laws impose no
constraint on $\hat{w}(k)$ for $k \neq \pm 1$.
\end{proof}

Theorems~\ref{thm:constraint-deficit}
and~\ref{thm:debris-bound} explain the observed collision
phenomenology from complementary perspectives.
Theorem~\ref{thm:constraint-deficit} shows \emph{why} debris
exists: the conservation laws constrain only the fundamental
Fourier mode of the radiation, leaving all higher modes free.
Theorem~\ref{thm:debris-bound} shows \emph{what} the conservation
laws determine: the debris energy, mass, and the vanishing of its
dominant spatial frequency. The fifth law (nonlocal momentum)
provides an additional coupling constraint.
What remains undetermined (the higher Fourier modes and detailed
spatial profile of the debris) is governed by the full nonlinear
dynamics, not by conservation considerations alone.

\paragraph{Computational methods.}
All collision simulations use the spectrally-filtered pseudospectral
method described in Appendix~\ref{app:numerical-method}: $N_x = 512$
grid points, joint LP filtering (Variant~C, $q = 2$), biharmonic ramp
dissipation with $\delta = 0.15$, and fixed-step RK4 integration.
Post-collision amplitudes are extracted by Nelder--Mead
least-squares fitting of compacton profiles.  Spectral analysis
is restricted to wavenumbers below the filter cutoff; see
Appendix~\ref{app:collision-detail} for additional detail on
debris extraction and signal-to-noise methodology.

\FloatBarrier
\subsection{The conservation-elasticity landscape}
\label{subsec:landscape}

Contrasting K(m,m) with integrable systems sharpens the collision results
of the preceding subsection.

\begin{proposition}[Integrable Endpoint]\label{prop:integrable-endpoint}
For completely integrable dispersive equations with an infinite
hierarchy of independent conservation laws, multi-soliton collisions
are perfectly elastic: $A'_i = A_i$ and $E_w = 0$.
\end{proposition}

This classical result~\cite{gardner1967method,ablowitz1981solitons,lax1968integrals}
follows from the inverse scattering transform: discrete spectral data
is preserved under isospectral flow, and pure discrete spectrum
excludes radiation.

A second reference point comes from the Benjamin--Bona--Mahony
(BBM) equation $u_t + u_x + uu_x - u_{xxt} = 0$, which possesses
exactly three conservation
laws~\cite{benjamin1972model,olver1979euler} and fails every
test of integrability.  Bona, Pritchard, and
Scott~\cite{bona1980solitary} demonstrated numerically that BBM
solitary wave collisions produce permanent amplitude changes and
a trailing dispersive wavetrain.  Martel, Merle, and
Mizumachi~\cite{martel2010description} subsequently proved this
inelasticity rigorously: for nearly equal BBM solitary waves,
the collision generates a dispersive tail whose $H^1$ norm is
bounded below by a positive quantity depending on the amplitude
ratio.  The Martel--Merle--Mizumachi result establishes that
BBM collisions are not perfectly elastic, but it does not
specify how the elasticity compares quantitatively with other
non-integrable systems.

\paragraph{Why compare BBM and K(m,m)?}
The comparison of structurally dissimilar equations is
essential for testing whether conservation law count governs
collision elasticity as a general principle, or only within
particular equation classes.  BBM solitary waves have smooth,
exponentially decaying tails ($\operatorname{sech}^2$ profiles),
while K(m,m) compactons have sharp boundaries with the solution
vanishing identically beyond finite support.  If the
conservation-elasticity relationship were controlled by~$N$
alone, the collision character should depend primarily on the
number of conservation laws regardless of interaction geometry.
If, on the other hand, tail behavior or support structure
modulates the relationship, comparing equations with different
interaction geometries would reveal this additional dependence.
We therefore computed BBM collision elasticity using the same
amplitude-based measure~\eqref{eq:elasticity} applied to K(m,m)
in~\S\ref{subsec:quasi-elastic}.

\paragraph{BBM collision simulations.}
We solve the BBM equation using a pseudospectral method
($N_x = 4096$ grid points, adaptive Runge--Kutta (RK45) time
stepping) and measure collision
elasticity via~\eqref{eq:elasticity} across a range of amplitudes
and amplitude ratios.
Full simulation parameters, validation details, and individual
collision measurements appear in
Appendix~\ref{app:bbm-detail}.

\paragraph{Result.}
Across all tested parameter combinations, the BBM collision
elasticity satisfies
\begin{equation}\label{eq:bbm-elasticity}
\mathcal{E}_{\mathrm{BBM}} \approx 0.9993\text{--}1.0001,
\end{equation}
with debris energy fractions of order $10^{-4}$ to~$10^{-5}$.
This is comparable to the K(m,m) compacton
collisions measured in~\S\ref{subsec:quasi-elastic}, where
$\mathcal{E}_{\mathrm{K}(m,m)} \approx 0.999$--$1.000$
despite K(m,m) possessing five conservation laws to BBM's three.
By debris energy the K(m,m) collisions are the cleaner: at
$A_1 = 1$, $A_2 = 0.5$ the BBM debris fraction is
$1.4 \times 10^{-4}$ (Supplement, Table~S3), larger than every
K(m,m) value in Table~\ref{tab:multi-m-results}.

The result is consistent with the rigorous analysis of
Martel, Merle, and Mizumachi, who proved the existence of a
nonzero dispersive residual without estimating its magnitude.
The near-elasticity of BBM collisions reflects the perturbative
nature of interactions between smooth, exponentially localized
waves: the overlap region evolves slowly, transferring very
little energy to radiation.  K(m,m) compactons interact through abrupt contact at their support
boundaries, yet they shed less debris energy, consistent with their
larger conservation law count.

This single comparison does not separate the effects of
conservation law count and interaction geometry: by the amplitude
measure the two equations are indistinguishable at the resolution of
these runs, and by debris energy the K(m,m) collisions are the
cleaner.  Other structural differences, such as BBM's
uni-directional phase space, which limits scattering channels, may
also contribute.  The landscape relating conservation structure to
collision outcomes likely depends on several factors.

Taken together, Theorem~\ref{thm:constraint-deficit},
Theorem~\ref{thm:debris-bound},
Proposition~\ref{prop:integrable-endpoint}, and the BBM
comparison place compacton collisions within a broader landscape
relating conservation structure and interaction geometry to
collision outcomes.

The conservation-elasticity landscape may be richer than a simple
monotone relationship between conservation count and elasticity.
BBM solitary waves have smooth exponential tails that mediate
gradual, perturbative interactions, while K(m,m) compactons interact
through abrupt contact at their support boundaries, so the two
families differ in interaction geometry as well as in conservation
law count.  The BBM equation at $N = 3$ collides about as
elastically as K(m,m) at $N = 5$ by the amplitude measure, while
shedding more debris energy, which leaves open whether geometry
modulates the role of the count.  This motivates a refined
conjecture.

\begin{conjecture}[Within-Class Conservation-Elasticity Correlation]\label{conj:monotonicity}
For nonlinear dispersive equations sharing a common interaction
geometry (solitary wave tail behavior), the collision elasticity
$\mathcal{E}$ at fixed collision parameters is expected to
correlate strongly with the number $N$ of independent conservation
laws, with $\mathcal{E} \to 1$ as $N \to \infty$.  Across geometric
classes, the interaction geometry introduces an independent
modulation: equations with smoother, more gradually decaying
solitary wave profiles are expected to produce more elastic
collisions at equal conservation law count.
\end{conjecture}

The within-class component of
Conjecture~\ref{conj:monotonicity} is anchored by two proved
endpoints: $N = \infty$ gives $\mathcal{E} = 1$
(Proposition~\ref{prop:integrable-endpoint}), and the constraint
deficit of Theorem~\ref{thm:constraint-deficit} explains why
finite~$N$ permits $\mathcal{E} < 1$.  Within the K(m,m)
family (where all members share the same compact-support
geometry), the five conservation laws produce quasi-elastic
collisions with $\mathcal{E}$ within $2 \times 10^{-3}$ of~$1$ and
debris energy fractions below $4 \times 10^{-5}$, close to but not
at the integrable endpoint.

The cross-class component remains open.  The BBM comparison is not
made at equal conservation law count, and it does not isolate
geometry: BBM, with three laws, collides about as elastically as
K(m,m) by the amplitude measure and sheds more debris energy.  A
test of the cross-class component needs equation families that share
a conservation law count but differ in tail behavior.

\paragraph{Limitations of the evidence.}
The within-class monotonicity component rests on one equation
family (K(m,m)) and the integrable endpoints; testing with
additional families would strengthen the claim.
The cross-class component has so far been probed by a single pair
(BBM versus K(m,m)), which does not resolve it.  Within K(m,m) itself, the case
$m = 3$ produces somewhat more elastic collisions than $m = 2$
despite the same conservation law count, suggesting that the
specific algebraic forms of the laws influence elasticity
beyond the raw count.  An additional subtlety, noted by
Garral\'on-L\'opez et al.~\cite{garralon2023compacton}, is an
odd-even asymmetry in collision robustness across the family.
Whether interaction geometry acts as a second independent factor,
and if so how it depends on geometric parameters, awaits further
investigation (Problem~4 in Appendix~\ref{app:open-problems};
extended discussion in Appendix~\ref{app:bbm-detail}).

\subsection{Broader implications: constraint deficit as a general framework}
\label{subsec:broader-implications}

The constraint-deficit mechanism, developed above for K(m,m)
compactons, rests on three ingredients that are not specific to any
single equation: a finite set of conservation laws, solitary wave
solutions parameterized by amplitude and speed, and collisions that
produce localized emergent structures plus a radiation field.  We
conjecture that any dispersive PDE possessing these ingredients admits
a constraint-deficit analysis, and that the framework may provide a quantitative lens on the
intermediate territory between integrable order and generic nonlinear
dynamics.

To make this precise, consider a dispersive PDE with $N$ independent
conservation laws and a family of solitary waves parameterized by a
single continuous variable (typically amplitude).  After a binary
collision, the post-collision state involves two emergent solitary
waves (parameterized by $2$ amplitudes and $2$ phase shifts) and a
radiation field with infinitely many degrees of freedom.  The $N$
conservation laws impose $N_c = N$ scalar constraints on the
post-collision state.  If $N_c \geq 4$
(two amplitude constraints plus at least two radiation constraints),
the emergent amplitudes are generically determined and at least some
radiation modes are suppressed.  The remaining radiation modes,
$\infty - (N_c - 2)$ in a formal count, are unconstrained.  As
$N_c \to \infty$, all radiation modes become constrained, recovering
perfect elasticity: the integrable limit.

This counting argument, which underlies
Theorem~\ref{thm:constraint-deficit}, suggests an analogous
analysis for any system satisfying the three candidate
prerequisites identified from the K(m,m) analysis.  One can then
define a quantitative measure, the \emph{conservation elasticity},
capturing the degree to which a finite conservation hierarchy
approximates the integrable ideal:

\begin{definition}[Conservation Elasticity]\label{def:conservation-elasticity}
For a dispersive PDE with $N$ independent conservation laws and
solitary wave collisions producing debris with energy $E_w$, the
\emph{conservation elasticity} is
\[
\eta(N) \;=\; 1 - \frac{E_w}{E_{\mathrm{tot}}},
\]
where $E_{\mathrm{tot}}$ is the total energy.  The integrable limit
corresponds to $\eta = 1$; generic Hamiltonian systems with $N = 2$
yield $\eta \ll 1$.
\end{definition}

\noindent For K(m,m) compactons, $\eta$ coincides with the
energy-based elasticity $\mathcal{E}_E$ of
Appendix~\ref{app:collision-detail}.  Theorem~\ref{thm:universal}
ensures that $\eta$ and the amplitude-based
elasticity~$\mathcal{E}$ are monotonically related, so the
within-family ordering is independent of the choice of measure.

The K(m,m) family provides the first systematic test:
$\eta$ ranges from $0.99996$ (at $m = 2$) to $0.9999995$ (at
$m = 3$) with $N = 5$ (Table~\ref{tab:multi-m-results}), while the
BBM comparison gives $\eta \approx 0.9994$--$0.99999$ at $N = 3$.  Extending this analysis to
other families with known conservation counts would begin to map the
$\eta$-vs-$N$ relationship.

The broader program of testing the constraint-deficit framework
across equation families remains conjectural; we return to this
question and its implications in the Concluding
Remarks~(\S\ref{sec:conclusion}).

\section{Compacton emergence from localized initial data}
\label{sec:emergence}

The collision experiments of the preceding section demonstrate that
pre-formed compactons survive interaction almost elastically.  A
complementary question probes the robustness of compactons from the
opposite direction: given a localized initial pulse that is
\emph{not} a compacton, does the K$(m,m)$ evolution decompose it
into compactons?  For the integrable KdV equation, the inverse
scattering transform provides an exact answer: an arbitrary
sufficiently smooth localized initial pulse decomposes into a finite
number of solitons whose amplitudes are determined by the
Schr\"odinger eigenvalues of the associated spectral
problem~\cite{lax1968integrals,miura1968korteweg}.
No analogous theory exists for the non-integrable
K$(m,n)$ family.  Numerical experiments reveal, however, that
compacton emergence follows strikingly regular patterns.

\subsection{Wide-compacton initial data}

To avoid Gibbs oscillations associated with discontinuous
initial data, we initialize the computation with a
\emph{wide compacton}, the K$(m,m)$ compacton shape stretched
by a width factor~$w$:
\[
u_0(x) =
\begin{cases}
A\cos^{\alpha}\!\Bigl(\dfrac{x - x_0}{w\ell}\Bigr),
  & |x - x_0| \leq w\pi\ell/2,\\[6pt]
0, & \text{otherwise},
\end{cases}
\]
with $\alpha$ and $\ell$ from Theorem~\ref{thm:universal}, so that
the half-width is $w$ times the natural half-width
$\pi\ell/2 = m\pi/(m-1)$.  We take $w = 2.5$ and $A = 1$ in all
cases.  This continuous, compactly supported pulse contains sufficient
excess width to decompose into multiple compactons while remaining
within the class of functions for which the numerical method
performs reliably.

\subsection{Emergence dynamics}
\label{subsec:emergence-dynamics}

Figure~\ref{fig:emergence} presents the evolution of a wide
initial pulse under the K$(m,m)$ equation for $m = 2, 3, 4, 5$.
In every case, the initial pulse rapidly steepens at the leading
edge, sheds compactons from the front, and leaves a dispersive
remnant at the rear.  By $t \approx 50$, three or more distinct
compactons have separated, propagating rightward at speeds governed
by the universal formula $c = \tfrac{m+1}{2m}\,A^{m-1}$ from
Theorem~\ref{thm:universal}.  The tallest (fastest)
compacton leads; progressively smaller compactons trail behind.
For $m = 2$, four compactons separate cleanly; for K$(3,3)$ a
marginal fourth appears at the noise floor (a peak amplitude of
$0.08$, below which peaks are not counted); for $m \geq 4$,
three are unambiguous while the trailing dynamics fall at or below
the numerical noise floor (see below).

\begin{figure}[!ht]
\centering
\includegraphics[width=\textwidth]{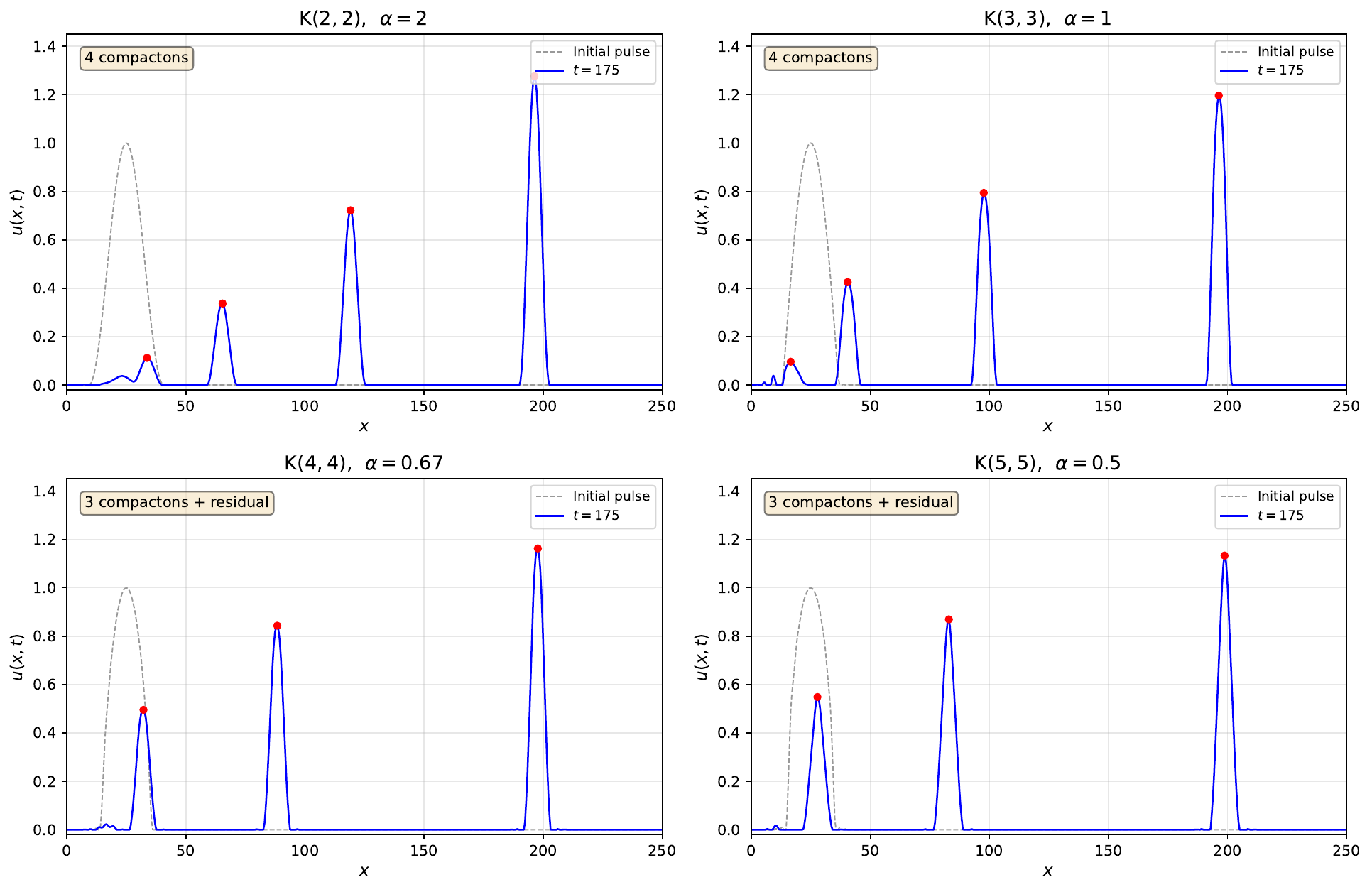}
\caption{Compacton emergence from a wide initial pulse
($w = 2.5$ times the natural half-width, amplitude $A = 1$)
for K$(m,m)$ with $m = 2, 3, 4, 5$.  Each panel shows the
initial profile (dashed grey) and the state at $t = 175$ (solid blue);
red markers indicate detected compacton peaks above the noise floor.
The boundary regularity exponent $\alpha = 2/(m-1)$ controls both
the Gibbs amplitude at the support boundary and the resolvability
of trailing compactons.  PSM with $N_x = 512$, $q = 2$;
graded dissipation $\delta \in [0.15, 0.30]$.}
\label{fig:emergence}
\end{figure}

Table~\ref{tab:emergence} reports the amplitudes and speeds of
the emerged compactons for each K$(m,m)$ equation and
for the mixed-exponent cases K$(2,3)$ and K$(3,2)$.  Three patterns
are evident.

First, the \emph{leading compacton speeds are remarkably similar
across the K$(m,m)$ family}, clustering between $0.95$ and $0.99$
despite different values of~$m$.  This is an empirical property of the emerged
leading amplitudes $A_1(m) = 1.28, 1.20, 1.16, 1.13$, which offset
the falling coefficient $(m+1)/(2m)$ through the factor $A_1^{m-1}$.

Second, \emph{trailing compactons separate at rates that depend
sensitively on~$m$}.  For $m = 2$, the linear speed--amplitude
relation $c \propto A$ produces moderate speed differences between
large and small compactons.  For $m \geq 3$, the power-law
relation $c \propto A^{m-1}$ causes the smallest resolved
compacton to be nearly stationary ($c_3 \leq 0.12$), creating a
large speed spread between the leading and trailing structures.
This $m$-dependent separation rate has no counterpart in the
integrable case, where the speed is always linear in amplitude.

Third, the \emph{emerged amplitudes are constrained by the
conservation laws}.  The four local conserved quantities
(Theorem~\ref{thm:finite-hierarchy}), together with $I_4$ where the
initial pulse satisfies $\hat u(\pm 1) = 0$ ($m = 2, 4$ here), relate
the initial pulse parameters to the emergent compacton amplitudes.
While these constraints cannot determine the amplitudes and
positions of $N > 1$ compactons together with the dispersive
remnant, they do
restrict the set of admissible configurations.

\begin{table}[!ht]
\centering
\renewcommand{\arraystretch}{1.15}
\caption{Compacton emergence from a wide initial pulse of unit
amplitude.  Amplitudes $A_i$ measured from peak detection at
$t = 175$; speeds $c_i$ computed from unrounded amplitudes with the
exact traveling-wave relation $c = [(n+1)/(m+n)]\,A^{m-1}$, which reduces to
Theorem~\ref{thm:universal} for K$(m,m)$ (validated against direct
propagation-speed measurements; see implications discussion).
For K$(2,2)$, a fourth compacton ($A_4 = 0.11$) is cleanly resolved;
for K$(3,3)$ a marginal fourth ($\dagger$) sits at the $0.08$ noise
floor; for $m \geq 4$ additional trailing structures fall below it
(see text).  For K$(4,4)$ and K$(5,5)$ the implied $\sum_i A_i^{m+1}$
exceeds the initial pulse energy $2.5$ by $2\%$, comparable to the
$\pm 2$--$3\%$ spread that two-decimal rounding of the amplitudes
alone induces; since grid sampling and dissipation both lower the
measured sum, any real excess would indicate that the emerged
structures at $t = 175$ are not yet pure compactons of the listed
amplitudes (overlapping remnant), which a profile refit would
settle.
Upper block: K$(m,m)$ with $m \in \{2,3,4,5\}$; lower block:
mixed-exponent cases K$(2,3)$ and K$(3,2)$.
PSM with $N_x = 512$, $q = 2$; $N$ is the number of resolved
compactons.}
\label{tab:emergence}
\begin{tabular}{@{}l cccc c cccc c@{}}
\toprule
 & \multicolumn{4}{c}{Amplitude $A_i$} & & \multicolumn{4}{c}{Speed $c_i$} & Spread \\
\cmidrule{2-5} \cmidrule{7-10}
Equation & $A_1$ & $A_2$ & $A_3$ & $A_4$ & & $c_1$ & $c_2$ & $c_3$ & $c_4$ & $c_1 - c_N$ \\
\midrule
K$(2,2)$ & $1.28$ & $0.72$ & $0.34$ & $0.11$ & & $0.96$ & $0.54$ & $0.25$ & $0.08$ & $0.87$ \\
K$(3,3)$ & $1.20$ & $0.79$ & $0.42$ & $0.10^{\dagger}$ & & $0.95$ & $0.42$ & $0.12$ & $0.007^{\dagger}$ & $0.94$ \\
K$(4,4)$ & $1.16$ & $0.84$ & $0.50$ & --- & & $0.98$ & $0.38$ & $0.08$ & --- & $0.90$ \\
K$(5,5)$ & $1.13$ & $0.87$ & $0.55$ & --- & & $0.99$ & $0.34$ & $0.05$ & --- & $0.94$ \\
\midrule
K$(2,3)$ & $1.14$ & $0.61$ & $0.35$ & $0.20$ & & $0.91$ & $0.49$ & $0.28$ & $0.16$ & $0.75$ \\
K$(3,2)$ & $1.33$ & $0.92$ & $0.47$ & --- & & $1.06$ & $0.51$ & $0.13$ & --- & $0.93$ \\
\bottomrule
\end{tabular}
\end{table}

\paragraph{Resolution limit from the regularity hierarchy.}
For $m \geq 4$, the fourth compacton amplitude is not reported
because it falls below the numerical noise floor; for K$(3,3)$ it
is listed with a dagger.  Two
competing trends drive this resolution limit.  The emerged
amplitude $A_4$ decreases with $m$: cleanly resolved at
$A_4 = 0.11$ for K$(2,2)$, dropping to $A_4 \approx 0.10$ for
K$(3,3)$ where it lies near the $0.08$ noise threshold; no fourth
peak is detected for K$(4,4)$ or K$(5,5)$ at the present
spectral resolution.  Simultaneously, the Gibbs
oscillation amplitude at the compact-support boundary
\emph{increases} with~$m$, scaling with the boundary regularity
exponent $\alpha = 2/(m-1)$.  The $C^1$ boundary at $m = 2$
($\alpha = 2$) produces negligible Gibbs ringing, while the
H\"older boundary at $m = 5$ ($\alpha = 1/2$) produces substantial
overshoots that propagate as a spurious dispersive
wake~\cite{hyman2026filters}.  These trends cross near
$m = 3$, beyond which the expected fourth compacton amplitude
lies within or below the spectral noise floor.

This resolution limit is controlled by the same regularity
hierarchy that governs the convergence rates
of the numerical method~\cite{hyman2026filters}: the boundary
exponent $\alpha = 2/(m-1)$ simultaneously determines how many
compactons can be resolved and how rapidly the scheme converges
under grid refinement.  Whether a fourth compacton forms for $m \ge 4$ cannot be decided at
the present spectral resolution.

\subsection{Advection--dispersion balance in the mixed-exponent cases}

The comparison between K$(2,3)$ and K$(3,2)$ exposes the role
of the advection--dispersion balance in the emergence process.
When the advective exponent exceeds the dispersive one ($m > n$,
as in K$(3,2)$), energy concentrates into fewer, taller compactons:
K$(3,2)$ produces the largest leading amplitude ($A_1 = 1.33$) and
a speed spread across resolved compactons ($0.93$, only three
structures resolvable) comparable to the largest K$(m,m)$ value.  The reverse case $m < n$
(K$(2,3)$) distributes energy more uniformly across the emerged
compactons, yielding a narrower amplitude range and the slowest
separation of any case tested ($c_1 - c_4 = 0.75$).

This asymmetry admits a physical interpretation.  The exponent~$m$
controls nonlinear steepening at the compacton peak, where
$u \sim O(1)$, while the exponent~$n$ governs dispersive spreading
at the edges, where $u \ll 1$.  Stronger advection relative to
dispersion favors the formation of coherent, high-amplitude
structures.

\subsection{Implications for the structure-beyond-integrability framework}

The emergence experiments provide a third line of evidence, alongside
the universal formulas
(Theorem~\ref{thm:universal}) and the quasi-elastic
collisions (\S\ref{sec:collision-dynamics}), that the K$(m,m)$ equations
possess more structure than a generic non-integrable system.  In
particular:

\begin{enumerate}
\item Direct measurement of the leading compacton's propagation
  speed (by linear regression of peak positions across snapshots
  at $t = 100, 125, 150, 175$) confirms the universal formula
  $c = [(m+1)/(2m)]\,A^{m-1}$ to within $1\%$ for all four
  K$(m,m)$ cases.  For the mixed-exponent cases the leading speed
  obeys the exact traveling-wave relation
  $c = [(n+1)/(m+n)]\,A^{m-1}$, obtained from the crest of the first
  integral of the K$(m,n)$ profile equation; the exact speed
  equals the K$(m,m)$ prediction, evaluated with the advective
  exponent~$m$, times the factor $2m(n+1)/[(m+n)(m+1)]$, that is
  $1.067$ ($+6.7\%$) for K$(2,3)$ and $0.900$ ($-10.0\%$) for
  K$(3,2)$; the directly
  measured leading speeds then agree with the exact relation to
  within $0.7\%$ for K$(2,3)$ and $3.1\%$ for K$(3,2)$
  (Supplement, Table~S6).
\item The leading compacton speed $c_1$ is approximately
  independent of~$m$ for unit-amplitude initial data, varying
  only between $0.95$ and $0.99$ across $m \in \{2,3,4,5\}$, a
  universality not predicted by any existing theory.
\item The emergence process respects the four local conservation
  laws of Theorem~\ref{thm:finite-hierarchy} (and $I_4$ for
  $m = 2, 4$, where the initial pulse lies in its domain), and the constraint-deficit
  mechanism of Theorem~\ref{thm:debris-bound} should apply to the
  dispersive remnant left behind.
\end{enumerate}

Developing a predictive theory for compacton emergence (analogous
to inverse scattering for KdV) remains an open problem.
The numerical results in Table~\ref{tab:emergence}
provide quantitative targets for such a theory.  A natural
first step would be to establish whether the number of emerged
compactons depends on any calculable property of the initial pulse
(its mass $I_0$, energy $I_1$, or some analogue of the
Schr\"odinger eigenvalue count).

\section{Traveling waves on a positive background}
\label{sec:shelf-waves}

The compactons studied in Sections~\ref{sec:structure-beyond}
and~\ref{sec:nonintegrability} arise as traveling waves on a zero
background: the undisturbed state is $u = 0$, and the wave amplitude
returns to zero at finite distance.  A natural generalization replaces
the zero background with a constant positive shelf $u = a > 0$ and
seeks traveling wave solutions of the form
\begin{equation}\label{eq:shelf-ansatz}
u(x,t) = a + f(x - ct), \qquad a > 0, \quad
f(\xi) \to 0 \ \text{as}\ |\xi| \to \infty.
\end{equation}
This extension has physical motivation: real wave phenomena rarely
occur on quiescent backgrounds, and the positive shelf introduces a
characteristic scale that interacts nontrivially with the nonlinear
dispersion.  It also has mathematical motivation: the background
parameter $a$ controls a bifurcation in the traveling wave phase
portrait that reveals the compacton as a singular limit of smooth
solitary and periodic waves.

The analysis reinforces the constraint-deficit framework of
\S\ref{subsec:quasi-elastic}.  The compacton's emergence as an
isolated object (not part of a continuous family of smooth
traveling waves) explains why collision outcomes cannot smoothly
interpolate between elastic and inelastic: the compacton is
structurally rigid, and any perturbation either preserves it or
generates debris.  The singular limit analysis provides independent
confirmation of the compacton's special status within the K(m,m)
traveling-wave landscape.

Lax's theory of KdV periodic solutions~\cite{lax1975periodic}
constructs $N$-phase quasi-periodic waves, which reduce to cnoidal
waves for $N = 1$, by variational methods, using the infinitely many conserved functionals $F_j$ in
involution~\cite{lax1976almost}.  The solutions lie on
$N$-dimensional invariant tori, parametrized by spectral data from
the associated Schr\"odinger operator.  A central question for the
non-integrable K(m,m) family is: what periodic wave structure
survives when the infinite conservation hierarchy collapses to a
finite set?

This section answers that question through a systematic phase-plane
analysis.  The main results are: a critical speed formula
$c^* = ma^{m-1}$ that separates solitary from periodic behavior
(\S\ref{subsec:phase-classification}); explicit first-integral
and quadrature formulas for K(2,2) shelf waves
(\S\ref{subsec:k22-explicit}); a singular limit theorem showing
how the compacton emerges through simultaneous loss of regularity
and support compactification
(\S\ref{subsec:singular-limit}); and a precise comparison between
the resulting one-parameter periodic families and the rich
multi-dimensional torus structure of Lax's integrable theory
(\S\ref{subsec:lax-periodic-connection}).

\subsection{Hamiltonian reduction and the \texorpdfstring{$v$}{v}-variable}
\label{subsec:shelf-reduction}

Substituting the traveling waves $f(x - ct)$ on a shelf $a > 0$, $u = a + f$,  into the K(m,m)
equation~\eqref{eq:kmm}, writing $\xi = x - ct$, and integrating
once yields
\begin{equation}\label{eq:shelf-tw-ode}
-c(a + f) + (a + f)^m + \bigl[(a+f)^m\bigr]'' = C_0,
\end{equation}
where the integration constant $C_0 = -ca + a^m$ is determined by
the boundary condition $f \to 0$, $f' \to 0$ as $|\xi| \to \infty$.
Rearranging gives a second-order ODE for the total field $a + f$.

The equation simplifies through the substitution
\begin{equation}\label{eq:v-variable}
v = (a + f)^m,
\end{equation}
which converts the degenerate-dispersion structure into polynomial
form.  Since $a + f > 0$ for shelf waves (the field remains above
zero), this change of variable is a smooth diffeomorphism.  In the
$v$-variable, equation~\eqref{eq:shelf-tw-ode} becomes
\begin{equation}\label{eq:v-ode}
v'' + V'(v) = 0,
\end{equation}
where the potential is
\begin{equation}\label{eq:potential}
\boxed{V(v) = \frac{1}{2}\,v^2
  - \frac{cm}{m+1}\,v^{(m+1)/m}
  + (ca - a^m)\,v.}
\end{equation}
Equation~\eqref{eq:v-ode} is a Hamiltonian ODE with conserved
energy
\begin{equation}\label{eq:shelf-energy}
E = \tfrac{1}{2}(v')^2 + V(v).
\end{equation}

When $a = 0$, the substitution $v = u^m$ is singular at $u = 0$:
this is the point where the degenerate dispersion coefficient
$u^{m-1}$ vanishes and compact support becomes possible.  For
$a > 0$, the total field remains strictly positive and the
$v$-variable provides a globally smooth coordinate.  The transition
$a \to 0^+$ transforms the smooth Hamiltonian system into a
degenerate one, a mechanism we analyze in
\S\ref{subsec:singular-limit}.

\subsection{Phase-plane classification}
\label{subsec:phase-classification}

The equilibria of~\eqref{eq:v-ode} satisfy $V'(v) = 0$.  Direct
computation gives
$V'(v) = v - c\,v^{1/m} + (ca - a^m)$,
and substitution of $v_0 = a^m$ confirms that $V'(a^m) = 0$
for every $c$ and $a > 0$.  This is the equilibrium corresponding
to the undisturbed shelf $f = 0$.

\begin{theorem}[Critical speed and equilibrium classification]
\label{thm:critical-speed}
For K(m,m) with $m \geq 2$ and shelf height $a > 0$, define the
critical speed
\begin{equation}\label{eq:critical-speed}
\boxed{c^* = m\,a^{m-1}.}
\end{equation}
The equilibrium $v_0 = a^m$ of the potential~\eqref{eq:potential}
has the following character:
\begin{enumerate}
\item[\textnormal{(a)}] If $c < c^*$, then $V''(v_0) > 0$ and
  $v_0$ is a center: the phase plane contains a family of nested
  periodic orbits.
\item[\textnormal{(b)}] If $c = c^*$, then $V''(v_0) = 0$ and
  $v_0$ is a degenerate equilibrium.
\item[\textnormal{(c)}] If $c > c^*$, then $V''(v_0) < 0$ and
  $v_0$ is a saddle point: the phase plane contains a homoclinic
  orbit enclosing the region $v_0 < v < v_1$.
\end{enumerate}
\end{theorem}

\begin{proof}
Differentiating $V'(v) = v - c\,v^{1/m} + (ca - a^m)$ yields
\begin{equation}\label{eq:V-second-deriv}
V''(v) = 1 - \frac{c}{m}\,v^{(1-m)/m}.
\end{equation}
At $v_0 = a^m$, this evaluates to
$V''(a^m) = 1 - c/(m\,a^{m-1}) = 1 - c/c^*$.
The sign of $V''(a^m)$ determines the character of the equilibrium:
positive gives a center (restoring force), negative gives a saddle
(repelling), and zero gives the degenerate boundary case.
\end{proof}

The critical speed formula encodes a competition between the
propagation velocity $c$ and the effective dispersion speed
$ma^{m-1}$ set by the shelf height.  When the wave travels faster
than this threshold, the shelf equilibrium becomes unstable and a
localized excursion (solitary wave) is possible.  When it travels
more slowly, the equilibrium is stable and only oscillatory
(periodic) solutions exist near it.

\begin{theorem}[Existence of solitary and periodic waves]
\label{thm:existence}
Under the hypotheses of Theorem~\ref{thm:critical-speed}:
\begin{enumerate}
\item[\textnormal{(a)}] \textbf{Solitary waves.}
  For $c > c^*$, the energy surface
  $E = V(a^m)$ contains a unique homoclinic orbit in the half-plane
  $v > 0$.  In the original variable, $u = a + f$ where
  $f(\xi) > 0$ with $f(\xi) \to 0$ exponentially as
  $|\xi| \to \infty$.  The profile is $C^\infty$ and symmetric
  about its maximum.

\item[\textnormal{(b)}] \textbf{Periodic waves.}
  For $c < c^*$, a one-parameter family of periodic orbits exists
  near $v_0 = a^m$, parametrized by energy
  $E \in (V(v_0), E_{\mathrm{max}})$.  The period satisfies
  $T(E) \to 2\pi/\omega_0$ as $E \to V(v_0)^+$ and
  $T(E) \to \infty$ as $E \to E_{\mathrm{max}}$ if $a^{m-1} < c < c^*$,
  where $E_{\mathrm{max}} = V(v_s)$ at the saddle $v_s > 0$; if
  $0 < c \le a^{m-1}$ there is no saddle in $v > 0$, the family ends at
  the orbit through $v = 0$ with $E_{\mathrm{max}} = V(0)$, and
  $T(E_{\mathrm{max}})$ is finite (equal to $2\pi m/(m-1)$ at
  $c = a^{m-1}$).  Here
  \begin{equation}\label{eq:small-osc-freq}
  \omega_0 = \sqrt{V''(a^m)}
    = \sqrt{1 - c/(m\,a^{m-1})}.
  \end{equation}
\end{enumerate}
\end{theorem}

\begin{proof}
Part (a): When $c > c^*$, the equilibrium $v_0 = a^m$ is a saddle.
The separatrix energy is $E_{\mathrm{sep}} = V(a^m)$.  On the level
set $E = E_{\mathrm{sep}}$, the equation $V(v) = V(a^m)$ with
$v > a^m$ has a unique positive root $v_1$ (the crest value),
because $V(v) \to +\infty$ as $v \to +\infty$ while
$V(v) < V(a^m)$ for $v$ slightly above $a^m$.  The resulting
homoclinic orbit connects the saddle to itself through the excursion
$v_0 \to v_1 \to v_0$.  Symmetry of the orbit follows from the
time-reversibility $\xi \mapsto -\xi$ of~\eqref{eq:v-ode}.
Smoothness is inherited from the smooth Hamiltonian flow away from
the singular set $v = 0$, which the orbit never approaches since
$v \geq a^m > 0$.

Part (b): Standard Hamiltonian center theory.  The Lyapunov
center theorem guarantees a family of periodic orbits near any
non-degenerate center.  The period formula follows from the
quadrature $T(E) = 2\int_{v_{\min}}^{v_{\max}}
dv/\sqrt{2(E - V(v))}$, and the stated limits are standard
consequences of the linearization frequency $\omega_0$ and the
logarithmic divergence of the integral near the separatrix when a
saddle bounds the family; for $c \le a^{m-1}$ the boundary orbit
reaches $v = 0$, where the integrand behaves like $v^{-1/2}$ for
$c < a^{m-1}$ and like $v^{-(m+1)/(2m)}$ at $c = a^{m-1}$, integrable
in both cases, so the limiting period is finite.
\end{proof}

The next result strengthens part~(b) of
Theorem~\ref{thm:existence} by establishing strict
monotonicity of the period function.

\begin{proposition}[Period monotonicity]\label{prop:period-monotone}
For $m = 2, 3$ and every $c \in (0, c^*)$, the period function
$T(E)$ of the K(m,m) periodic family is strictly increasing:
$T'(E) > 0$ for all $E \in (V(v_0), E_{\mathrm{max}})$.
\end{proposition}

\begin{proof}
The reduced dynamics~\eqref{eq:v-ode} form the one-degree-of-freedom
Hamiltonian system $\tfrac12(v')^2+V(v)=E$ with the
potential~\eqref{eq:potential}. For $c\in(0,c^*)$ the equilibrium
$v_0=a^m$ is a nondegenerate center,
\[
V'(v_0)=0,\qquad
V''(v_0)=1-\frac{c}{c^*}>0,\qquad c^*=m\,a^{m-1}.
\]
Chicone's criterion~\cite{chicone1987monotonicity} applies to a
Hamiltonian $\tfrac12 y^2+W(x)$ with a center at the origin: the period
function is monotone increasing when $W/(W')^2$ is convex on the period
annulus. Shifting the center to the origin by $w=v-v_0$ and writing
$\widetilde V(w)=V(v_0+w)-V(v_0)$, the Chicone function
$\chi=\widetilde V/(\widetilde V')^2$ satisfies
\[
\chi''=\frac{N}{(\widetilde V')^4},\qquad
N=6\,\widetilde V\,(\widetilde V'')^2-3(\widetilde V')^2\widetilde V''
   -2\,\widetilde V\,\widetilde V'\,\widetilde V''' .
\]
Since $\widetilde V'''(w)=V'''(v_0+w)=\dfrac{c(m-1)}{m^{2}}\,(v_0+w)^{1/m-2}>0$,
the substitution to the physical field $U=v^{1/m}$ rationalizes $N$ and
exhibits the center as a quadruple root:
\[
N=\frac{c\,(U-a)^4\,R_m(U)}{\kappa_m\,U^{\,2m-1}},\qquad \kappa_m>0 \text{ (Appendix~\ref{app:symbolic})},
\]
so $\operatorname{sgn}N=\operatorname{sgn}R_m(U)$ on the annulus. For
$m=2$, $R_2(U)=3U^2+(12a-8c)U+(3a-2c)(a-c)$ is an upward parabola with
$R_2(0)=(3a-2c)(a-c)\ge 0$ when $c\le a$, and $R_2(c-a)=3(c-a)(2a-c)>0$ at
the inner (saddle) edge when $a<c<c^*=2a$; as its vertex lies left of that
edge, $R_2>0$ across the annulus. For $m=3$ the inner (saddle) edge is
also explicit, $U_s=\bigl(\sqrt{4c/a^{2}-3}-1\bigr)a/2$ for $a^{2}<c<c^*=3a^{2}$,
and one finds $R_3(U_s)=-\tfrac32\,a^{5}(\tilde c-1)(4\tilde c-3)\bigl(\sqrt{4\tilde c-3}-3\bigr)>0$
with $\tilde c=c/a^{2}$, while $R_3(0)=2a^{5}(\tilde c-2)(\tilde c-1)\ge0$ when
$c\le a^{2}$.  In fact $R_3>0$ on the whole annulus: here
\[
\begin{aligned}
R_3(U)={}&5U^5+20aU^4+10(5a^2-2c)U^3+20a(2a^2-c)U^2\\
&+(2a^2-c)(8a^2-3c)U+2a(a^2-c)(2a^2-c),
\end{aligned}
\]
whose coefficients are all nonnegative when $c\le a^2$, so $R_3>0$
for $U>0$; when $a^2<c<3a^2$, writing $c=a^2+aU_s+U_s^2$ and
$U=U_s+t$ gives
\[
\begin{aligned}
R_3={}&5t^5+5(4a+5U_s)\,t^4+30(a+U_s)^2\,t^3\\
&+10(a+U_s)(2a^2+5aU_s-U_s^2)\,t^2\\
&+(a+U_s)(a+2U_s)(5a^2+17aU_s-16U_s^2)\,t\\
&+3U_s(a-U_s)(a+U_s)(a+2U_s)^2,
\end{aligned}
\]
whose coefficients are all positive for $0<U_s<a$, so $R_3>0$ for
$U\ge U_s$. Hence, for $m = 2, 3$, $\chi'' > 0$ on the annulus, including
its continuous extension through the center, and Chicone's criterion
gives $T'(E)>0$.
Direct quadrature of $T(E)=\sqrt2\!\int dv/\sqrt{E-V(v)}$ confirms the
strict monotonicity (Supplement, Figure~S2, for K$(2,2)$; the other
cases were checked the same way).  For $m = 4, 5$ the explicit $R_m$ (Appendix~\ref{app:symbolic}) is
found positive on the annulus numerically, and the same holds for
$6 \le m \le 12$; no proof is given for $m \ge 4$.
\end{proof}

These two families, separated by the critical speed $c = c^*$ at
which the solitary-wave branch terminates
(Remark~\ref{rem:algebraic-soliton}), organize into the
classification shown in Table~\ref{tab:wave-taxonomy},
Figure~\ref{fig:wave-taxonomy}, and the phase portraits of
Figure~\ref{fig:phase-portraits} and Supplement, Figure~S1.  The bifurcation between
these regimes in the $(a, c)$ parameter plane is shown in
the bifurcation diagram (Supplement, Figure~S3).

\begin{table}[ht]
\centering
\caption{Traveling wave taxonomy for K(m,m) on background shelf 
$u = a >0$.  $^\dagger$Boundary regularity depends on~$m$:
$C^1$ for $m=2$, $C^0$ for $m=3$, $C^{0,\alpha}$ with
$\alpha = 2/(m-1)$ for $m > 3$
(Table~\ref{tab:regularity}).  ``Cnoidal'' denotes periodic waves;
for K(m,m) they are not Jacobi $\mathrm{cn}^2$ profiles
(\S\ref{subsec:k22-explicit}).  At $c = c^*$ the solitary-wave branch
terminates and no wave of either type exists
(Remark~\ref{rem:algebraic-soliton}).}
\label{tab:wave-taxonomy}
\begin{tabular}{@{} l c c c @{}}
\toprule
\textbf{Regime} & \textbf{Wave type} &
  \textbf{Tail behavior} & \textbf{Regularity} \\
\midrule
$a = 0$, $c > 0$ & Compacton & Compact support
  & $C^1$--$C^{0,\alpha}$$^\dagger$ \\
$a > 0$, $c > c^*$ & Solitary wave & Exponential decay
  & $C^\infty$ \\

$a > 0$, $c < c^*$ & Cnoidal wave & Periodic
  & $C^\infty$ \\
\bottomrule
\end{tabular}
\end{table}

\FloatBarrier
\begin{figure}[!htbp]
\centering
\includegraphics[width=\textwidth]{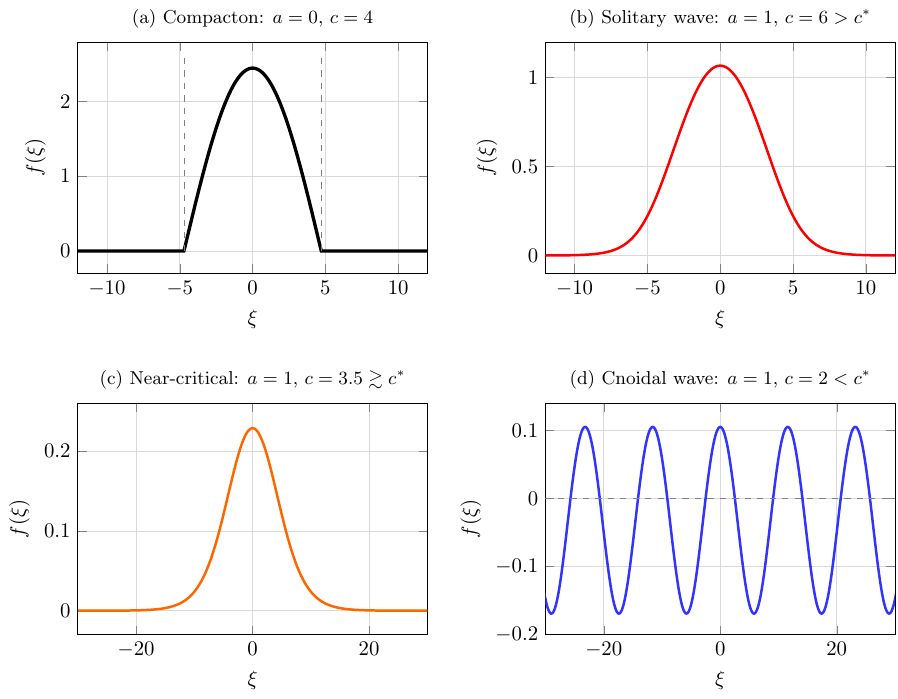}
\caption{Computed K(3,3) traveling wave profiles illustrating
the regimes of Table~\ref{tab:wave-taxonomy}.
The perturbation $f(\xi) = U(\xi) - a$ is plotted against
the co-moving coordinate $\xi = x - ct$.
(a)~Compacton ($a = 0$, $c = 4$): the exact
$U = \sqrt{3c/2}\,\cos(\xi/3)$
profile with compact support (dashed lines mark the boundary
$|\xi| = 3\pi/2$); the nonzero slope $U'(\pm 3\pi/2) \neq 0$
produces the visible corners characteristic of $C^0$ regularity.
(b)~Solitary wave ($a = 1$, $c = 6 \gg c^* = 3$): smooth pulse
with exponential decay rate $\kappa = \sqrt{c/(3a^2) - 1} = 1$
and completely rounded profile---the shelf eliminates the corner.
(c)~Near-critical solitary wave ($a = 1$, $c = 3.5$,
$\kappa \approx 0.41$): as $c \to {c^*}^+$ the pulse broadens
($\kappa \to 0$) and its amplitude falls to zero, so the solitary
branch ends at $c = c^*$ (Remark~\ref{rem:algebraic-soliton}).
(d)~Cnoidal wave ($a = 1$, $c = 2 < c^*$): periodic orbit
around the center equilibrium, with anharmonic profile
reflecting the nonlinear potential.
All profiles obtained by numerical integration of the
traveling wave ODE in the $v$-variable ($v = U^3$).}
\label{fig:wave-taxonomy}
\end{figure}

\begin{remark}[Structural instability of compactons]
\label{rem:structural-instability}
Any positive perturbation of the background, however small,
transforms the compacton into a $C^\infty$ solitary wave with
exponential tails.  The compact-support wave exists only at the
singular endpoint $a = 0$ of the parameter family.  This contrasts
with KdV solitons, which are smoothly embedded in the cnoidal wave
family: the sech$^2$ profile deforms continuously into Jacobi
elliptic cnoidal waves as the modulus varies.  For K(m,m), the
analogous deformation involves a discontinuous change in both
regularity and support structure.
\end{remark}

\begin{remark}[The critical speed $c = c^*$]\label{rem:algebraic-soliton}
At $c = c^* = m\,a^{m-1}$ the center and saddle equilibria of $V$
coalesce in a saddle-node bifurcation: $V'(a^m) = V''(a^m) = 0$ while
$V'''(a^m) = (m-1)/(m\,a^m) > 0$, so $v_0 = a^m$ is a stationary
inflection point of~$V$.  The homoclinic orbit does not survive.  As
$c \to {c^*}^+$ the solitary wave broadens ($\kappa \to 0$) and
flattens, its amplitude vanishing linearly,
\[
f_{\max} \sim \frac{3a\,(c - c^*)}{(m-1)\,c^*},
\]
which for K(2,2) is the onset behavior of~\eqref{eq:amplitude-speed}.
At $c = c^*$ itself $V$ is strictly increasing for $v > 0$, so an
orbit with $a + f > 0$ that leaves the shelf cannot return to it.
The level set $E = V(a^m)$ carries only a one-sided front of
depression, which approaches the shelf algebraically,
$f(\xi) \sim -12a/\bigl((m-1)\,\xi^2\bigr)$, with exponent $-2$ for
every~$m$, and reaches $u = 0$ at finite~$\xi$ with infinite slope.
The critical curve $c = c^*(a)$ is therefore the codimension-one
boundary between the solitary wave regime ($c > c^*$, exponential
tails) and the periodic regime ($c < c^*$, center oscillations); no
solitary wave lies on it.
\end{remark}

\begin{figure}[!htbp]
\centering
\includegraphics[width=\textwidth]{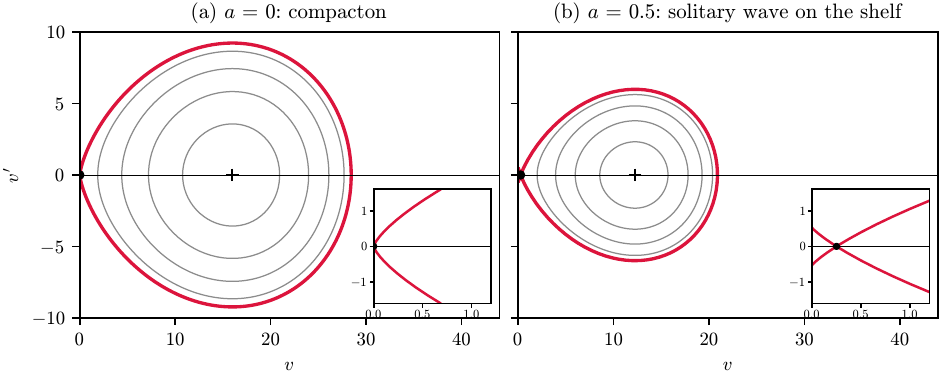}
\caption{Phase portraits of the K(2,2) traveling wave
ODE~\eqref{eq:v-ode} in the $(v, v')$ plane at speed $c = 4$,
without and with a background shelf.  Curves are level sets of the
energy~\eqref{eq:shelf-energy}; the thick curve is the level through
the shelf equilibrium $v_0 = a^2$ (dot), and $+$ marks the center.
(a)~$a = 0$: the shelf equilibrium at the origin is degenerate, and
the compacton orbit reaches it at finite~$\xi$ with a vertical
tangent (inset), which is why the support is compact.
(b)~$a = 0.5$, so that $c > c^* = 2a$: the shelf equilibrium is a
saddle, and the homoclinic orbit approaches it only as
$|\xi| \to \infty$, along its stable and unstable directions (inset),
which gives the exponential tails of the shelf solitary wave.  As
$a \to 0^+$ the saddle moves into the origin and the loop in~(b)
deforms into the compacton orbit in~(a)
(\S\ref{subsec:singular-limit}).}
\label{fig:phase-portraits}
\end{figure}

\subsection{Exact reduction for K(2,2)}
\label{sec:shelf-k22}
\label{subsec:k22-explicit}

The simplest member of the family, K(2,2):
$u_t + (u^2)_x + (u^2)_{xxx} = 0$,
admits the most complete analytical reduction within the family,
illustrating the general theory while revealing the limits of
closed-form analysis.
Setting $w = a + f$ (the total field) and $m = 2$, the
potential~\eqref{eq:potential} becomes
\begin{equation}\label{eq:V-k22}
V(v) = \frac{1}{2}\,v^2 - \frac{2c}{3}\,v^{3/2}
  + (ca - a^2)\,v,
\end{equation}
with equilibrium at $v_0 = a^2$ and critical speed $c^* = 2a$.

\subsubsection*{First integral and factored form}

An alternative approach works directly with the total-field
variable $w = \sqrt{v}$ (equivalently, $v = w^2$).%
\footnote{The symbol~$w$ is reused here for the total-field
variable; it should not be confused with the debris field~$w(x,t)$
of \S\ref{subsec:quasi-elastic}.}
The K(2,2)
traveling wave ODE in the $w$-variable reads
\begin{equation}\label{eq:w-ode}
-cw + w^2 + 2w\,w'' + 2(w')^2 + (ca - a^2) = 0.
\end{equation}
Treating $(w')^2$ as a dependent variable and solving the
resulting linear ODE yields a first integral
\begin{equation}\label{eq:first-integral}
(w')^2 = \frac{cw}{3} - \frac{w^2}{4}
  - \frac{ca - a^2}{2} + \frac{D}{w^2},
\end{equation}
where the constant $D = ca^3/6 - a^4/4$ is fixed by the boundary
conditions $w \to a$, $w' \to 0$ as $|\xi| \to \infty$.

For K(2,2), the compacton limit can be verified explicitly.
Setting $a = 0$ in the first integral~\eqref{eq:first-integral}
gives $D = 0$, so $(w')^2 = cw/3 - w^2/4$, with solution
\begin{equation}\label{eq:compacton-recovery}
w_0(\xi) = \frac{4c}{3}\cos^2\!\left(\frac{\xi}{4}\right),
\qquad |\xi| \leq 2\pi,
\end{equation}
recovering the standard K(2,2)
compacton~\cite{rosenau1993compactons}.

For $a > 0$, however, the first integral does not yield
an explicit profile $w(\xi)$.  Multiplying~\eqref{eq:first-integral} by $4w^2$ and factoring the
resulting quartic polynomial produces the central identity of this
section.  Since $w = a$ is a double root (reflecting the boundary
condition $w \to a$, $w' \to 0$), the quartic takes the form
\begin{equation}\label{eq:factored-form}
\boxed{4w^2(w')^2 = (w - a)^2\,(w_+ - w)\,(w - w_-),}
\end{equation}
where $w_\pm$ are the roots of the residual quadratic
\begin{equation}\label{eq:residual-quadratic}
Q(w) = w^2 + \Bigl(2a - \frac{4c}{3}\Bigr)\,w
  + \Bigl(a^2 - \frac{2ca}{3}\Bigr).
\end{equation}
The discriminant of $Q$ simplifies to
\begin{equation}\label{eq:disc}
\Delta_Q = \frac{8c}{9}\,(2c - 3a),
\end{equation}
which is positive for all $c > c^* = 2a > 0$.  The two roots are
\begin{equation}\label{eq:wpm}
w_\pm = \frac{2c}{3} - a \pm
  \frac{1}{3}\sqrt{2c(2c - 3a)}\,,
\end{equation}
with $w_+ > a$ giving the separatrix crest (the maximum value
attained on the homoclinic orbit) and $w_- < 0$ lying in
the unphysical half-plane.  The double root at $w = a$ reflects the
saddle-point structure at the shelf equilibrium, and the orbit
oscillates between $w = a$ and $w = w_+$.

\subsubsection*{Solitary wave properties}

From the factored form~\eqref{eq:factored-form}, the solitary wave
profile has the following properties.

The peak excursion above the shelf is
\begin{equation}\label{eq:amplitude-speed}
\boxed{f_{\max} = w_+ - a
  = \frac{2c}{3} - 2a + \frac{1}{3}\sqrt{2c(2c - 3a)}\,.}
\end{equation}
In the $v$-variable, the crest value is $v_1 = w_+^2$, which lies
\emph{above} the second equilibrium $v_c = (c - a)^2$: the
separatrix overshoots the center before returning to the saddle.
In the compacton limit $a \to 0$, the formula reduces to
$f_{\max} \to 4c/3$, recovering the known K(2,2) compacton
amplitude.

The exponential decay rate, determined by linearization of the
Hamiltonian system about the saddle, is
\begin{equation}\label{eq:decay-rate}
\kappa = \sqrt{\frac{c - 2a}{2a}},
\end{equation}
so that $f(\xi) \approx A\,e^{-\kappa|\xi|}$ for large
$|\xi|$, where $A$ depends on $a$ and $c$.  As $a \to 0^+$,
the decay rate $\kappa \to \infty$:
the exponential tails steepen without bound, approaching the
vertical cutoff of the compacton.

\subsubsection*{Non-sech\texorpdfstring{$^2$}{2} profile and quadrature}

Testing any sech$^2$ ansatz against~\eqref{eq:factored-form} reveals
an obstruction: the right-hand side factors as
$(w - a)^2(w_+ - w)(w - w_-)$ with distinct linear terms rather
than a perfect square.  The exact profile is therefore
\emph{not} sech$^2$ and must be obtained by implicit quadrature:
\begin{equation}\label{eq:quadrature}
\xi = \pm \int_{w(\xi)}^{w_+}
  \frac{2w'\,dw'}{(w' - a)\sqrt{(w_+ - w')(w' - w_-)}}\,.
\end{equation}
This non-standard character is a direct consequence of the
nonlinear dispersion in K(2,2), which produces a $v^{3/2}$ term in
the potential rather than the cubic responsible for KdV's sech$^2$
profiles.  The integral~\eqref{eq:quadrature} is nevertheless
elementary.  Splitting $2w/(w - a) = 2 + 2a/(w - a)$ and writing
$s = w - a$, $A_+ = w_+ - a$ and $A_- = a - w_-$ gives
\[
\begin{aligned}
\xi &= 2\Bigl[\frac{\pi}{2}
  - \arcsin\frac{2s - A_+ + A_-}{A_+ + A_-}\Bigr]\\
&\quad + \frac{1}{\kappa}\,
 \ln\frac{2A_+A_- + (A_+ - A_-)\,s
          + 2\sqrt{A_+A_-}\,\sqrt{(A_+ - s)(s + A_-)}}
         {s\,(A_+ + A_-)},
\end{aligned}
\]
where the coefficient of the logarithm is exactly $1/\kappa$ because
$A_+A_- = 2a(c - 2a)$.  The profile is therefore explicit as the
inverse of an elementary function, and the limit $s \to 0$ recovers the
decay rate~\eqref{eq:decay-rate} directly.

\subsubsection*{Periodic waves}

For $c < c^* = 2a$, the equilibrium $v_0 = a^2$ is a center,
and the phase portrait contains a family of nested periodic orbits
with small-amplitude period
$T_0 = 2\pi/\sqrt{1 - c/(2a)}$,
which diverges as $c \to (2a)^-$.  In the $w$-variable
the first integral is a quartic polynomial $P_4(w)$, so the cnoidal
profiles are given implicitly by elliptic integrals.  They are not
the Jacobi $\mathrm{cn}^2$ profiles of KdV: the $v^{3/2}$ term in the
potential makes the integrand $w\,dw/\sqrt{P_4(w)}$ rather than
$dw/\sqrt{P_4(w)}$.  Detailed period-energy analysis appears in
Appendix~\ref{app:periodic-detail}.

\subsection{The singular limit: from shelf waves to compactons}
\label{subsec:singular-limit}

The most revealing feature of the shelf-wave family is the
behavior as $a \to 0^+$.  The background vanishes, the smooth
Hamiltonian system degenerates, and the compacton appears as a
singular endpoint of the one-parameter family of solitary waves.
The following theorem makes this precise.

\begin{theorem}[Compacton as singular limit]
\label{thm:singular-limit}
Let $w_a(\xi) = a + f_a(\xi)$ denote the K(m,m) shelf solitary wave
of speed $c > 0$ on background $a > 0$, and let $f_0$ denote the
K(m,m) compacton of speed $c$ (on zero background).  Then:
\begin{enumerate}
\item[\textnormal{(a)}]
  $w_a - a \to f_0$ in $L^2(\mathbb{R})$ as $a \to 0^+$.
\item[\textnormal{(b)}]
  Convergence cannot hold in $C^k$ for $k$ exceeding the
  compacton's boundary regularity: each $f_a$ is
  $C^\infty(\mathbb{R})$, while $f_0$ has the regularity of the
  K(m,m) compacton profile: $C^1$ for $m = 2$, $C^0$ for $m = 3$,
  and $C^{0,\alpha}$ with $\alpha = 2/(m-1)$ for $m > 3$
  (cf.\ Table~\ref{tab:regularity}).
\item[\textnormal{(c)}]
  $\operatorname{supp}(f_a) = \mathbb{R}$ for every $a > 0$
  (exponential tails), but $\operatorname{supp}(f_0)$ is a compact
  interval.
\item[\textnormal{(d)}]
  The decay rate satisfies
  $\kappa(a) = \sqrt{(c - m\,a^{m-1})/(m\,a^{m-1})} \to \infty$
  as $a \to 0^+$.
\end{enumerate}
\end{theorem}

\begin{proof}
(a) Fix $\varepsilon > 0$ and set $R_\varepsilon = \pi\ell_0/2 +
\varepsilon$, where $\ell_0 = 2m/(m-1)$ is the compacton
half-width parameter.  We split the $L^2$ norm and bound
each piece.

\emph{Step~1 (Uniform bound).}
The potential $V_a(v)$ depends continuously on~$a$, and for large
$v$ is dominated by the positive term $\frac{1}{2}v^2$, which
drives $V_a(v) \to \infty$.  Consequently the separatrix crest
$v_1(a)$ (the largest root of $V_a(v) = V_a(a^m)$) varies
continuously and remains uniformly bounded: $v_1(a) \leq M$
for all sufficiently small~$a$.  Since $f_a = v^{1/m} - a$ is
controlled by $v_1(a)^{1/m}$, the family $\{f_a\}$ satisfies
$\|f_a\|_{L^\infty} \leq C$ independently of~$a$.

\emph{Step~2 (Interior convergence).}
Fix $\eta \in (0, \pi\ell_0/2)$.  On $[-(\pi\ell_0/2-\eta),\,
\pi\ell_0/2-\eta]$ the limit orbit stays in $v > 0$, where the vector
field is smooth, so $f_a \to f_0$ uniformly there by smooth
dependence on~$a$ (the separatrix crest $v_1(a)$ converges by
continuity).  On $\pi\ell_0/2-\eta \le |\xi| \le R_\varepsilon$,
evenness and monotonicity of $f_a$ in $|\xi|$ give
$0 \le f_a(\xi) \le f_a(\pi\ell_0/2-\eta) \to f_0(\pi\ell_0/2-\eta)$,
which is small for small~$\eta$.  Since $|f_a - f_0| \le 2C$,
dominated convergence and $\eta \to 0$ give
$\int_{|\xi|\le R_\varepsilon}|f_a - f_0|^2\,d\xi \to 0$; in
particular $f_a(R_\varepsilon) \to 0$.

\emph{Step~3 (Tail decay).}
For $|\xi| > R_\varepsilon$ the compacton vanishes, so it suffices to
show $\int_{|\xi|>R_\varepsilon} f_a^2\,d\xi \to 0$.  Write
$v_R = (a + f_a(R_\varepsilon))^m$ and $v_0 = a^m$.  On the tail $v$
decreases from $v_R$ to $v_0$, and the first integral gives
$|v'| = \rho(v)\,(v - v_0)$ with
$\rho(v)^2 = 2\,[V_a(v_0) - V_a(v)]/(v - v_0)^2$.  Let
$\rho(a) := \min_{v \in [v_0, v_R]} \rho(v) > 0$.  Then
$|v'| \ge \rho(a)(v - v_0)$, so
\[
v(\xi) - v_0 \;\le\; (v_R - v_0)\, e^{-\rho(a)(|\xi| - R_\varepsilon)},
\qquad |\xi| \ge R_\varepsilon .
\]
(The linearized rate $\kappa(a)$ of part~(d) is the limit of $\rho$
as $v \to v_0^+$, not a valid rate from a finite matching point.)
Superadditivity, $(a+f)^m \ge a^m + f^m$, gives
$f_a \le (v - v_0)^{1/m}$, hence
\[
\int_{|\xi| > R_\varepsilon} f_a^2\,d\xi
\;\le\; \frac{m\,(v_R - v_0)^{2/m}}{\rho(a)} .
\]
By the mean value theorem, for each $v \in (v_0, v_R]$ there is
$\tilde v \in (v_0, v)$ with $\rho(v)^2 = -V_a''(\tilde v)$; since
$V_a''(w) = 1 - (c/m)\,w^{1/m-1}$ is increasing in~$w$,
$\rho(v)^2 \ge -V_a''(v_R) = (c/m)\,v_R^{(1-m)/m} - 1$, and the same
bound holds for $\rho(a)$.  As $a \to 0^+$, $v_R \to 0$ by Step~2, so
$\rho(a) \to \infty$ while $v_R - v_0 \to 0$, and the tail estimate
follows.
Combining Steps~2 and~3 yields
$\|f_a - f_0\|_{L^2} \to 0$.

(b) For $a > 0$, the profile $f_a$ is analytic because the
Hamiltonian flow has $v > a^m > 0$.  The compacton $f_0$ has
reduced regularity at the support boundary
(\S\ref{subsec:regularity}): $C^1$ for $m=2$, $C^0$ for $m=3$,
and $C^{0,\alpha}$ for $m > 3$.  This regularity gap cannot be
bridged by convergence in any $C^k$ norm exceeding the limit
regularity.

(c) Immediate from the exponential versus compact support structure.

(d) The decay rate $\kappa(a) = \sqrt{c/(m\,a^{m-1}) - 1}$
diverges as $a \to 0^+$ since $m\,a^{m-1} \to 0$.
\end{proof}

For K(2,2), setting $a = 0$ recovers the standard
compacton~\eqref{eq:compacton-recovery}.  The regularity drops from
$C^\infty$ (for any $a > 0$) to $C^1$ ($m - 1 = 1$ for K(2,2)):
the second derivative $w_0''$ has a jump discontinuity at
$\xi = \pm 2\pi$.  This regularity drop is
codimension-one in the $(a, c)$ parameter space, occurring along
the boundary $a = 0$ where the dispersion coefficient
$(a+f)^{m-1}$ degenerates to zero at
$u = 0$.  The resulting compacton regularity degrades with
increasing $m$: $C^1$ for K(2,2), $C^0$ for K(3,3), and
$C^{0,\alpha}$ with $\alpha = 2/(m-1)$ for $m > 3$
(Table~\ref{tab:regularity}).

The singular limit theorem identifies a \emph{double singularity}
at $a = 0$ (the period divergence of the periodic family is a separate
limit, $E \to E_{\max}$ at fixed $a$ with $c > a^{m-1}$):
regularity is lost ($C^\infty$ to the compacton's boundary class),
and topology changes
(exponential tails become compact support).  These two phenomena
occur simultaneously through a single mechanism: the vanishing
of the dispersion coefficient at the background level.
Figure~\ref{fig:singular-limit} illustrates the progressive
steepening of the wave tails as $a \to 0^+$ for K(2,2).

\FloatBarrier
\begin{figure}[!htbp]
\centering
\includegraphics{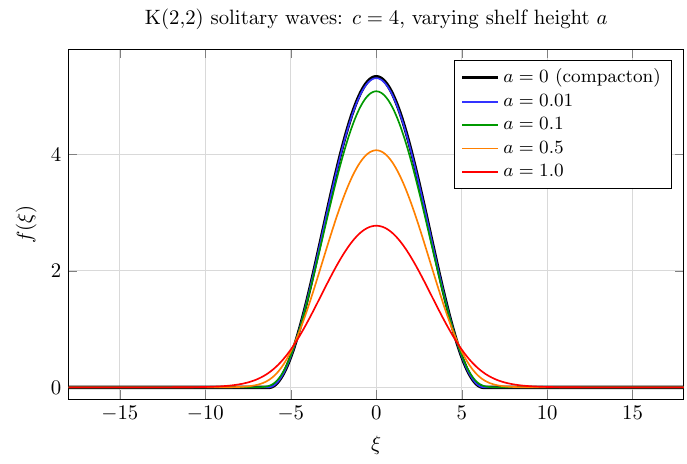}
\caption{Singular limit sequence for K(2,2) shelf solitary waves
with fixed speed $c = 4$.  As the shelf height $a$ decreases
toward zero, the smooth solitary waves (colored curves, obtained
by numerical integration of the factored first integral~\eqref{eq:factored-form})
converge to the compacton (black curve,
$f_0 = \tfrac{4c}{3}\cos^2(\xi/4)$).  The amplitude
$f_{\max} = w_+ - a$ increases toward $4c/3$ while the decay rate
$\kappa = \sqrt{(c-2a)/(2a)}$ diverges, steepening the
tails toward the vertical cutoff of the compacton.
The transition involves simultaneous loss of
analyticity ($C^\infty \to C^1$) and compactification of
support.  These exact profiles confirm the non-sech$^2$
character established in \S\ref{sec:shelf-k22}.}
\label{fig:singular-limit}
\end{figure}

\subsection{Periodic waves and the legacy of Lax (1975)}
\label{subsec:lax-periodic-connection}

In one of the foundational papers of integrable systems theory,
Lax~\cite{lax1975periodic} constructed periodic solutions of the
KdV equation through a variational principle: minimize the
conserved functional $F_{N+1}$ subject to the constraints
$F_j = A_j$ for $j = -1, 0, \ldots, N-1$.  The solutions lie on
$N$-dimensional invariant tori and are almost periodic in time.
The simplest case $N = 1$ recovers the classical cnoidal waves.
The 1976 companion~\cite{lax1976almost} reviewed the construction and
extended it to minimax critical points, and in its appendix the
present author computed the $N = 2$ case numerically, observing that
the solution resumed its initial shape, in a shifted position, after
a finite time.\footnote{I had just arrived at the Courant Institute as a first-year
graduate student when Peter Lax asked
me to compute the almost periodic cnoidal waves of the KdV equation.
The experience had a profound impact on my career and, eventually,
led to this paper.}

The variational characterization extends formally to K(m,m): traveling
waves satisfy $\delta\mathcal{H}/\delta u = c\,\delta P/\delta u + \kappa$
with $\kappa \in \ker\mathfrak{D}$ (\S\ref{subsec:hamiltonian-geometry}), which is
the $N = 1$ case of Lax's constrained optimization.  The
fundamental difference lies in the dimension of the resulting
solution families.

\subsubsection*{Finite conservation and one-parameter families}

Lax's construction of $N$-dimensional tori requires $N + 1$
functionally independent conserved functionals in involution.  For
KdV, the infinite hierarchy $\{F_j\}_{j=-1}^{\infty}$ provides
this for arbitrary $N$, and Lax's Theorem~5.6
of~\cite{lax1975periodic} establishes that the resulting solution
set is an $N$-dimensional manifold.

For K(m,m) with $m \geq 2$, the mechanism behind Lax's construction
is absent.  The equation carries five known conservation laws
(Theorem~\ref{thm:finite-hierarchy}), but, lacking a recursion
operator, they do not extend to an infinite hierarchy, and although
the five are in involution (\S\ref{subsec:conservation-laws}), three
of them are Casimirs of $\mathfrak{D}$, which generate no flow.  The two remaining
functionals, $\mathcal{H}$ and $P$, generate the time flow and
spatial translation, and their relative equilibria,
$\delta\mathcal{H}/\delta u = c\,\delta P/\delta u + \kappa$, are the
traveling waves: precisely Lax's $N = 1$ case.  The construction
therefore yields one-parameter families of periodic orbits,
parametrized by a single energy value.  No
two-parameter families arise from it, no invariant tori of dimension
greater than one are furnished by this mechanism, and the angle-action variables of
the integrable theory have no counterpart.

Table~\ref{tab:lax-comparison} collects the parallel and divergent
features.

\begin{table}[ht]
\centering
\small
\caption{Periodic wave structure: KdV (Lax, 1975) versus
K(m,m).  $T_0$ is the small-amplitude period and $T_{\max}$ the
largest period in the family (Theorem~\ref{thm:existence}).}
\label{tab:lax-comparison}
\begin{tabular}{@{} l l l @{}}
\toprule
\textbf{Feature} & \textbf{KdV} & \textbf{K(m,m), $m \geq 2$} \\
\midrule
Conservation laws & Infinitely many $F_j$
  & Five, in involution; three Casimirs \\
Periodic solutions & $N$-parameter families
  & One-parameter families \\
Solution manifold & $N$-dimensional tori
  & One-dimensional curves \\
Invariant tori & Yes (commuting flows)
  & No \\
Spectral characterization & Finite-gap potentials
  & None \\
Cnoidal wave formula & Jacobi elliptic $\mathrm{cn}^2$
  & Elliptic, not $\mathrm{cn}^2$ \\
Period range & $T_0 \to \infty$
  & $T_0 \to T_{\max}$ ($\infty$ only for $c > a^{m-1}$) \\
\bottomrule
\end{tabular}
\end{table}

\subsubsection*{What survives without integrability}

The traveling wave ODE is Hamiltonian for K(m,m) just as for KdV,
generating one-parameter periodic families by the standard center
theorem with variational characterization as constrained extremals
of $H$ subject to fixed $P$.  The period function $T(E)$ increases
monotonically from $T_0$ to $T_{\max}$, with $T_{\max} = \infty$ only
for $c > a^{m-1}$
(Proposition~\ref{prop:period-monotone}, for $m = 2, 3$;
Supplement, Figure~S2).

What does not survive is the multi-dimensional structure:
without the infinite hierarchy in involution, there is no
infinite family of commuting flows, no angle-action coordinates, no spectral
characterization through finite-gap potentials.  The KdV
periodic theory lives on finite-dimensional invariant tori;
the K(m,m) periodic theory lives on one-dimensional energy
curves.

\subsubsection*{The compacton as singular endpoint}

In Lax's theory, the KdV soliton appears as the
$T \to \infty$ limit of cnoidal waves: the elliptic modulus
$k \to 1$, but the solution remains analytic throughout.
For K(m,m), the corresponding limit is far more
singular: regularity loss ($C^\infty$ to the
compacton's boundary class; see Table~\ref{tab:regularity}),
and support compactification occur simultaneously
(Theorem~\ref{thm:singular-limit}).  This double singularity,
which has no analogue in Lax's integrable theory, arises from the
degenerate dispersion coefficient $u^{m-1}$ vanishing at $u = 0$.
The compacton is not merely a limiting wave profile but a point
where the Hamiltonian structure of the traveling wave ODE undergoes
qualitative degeneration.

\section{Concluding remarks}
\label{sec:conclusion}

Lax's 1968 discovery that operator compatibility conditions can
unlock the complete analytical structure of nonlinear dispersive
equations~\cite{lax1968integrals,gardner1967method,ablowitz1981solitons}
naturally raises the question of what structure persists when that
compatibility fails.

The K(m,m) compacton family provides a concrete answer.  These
equations fail every test of complete integrability
(\S\ref{sec:nonintegrability}): no Lax pair exists, the Painlev\'{e}
property fails, the symmetry algebra is finite-dimensional, and the
conservation hierarchy terminates.  Yet substantial mathematical
organization persists.  Universal algebraic formulas govern all
compacton parameters through a single
exponent (Theorem~\ref{thm:universal}).  These formulas yield
consequences that had not been previously identified:
the trigonometric conservation integrals vanish on every
K(m,m) compacton (Proposition~\ref{prop:trig-vanishing}), the
cosine-power ansatz is exclusive to the diagonal $m = n$
(Proposition~\ref{prop:exclusivity}), and the mass and translation
generator increase along the compacton family for all $m > 1$
(Corollary~\ref{cor:VK}).  Five known independent conservation
laws (exactly four local and one nonlocal) constrain the dynamics
in ways that are both strong enough
to produce quasi-elastic collisions and too weak to enforce
perfect elasticity.  The constraint-deficit analysis
(Theorem~\ref{thm:constraint-deficit}) shows that while these
laws constrain emergent amplitudes and force the fundamental
Fourier mode of collision debris to vanish, all higher radiation
modes remain unconstrained.  Recent high-resolution numerical
studies~\cite{garralon2023compacton,garralon2024numerical} have
independently confirmed and extended this quasi-elastic picture,
revealing additional structure, including odd-even asymmetry in
collision robustness, that the analytical framework accommodates
but does not yet fully explain.  The analysis generates testable
predictions, including a $k = 1$ notch in the debris spectrum
and a spectral peak shift with increasing~$m$, both of which are
confirmed numerically across $m = 2, 3, 4, 5$
(Figure~\ref{fig:debris-spectrum}).  Hamiltonian
geometry provides variational principles and stability theory
despite the absence of a bi-Hamiltonian pair.

The shelf-wave analysis of \S\ref{sec:shelf-waves} adds
a further dimension: traveling waves on a positive background
$u = a > 0$ organize into a phase-plane classification governed by
the critical speed $c^* = m\,a^{m-1}$, with smooth solitary waves
for $c > c^*$ and one-parameter cnoidal families for $c < c^*$.
The compacton emerges as a singular limit $a \to 0^+$ in which
two structures collapse simultaneously:
regularity drops ($C^\infty$ to the compacton's boundary class),
and the support
compactifies.  This double singularity has
no analogue in Lax's integrable periodic theory.  The
finite conservation hierarchy limits the periodic waves to
one-parameter families on energy curves, a
stark reduction from the $N$-dimensional tori of the
KdV theory~\cite{lax1975periodic,lax1976almost}.

This is not residual integrability; it is a different organizing
principle for the K(m,m) family, one in which scaling universality and finite conservation
structure replace the spectral machinery of the IST.

\paragraph{Implications for future research.}
The constraint-deficit framework yields concrete predictions, two of
which have been confirmed numerically.
The predicted spectral notch at wavenumber $k = 1$ in collision
debris and the systematic peak shift with increasing~$m$ are both
verified across all tested values
(Figure~\ref{fig:debris-spectrum}, Table~\ref{tab:multi-m-results}).
The BBM collision experiments
(\S\ref{subsec:landscape}) found BBM collisions about as elastic as
K(m,m) collisions by the amplitude measure, with more debris energy,
leaving open whether interaction geometry modulates the role of
conservation count (Conjecture~\ref{conj:monotonicity}).  Beyond K(m,m),
the methods employed here (pattern recognition across
parameter families, conservation counting, constraint
analysis of collision outcomes, and phase-plane uniqueness
arguments) could in principle apply to other systems with
finite conservation structure, though each case would require
independent verification.

Most physically relevant nonlinear dispersive systems are not
completely integrable.  Whether the K(m,m) family is representative
of a broader class remains an open question; the territory between integrable order and generic chaos
may be richer than the integrable/non-integrable dichotomy
suggests, but this paper establishes structure only for K(m,m).
The BBM collision study reported here provided one such test;
extending these
comparisons to additional equation families with known
conservation counts, and developing the stability theory for
singular compacton profiles, are natural next steps that would
clarify the scope of these ideas.  The eight problems posed in
Appendix~\ref{app:open-problems}
outline a research program aimed at understanding mathematical
structure beyond the integrable paradigm.

Lax showed what structure looks like when integrability succeeds.
The K(m,m) equations show what structure looks like when it fails.
Between these poles lies a vast landscape of nonlinear dispersive
systems, most of which are not completely integrable, where
the question is not whether structure exists but what form it
takes.  The constraint-deficit framework, the conservation-elasticity
landscape, and the singular limit analysis developed here offer
partial answers for one family of equations.  Extending these
ideas to the broader territory remains the central challenge.

Peter Lax approached mathematical problems with a conviction that
the right viewpoint would reveal hidden order, and that computation
and theory were partners in this search.  The work presented here,
which combines numerical experiment with analytical proof to uncover
organizing principles in a non-integrable system, is offered in the
same spirit.  The search for structure in nonlinear phenomena,
wherever it may be found, continues the tradition that Lax
established.

\appendix

\section{Trigonometric orthogonality of K(m,m) compactons}
\label{app:jacobian}

This appendix provides the detailed calculation underlying
Proposition~\ref{prop:trig-vanishing} and the constraint partition
in Theorem~\ref{thm:constraint-deficit}.

\subsection*{A.1\quad The coefficient $B_\alpha$}

For a K(m,m) compacton $U(x) = A\cos^\alpha\!\bigl((x - x_0)/\ell\bigr)$
supported on $|x - x_0| \leq \pi\ell/2$, the trigonometric
conservation integral $I_2 = \int U(x)\sin(x)\,dx$ is evaluated by
the substitution $\theta = (x - x_0)/\ell$, giving
\begin{equation}\label{eq:I2-expansion}
I_2 = A\ell\bigl[\sin(x_0)\,B_\alpha + \cos(x_0)\,S_\alpha\bigr],
\end{equation}
where
\begin{align}
B_\alpha &= \int_{-\pi/2}^{\pi/2}
  \cos^\alpha(\theta)\cos(\ell\theta)\,d\theta,
  \label{eq:B-alpha-def} \\
S_\alpha &= \int_{-\pi/2}^{\pi/2}
  \cos^\alpha(\theta)\sin(\ell\theta)\,d\theta.
  \label{eq:S-alpha-def}
\end{align}
By parity, $S_\alpha$ vanishes: the integrand
$\cos^\alpha(\theta)\sin(\ell\theta)$ is odd on the symmetric
domain $[-\pi/2, \pi/2]$ for every $\alpha > 0$ and every
$\ell > 0$.

\subsection*{A.2\quad Proof that $B_\alpha = 0$}

Central to the proof is the identity linking the K(m,m) compacton
parameters:
\begin{equation}\label{eq:ell-alpha-identity}
\ell - \alpha = 2 \qquad \text{for all } m > 1,
\end{equation}
which follows directly from $\alpha = 2/(m-1)$ and
$\ell = 2m/(m-1)$: their difference is
$\ell - \alpha = 2m/(m-1) - 2/(m-1) = 2(m-1)/(m-1) = 2$.

The Fourier cosine transform of $\cos^\alpha(t)$ on
$[-\pi/2, \pi/2]$ is given by~\cite{olver1993applications}
\begin{equation}\label{eq:fourier-cosine-formula}
\int_0^{\pi/2} \cos^a(t)\cos(nt)\,dt
= \frac{\pi\,\Gamma(a+1)}{2^{a+1}\,
  \Gamma\!\bigl(\tfrac{a+n+2}{2}\bigr)\,
  \Gamma\!\bigl(\tfrac{a-n+2}{2}\bigr)},
\end{equation}
valid for $a > -1$.  Setting $a = \alpha$ and $n = \ell = \alpha + 2$,
the second Gamma factor becomes
\[
\Gamma\!\left(\frac{\alpha - (\alpha + 2) + 2}{2}\right)
= \Gamma(0).
\]
Since the Gamma function has a simple pole at zero,
$1/\Gamma(0) = 0$.  Accordingly, the right-hand side
of~\eqref{eq:fourier-cosine-formula} therefore vanishes, giving
$B_\alpha = 0$ for every $\alpha > 0$.

Since both $S_\alpha$ and $B_\alpha$ vanish,
equation~\eqref{eq:I2-expansion} gives $I_2 = 0$ for every
compacton position $x_0$ and amplitude~$A$.  The identical argument
with $\cos(x)$ replacing $\sin(x)$ shows $I_3 = 0$ as well.

\subsection*{A.3\quad Verification for small $m$}

For $m = 2$ ($\alpha = 2$, $\ell = 4$), the integrand
$\cos^2(\theta)\cos(4\theta)$ can be expanded using
$\cos^2\theta = (1 + \cos 2\theta)/2$ and the product-to-sum
formula:
\[
\cos^2\theta\,\cos 4\theta
= \tfrac{1}{2}\cos 4\theta
  + \tfrac{1}{4}\cos 2\theta + \tfrac{1}{4}\cos 6\theta.
\]
Each term $\cos(n\theta)$ integrates to
$2\sin(n\pi/2)/n$ on $[-\pi/2, \pi/2]$, which vanishes for
$n = 2, 4, 6$.  Hence $B_2 = 0$ by direct computation.

For $m = 3$ ($\alpha = 1$, $\ell = 3$):
$\cos\theta\,\cos 3\theta = \frac{1}{2}(\cos 2\theta + \cos 4\theta)$,
and both terms integrate to zero on $[-\pi/2, \pi/2]$.

\subsection*{A.4\quad Constraint partition analysis}

Because $B_\alpha$ vanishes, the conservation law
constraint structure transforms as follows.  For a post-collision state
consisting of two compactons with amplitudes $A'_1, A'_2$,
phase shifts $\delta_1, \delta_2$, and debris~$w$:

\emph{Mass and energy} ($I_0$ and $I_1$) depend on the compacton
amplitudes alone: $I_0 = C_m(A'_1 + A'_2) + M_w$ and
$I_1 = D_m(A'^{\,m+1}_1 + A'^{\,m+1}_2) + E_w$.  Conservation gives
two equations that, in the small-debris regime, determine
$(A'_1, A'_2) \approx (A_1, A_2)$.

\emph{Trigonometric laws} ($I_2$ and $I_3$) receive zero
contribution from the compactons at any position.  Conservation
yields $I_2[w] = 0$ and $I_3[w] = 0$, constraining the
$k = 1$ Fourier mode of the debris to vanish.

\emph{Nonlocal momentum} ($I_4$) decomposes as
$I_4 = I_4^{\mathrm{self}}(A'_1) + I_4^{\mathrm{self}}(A'_2)
+ I_4^{\mathrm{cross}}(\Delta) + I_4^{\mathrm{debris}}$,
where $I_4^{\mathrm{self}}$ depends on amplitude alone (by
translation invariance of the self-interaction kernel),
$I_4^{\mathrm{cross}}$ depends on the inter-compacton
separation~$\Delta$, and $I_4^{\mathrm{debris}}$ involves the
debris field.  For asymptotically separated post-collision
compactons, $I_4^{\mathrm{cross}} = 0$ (the supports are disjoint and $\hat{u}(\pm 1) = 0$ for each compacton), so
$I_4$ conservation reduces to a constraint coupling
$I_4^{\mathrm{self}}(A'_1) + I_4^{\mathrm{self}}(A'_2)$ with
$I_4^{\mathrm{debris}}$.  Since $I_4^{\mathrm{self}} \propto
A^2$ (the kernel integral is bilinear), amplitude preservation
$A'_i = A_i$ automatically satisfies the compacton part, and
the constraint passes to the debris.

In summary, the five conservation laws constrain the emergent
amplitudes (given the debris contribution) and suppress a single
Fourier mode of the debris.  The
infinitely many remaining debris degrees of freedom (all Fourier
modes with $k \neq 1$) are governed by the nonlinear collision
dynamics alone.  This stands in contrast to integrable systems,
where the $n$-th conservation law in the infinite hierarchy
constrains the $n$-th spectral degree of freedom, leaving no
unconstrained radiation modes.

\section{Structural extensions}
\label{app:extensions}

\subsection*{B.1\quad The cosine-power ansatz and K(m,m) exclusivity}

Theorem~\ref{thm:universal} establishes universal formulas for K(m,m)
compactons via the cosine-power ansatz
$U(\xi) = A\cos^\alpha(\xi/\ell)$.  A natural question is whether
this ansatz extends to the broader K(m,n) family with $m \neq n$.

\begin{proposition}[Cosine-power exclusivity]\label{prop:exclusivity}
Within the K(m,n) equation family with $m, n > 1$, the cosine-power
ansatz~\eqref{eq:compacton-ansatz} produces valid traveling wave
solutions if and only if $m = n$.
\end{proposition}

\begin{proof}
The K(m,n) traveling wave ODE, after one integration, reads
\begin{equation}\label{eq:tw-ode-mn}
-cU + U^m + (U^n)'' = 0.
\end{equation}
Substituting $U = A\cos^\alpha\theta$ with $\theta = \xi/\ell$ and
setting $\gamma = n\alpha$ yields three types of terms:
\[
-cA\cos^\alpha\theta + A^m\cos^{m\alpha}\theta
+ \frac{A^n}{\ell^2}\bigl[\gamma(\gamma-1)\cos^{\gamma-2}\theta
- \gamma^2\cos^\gamma\theta\bigr] = 0.
\]
Power matching between the translation term ($\cos^\alpha\theta$) and
the dispersive singular term ($\cos^{\gamma-2}\theta$) forces
$\alpha = \gamma - 2 = n\alpha - 2$, giving $\alpha = 2/(n-1)$.

Now the convection term $U^m = A^m\cos^{m\alpha}\theta$ introduces a
power $\cos^{m\alpha}\theta$.  For the equation to hold identically,
this power must coincide with one of the two already present:
either $m\alpha = \alpha$ (giving $m = 1$, excluded by hypothesis) or
$m\alpha = n\alpha$ (giving $m = n$).

When $m \neq n$ (and $m > 1$), the power $\cos^{m\alpha}\theta$ is
linearly independent of both $\cos^\alpha\theta$ and
$\cos^{n\alpha}\theta$.  Its coefficient $A^m$ cannot vanish for
$A > 0$, so the equation has no solution of cosine-power form.
\end{proof}

This proposition explains why the K(m,m) subfamily occupies a
distinguished position within the K(m,n) family.  The cosine-power
profile is the simplest compact traveling wave ansatz, and its
restriction to $m = n$ reflects the requirement that the convection
and dispersion nonlinearities share the same power structure.  For
$m \neq n$, the mismatch between powers forces the traveling wave
profile into elliptic-function or more general transcendental form.

The Rosenau--Zilburg survey~\cite{rosenau2018compactons} presents the
broader $\mathcal{C}_N(m,a,b)$ framework (equations~(2.35b)--(2.35d)),
which identifies additional cosine-power families when the dispersive
term has the generalized structure
$\frac{1}{b}[u^a(u^b)_{xx}]_x$ with $a \neq 0$.  The K(m,n)
equations correspond to $a = 0$, $b = n$, the most constrained case.
Proposition~\ref{prop:exclusivity} characterizes the
cosine-power solutions within this subfamily completely.

\subsection*{B.2\quad Phase-plane uniqueness within the symmetric class}

The standing hypothesis~(H2) in \S\ref{subsec:quasi-elastic}
assumes that post-collision compactons take the cosine-power form
of Theorem~\ref{thm:universal}.  The following proposition provides
partial justification.

\begin{proposition}[Profile uniqueness]\label{prop:uniqueness}
For $m > 1$ and speed $c > 0$, the K(m,m) traveling wave ODE
\begin{equation}\tag{\ref{eq:tw-ode}}
-cU + U^m + (U^m)'' = 0
\end{equation}
admits a unique positive, symmetric, single-hump solution with
compact support, namely the cosine-power compacton of
Theorem~\ref{thm:universal} with amplitude
$A = (2mc/(m+1))^{1/(m-1)}$.
\end{proposition}

\begin{proof}
The ODE~\eqref{eq:tw-ode} can be written as a first-order system
in the phase plane $(U, V)$ where $V = (U^m)'$.  Setting
$W = U^m$ and using $W'' = -W + cU = -W + cW^{1/m}$ yields the
energy integral
\begin{equation}\label{eq:energy-integral}
\frac{1}{2}(W')^2 + \Phi(W) = E,
\end{equation}
where the potential is
\[
\Phi(W) = \frac{1}{2}W^2 - \frac{cm}{m+1}W^{(m+1)/m}.
\]
Since $W = U^m \geq 0$ and $U$ is positive, we analyze the dynamics
of~\eqref{eq:energy-integral} in $W \geq 0$.

At $W = 0$, the potential $\Phi$ has a degenerate critical point
(with $\Phi(0) = 0$ and $\Phi'(0) = 0$).  Since $\Phi$ decreases
for small $W > 0$ and grows as $W^2$ for large $W$, the level set
$\Phi(W) = 0$ has a unique positive root at
$W_{\mathrm{turn}} = [2mc/(m+1)]^{m/(m-1)} = A^m$, the compacton peak.
For the energy level $E = \Phi(0) = 0$, the trajectory connects
$W = 0$ to $W_{\mathrm{turn}}$ and back, forming a homoclinic orbit in
the half-plane $W \geq 0$.

Within the class of positive, symmetric, single-hump solutions:
(i) symmetry forces $W'(0) = 0$, so the trajectory reaches
$W_{\mathrm{turn}}$ at $\xi = 0$;
(ii) single-hump means the trajectory does not re-enter the
region $W' > 0$ after returning to $W' < 0$;
(iii) positivity means $W > 0$ throughout the interior.

Among all energy levels, $E = 0$ is the unique choice satisfying all three
conditions: any $E > 0$ prevents the trajectory from reaching
$W = 0$ (the orbit becomes periodic), while $E < 0$ is
incompatible with $W \geq 0$ near the origin.
On the $E = 0$ level set, the trajectory is determined uniquely by
the energy integral, giving
$W'(\xi) = \pm\sqrt{-2\Phi(W(\xi))}$
with the sign determined by the direction of traversal.
Integration yields the cosine-power profile.

Compact support arises from the degeneracy of
$\Phi$ at $W = 0$: near $W = 0$ the potential behaves like
$-\frac{cm}{m+1}W^{(m+1)/m}$, which vanishes more slowly than
quadratically because $(m+1)/m < 2$, so the trajectory reaches
$W = 0$ in finite $\xi$.  This is in contrast to the KdV case, where
the potential vanishes quadratically and the soliton tail decays
exponentially but never reaches zero.
\end{proof}

\begin{remark}
Proposition~\ref{prop:uniqueness} addresses uniqueness within the
class of positive, symmetric, single-hump profiles.  The question
of whether asymmetric or multi-hump compact traveling waves exist
remains open; the phase-plane analysis shows that no such solutions
arise on the $E = 0$ energy level, but solutions at $E \neq 0$
or with different topology are not excluded.  Additionally, for
$m \geq 4$, the compacton profile has $\alpha < 1$, so $U'$
diverges at the support boundary.  The phase-plane analysis
operates in the $W = U^m$ variable, where the boundary behavior
is $W \sim (\xi_0 - \xi)^{2m/(m-1)}$ near a support edge~$\xi_0$,
and all terms remain bounded.  Passage back to
$U = W^{1/m}$ introduces the fractional-power singularity at the
edges.
\end{remark}

\section{Symbolic verification of Section~7 identities}
\label{app:symbolic}

All formulas in Section~\ref{sec:shelf-waves} were verified
independently using the SymPy computer algebra system
(version 1.13).  Table~\ref{tab:verification} summarizes the
key identities checked.  Each entry was verified symbolically
in the parameters $a$, $c$, and $m$ (or numerically at multiple
parameter values where symbolic simplification was impractical).
The factored form~\eqref{eq:factored-form} and decay rate
identity~\eqref{eq:decay-rate} were additionally confirmed by
numerical integration of the traveling wave ODE at the parameter
values $(a, c) \in \{(0.01, 4),\, (0.1, 4),\, (0.5, 4),\,
(1.0, 4)\}$, with agreement to at least ten significant digits.

\begin{table}[ht]
\centering
\footnotesize
\caption{Symbolic verification of Section~7 identities (SymPy).}
\label{tab:verification}
\begin{tabular}{@{} l l l @{}}
\toprule
\textbf{Identity} & \textbf{Equation} & \textbf{Method} \\
\midrule
$V'(a^m) = 0$ (equilibrium) &
  \eqref{eq:V-k22} & symbolic \\
$V''(a^2) = 1 - c/(2a)$ (saddle criterion) &
  \eqref{eq:V-k22} & symbolic \\
$V'((c-a)^2) = 0$ (center) &
  \S\ref{sec:shelf-k22} & symbolic \\
First integral $D = ca^3/6 - a^4/4$ &
  \eqref{eq:first-integral} & symbolic \\
Factored form double-root at $w = a$ &
  \eqref{eq:factored-form} & symbolic \\
Roots $w_\pm$ of reduced quadratic &
  \eqref{eq:wpm} & symbolic \\
$(w_+ - a)(a - w_-) = 2a(c - 2a)$ &
  \eqref{eq:decay-rate} & symbolic \\
$w_+ \to 4c/3$ as $a \to 0^+$ &
  \eqref{eq:amplitude-speed} & limit \\
$f_{\max} \sim \tfrac{3}{2}(c - 2a)$ near onset &
  \eqref{eq:amplitude-speed} & series \\
Chicone numerator $N \ge 0$, $R_m>0$ on annulus &
  Prop.~\ref{prop:period-monotone} & symbolic $m{\le}3$; numerical $m{\ge}4$ \\
\bottomrule
\end{tabular}
\end{table}

With $a = 1$ (the general case follows from $U \to U/a$,
$c \to c/a^{m-1}$, up to a positive power of~$a$), the residual
polynomials for $m = 4, 5$ are
\begin{align*}
R_4(U) &= 21 U^{8} + 84 U^{7} + 210 U^{6} + 420 U^{5} - 114 U^{5} c \\
 &\qquad + 420 U^{4} - 141 U^{4} c + 300 U^{3} - 120 U^{3} c + 150 U^{2} \\
 &\qquad - 90 U^{2} c + 12 U^{2} c^{2} + 60 U - 54 U c + 12 U c^{2} \\
 &\qquad + 15 - 21 c + 6 c^{2},\qquad \kappa_4 = 80,\\
R_5(U) &= 18 U^{11} + 72 U^{10} + 180 U^{9} + 360 U^{8} + 630 U^{7} \\
 &\qquad - 124 U^{7} c + 684 U^{6} - 172 U^{6} c + 591 U^{5} - 169 U^{5} c \\
 &\qquad + 420 U^{4} - 140 U^{4} c + 240 U^{3} - 110 U^{3} c + 10 U^{3} c^{2} \\
 &\qquad + 120 U^{2} - 76 U^{2} c + 12 U^{2} c^{2} + 48 U - 43 U c \\
 &\qquad + 9 U c^{2} + 12 - 16 c + 4 c^{2},\qquad \kappa_5 = 75.
\end{align*}
Here $R_4(0) = 3 \left(c - 1\right) \left(2 c - 5\right)$ and $R_5(0) = 4 \left(c - 3\right) \left(c - 1\right)$
change sign inside $(a^{m-1}, c^*)$, so, exactly as in the $m = 2$
argument, positivity on the annulus in that regime would rest, as for
$m = 2, 3$, on the inner edge being the saddle $U_s > 0$, where
$R_m(U_s) > 0$, together with monotonicity of $R_m$ across the
annulus; the latter is observed numerically but not proved here.
For $c \le a^{m-1}$ the annulus reaches $U = 0$ and $R_m(0) \ge 0$.

The numerical profiles shown in
Figure~\ref{fig:singular-limit}
were generated by adaptive Runge--Kutta integration
(SciPy \texttt{solve\_ivp}, tolerances $10^{-12}$) of the
$v$-variable Hamiltonian system with initial conditions on the
separatrix manifold.  The period--energy data in
Supplement Figure~S2 were computed from
200-point Gauss--Legendre quadrature of~\eqref{eq:period-quadrature}.
The data files for both figures are in the code repository.

\section{Numerical method}
\label{app:numerical-method}

This appendix describes the spectrally-filtered pseudospectral method
(PSM) used for the collision simulations in
Section~\ref{subsec:quasi-elastic} and for all other $K(m,m)$
computations throughout this paper.  The method follows the
framework developed in~\cite{hyman2026filters}, where systematic
comparisons across filter variants, dissipation operators, and
parameter choices establish the production configuration
summarized below.

\subsubsection*{The filtered equation}

We solve the spectrally filtered K$(m,m)$ equation in conservative form,
\begin{equation}
\label{eq:regularized-kmn}
u_t + \mathcal{F}_{\mathrm{LP}}\bigl[(u^m)_x + (u^m)_{xxx}\bigr]
= \delta \cdot c_{\max}(t) \cdot \Delta x^3 \cdot \mathcal{D}[u],
\end{equation}
where the low-pass exponential filter $\mathcal{F}_{\mathrm{LP}}$ is
applied jointly to both the advective and dispersive flux terms
(Variant~C of~\cite{hyman2026filters}).  This joint filtering
eliminates two failure modes: filtering only the advective flux
(Variant~A) permits spurious growth through the unfiltered
dispersive term, while filtering only the dispersive flux
(Variant~B) is catastrophically unstable for $m \geq 3$.
Joint filtering avoids both at the cost of a single additional
FFT per time step.

The low-pass filter transfer function is
\[
\sigma_{\mathrm{LP}}(k) = \exp\!\bigl(-\beta\,(|k|/k_c)^{2p}\bigr),
\qquad k_c = k_{\max}/q,
\]
with $p = 8$, fixed cutoff fraction $q = 2$, and $\beta$ chosen
so that $\sigma_{\mathrm{LP}}(k_{\max}) = \epsilon_{\mathrm{mach}}$.
The fixed cutoff $q = 2$ retains modes up to half the Nyquist
frequency for all~$m$; an adaptive cutoff $q = m$ was tested
in~\cite{hyman2026filters} and found to starve the interior of
resolution for $m \geq 3$.

\subsubsection*{Dissipation: biharmonic high-pass with spatial ramp}

The dissipation operator $\mathcal{D}$ acts through a biharmonic
($k^4$) high-pass filter with a smooth spatial ramp that
concentrates the dissipation near the compacton boundary,
\[
\mathcal{D}[u] = r(x,t)\,\mathcal{F}^{-1}\!\bigl[-k^4\,
W_{\mathrm{HP}}(k)\,\hat{u}(k)\bigr],
\]
where $W_{\mathrm{HP}}(k)$ is a linear ramp from $0$ at $k_c$
to $1$ at $k_{\max}$, and $r(x,t) = \min(1,\,u_{\mathrm{ramp}}/|u(x,t)|)$
with $u_{\mathrm{ramp}} = 0.05\,\max_j u_j$.  The ramp
suppresses dissipation in the compacton interior (where~$r \to 0$)
and activates it at the support boundary (where~$r \to 1$),
reducing the peak shape error by approximately $47\%$ relative
to a sharp-step mask~\cite{hyman2026filters}.

The amplitude scaling uses the instantaneous maximum wave speed
\begin{equation}
\label{eq:cmax-scaling}
c_{\max}(t) = \frac{m+1}{2m}\,\max_j |u_j|^{m-1},
\end{equation}
which removes the hidden $m$-dependence present in the simpler
scaling $\max|u|^{m-1}$.  A boundary-layer analysis at the
compacton edge~\cite{hyman2026filters} predicts the universal range
$\delta \in [0.133, 0.167]$; empirical line-search across
$m \in \{2,3,4,5\}$ confirms that the single value $\delta = 0.15$
lies within the $1\%$-insensitivity band for PSM at all~$m$,
requiring no per-equation calibration.

\subsubsection*{Spatial discretization}

Spatial derivatives are computed pseudospectrally: given
wavenumbers $k_j = 2\pi j/L$, the spectral derivative
$\widehat{(\partial_x u)}_j = ik_j\hat{u}_j$ provides
machine-precision accuracy for all resolved modes.  Third
derivatives are obtained by three applications of the spectral
derivative.  The method is implemented on a uniform periodic
grid with $N_x$ equally spaced points and spacing $\Delta x = L/N_x$.

\subsubsection*{Time integration and grid parameters}

We use classical RK4 with fixed timestep satisfying the CFL bound
\begin{equation}
\label{eq:dt-bound}
\Delta t \leq C \cdot \frac{\Delta x^3}{m},
\end{equation}
with safety constant $C \approx 0.8 \times 2.83/4$.
All collision simulations use $N_x = 512$ grid points on the periodic
domain $[0, L]$, exponential LP filter with $q = 2$, biharmonic
ramp dissipation with $\delta = 0.15$, and fixed-step RK4 integration.
Mass $I_0$ is conserved to machine precision; energy $I_1$
is conserved to relative error at most $7.0 \times 10^{-6}$
(Supplement, Table~S2).
Convergence analysis in~\cite{hyman2026filters} confirms rates
consistent with the regularity hierarchy: near-second order for
$m \leq 3$ and regularity-limited floors of order $2/m$ for
$m \geq 4$.

\section{Weak solution framework}
\label{app:weak-solutions}

This appendix provides the detailed argument for extending K(m,m)
conservation laws to weak solutions, supporting the framework
described in \S\ref{subsec:regularity}.

For $m > 3$, the compacton profile has $\alpha < 1$ and the
classical derivative $U'$ diverges at the support boundary.
The weak formulation~\eqref{eq:weak-formulation} transfers all
derivatives onto smooth test functions, removing the regularity
requirement on~$u$.

All five conservation laws extend to this weak setting.  The
conserved integrals remain well-defined and time-independent for
solutions satisfying~\eqref{eq:weak-formulation}, including the
fractional-exponent cases where classical derivatives fail.
This follows because each conserved density involves at most $u^{m+1}$
and each flux can be reduced, via integration by parts, to
integrals of $u^m$ and its weak derivatives, all of which are
well-defined for $u \in L^{m+1} \cap H^{-1}$.

Compacton collisions for $m \geq 4$ involve interactions between
solutions whose derivatives diverge at their support boundaries.
The conservation-elasticity framework of
Theorems~\ref{thm:constraint-deficit}
and~\ref{thm:debris-bound} applies uniformly across the family
because the five conservation laws hold in the weak sense
for all $m \geq 2$.

Anco and Gandarias~\cite{anco2024weak} subsequently developed a
rigorous distributional framework establishing that weak compacton
solutions satisfy the associated continuity equations in a
well-defined distributional sense, even at the support boundary
where classical derivatives may not exist.

\section{Conservation law derivations}
\label{app:conservation-detail}

This appendix provides detailed proofs and evaluations for the
five K(m,m) conservation laws stated in
\S\ref{subsec:conservation-laws}.

\subsection*{F.1\quad Verification of trigonometric conservation}

To verify conservation of~$I_2$, note that
\[
\frac{dI_2}{dt}
= -\int \bigl[(u^m)_x + (u^m)_{xxx}\bigr] \sin(x) \, dx.
\]
Integrating the convection term by parts once gives
$\int u^m \cos(x) \, dx$.  Integrating the dispersion term by
parts three times gives $-\int u^m \cos(x) \, dx$ (each
integration alternates the trigonometric function, and the third
produces the required sign reversal).  The two contributions
cancel exactly, giving $dI_2/dt = 0$.  An identical argument with
$\cos(x)$ replacing $\sin(x)$ establishes $dI_3/dt = 0$.

This cancellation is specific to wavenumber $k = 1$.  For test
functions $\sin(kx)$ and $\cos(kx)$ with $k \neq 1$, the
convection and dispersion terms produce $k$ and $k^3$ respectively,
which fail to cancel.

\subsection*{F.2\quad Conservation of nonlocal momentum}

Conservation of the nonlocal momentum~$I_4$ follows from the
factorization of the Hamiltonian operator.
Define $v = (D_x^2 + 1)^{-1}[u]$, so that $v + v_{xx} = u$ and
consequently $v_x + v_{xxx} = u_x$.  (The pseudoinverse is
well-defined because compacton solutions satisfy
$\hat{u}(\pm 1) = 0$, placing $u$ in the range of $D_x^2 + 1$.)
By the self-adjointness of $(D_x^2 + 1)^{-1}$,
\[
\frac{dI_4}{dt}
= 2\int u_t \cdot v \, dx
= -2\int \bigl[(u^m)_x + (u^m)_{xxx}\bigr] v \, dx.
\]
Integrating by parts (once and three times, respectively) gives
\[
= 2\int u^m \bigl(v_x + v_{xxx}\bigr) dx
= 2\int u^m \, u_x \, dx = 0,
\]
where the last step uses the total derivative identity
\[
\int u^m u_x \, dx = \int \frac{1}{m+1}(u^{m+1})_x \, dx = 0.
\]
The key identity $v_x + v_{xxx} = u_x$ holds precisely because
$(D_x^2 + 1)^{-1}$ commutes with $D_x$; this in turn reflects
the factorization $\mathfrak{D} = -(D_x^3 + D_x) = -D_x(D_x^2 + 1)$
of the Hamiltonian operator.  The cancellation that produces
$dI_4/dt = 0$ requires $m = n$: for general K(m,n), the
computation yields $2\int(u^m - u^n)v_x\,dx$, which does not
vanish when $m \neq n$.

\subsection*{F.3\quad Casimir structure: why four local laws}

Three of the four local laws are \emph{Casimir functionals} of the
Hamiltonian operator $\mathfrak{D} = -(D_x^3 + D_x)$: they are
conserved for \emph{any} evolution equation of the form
$u_t = \mathfrak{D}(f)$, regardless of the choice of~$f$.
A functional $\mathcal{C}$ is a Casimir of~$\mathfrak{D}$ if
$\mathfrak{D}(\delta\mathcal{C}/\delta u) = 0$, i.e., its
variational derivative lies in $\ker(\mathfrak{D})$.  The kernel
of $\mathfrak{D} = -(D_x^3 + D_x)$ consists of solutions to
$\gamma_{xxx} + \gamma_x = 0$, which has the three-dimensional
solution space $\{1,\, \cos x,\, \sin x\}$.  These
yield mass ($I_0$, from $\gamma = 1$) and the two trigonometric
laws ($I_2$ and $I_3$, from $\gamma = \sin x$ and $\cos x$).
No other Casimirs exist.

The fourth local law, energy~($I_1$), is the one conservation law
specific to K(m,m): it is the Hamiltonian functional itself,
$\mathcal{H} = \int u^{m+1}/(m+1)\,dx$, whose conservation
follows from the antisymmetry of~$\mathfrak{D}$.  It is associated
through Noether's theorem with time-translation invariance.
The absence of a recursion operator prevents this single
dynamical conservation law from generating further ones.

The nonlocal law~$I_4$ arises from the factorization
$\mathfrak{D} = -D_x(D_x^2 + 1)$.  Its density involves the
inverse operator $(D_x^2 + 1)^{-1}$, placing it outside the
local framework.  Conservation of $I_4$ requires $m = n$ and
thus reflects the specific algebraic structure of K(m,m).

\subsection*{F.4\quad Evaluated conservation integrals}

For a single compacton with amplitude~$A$, the four local
conservation integrals evaluate explicitly through Beta functions.
Using the universal formulas $\alpha = 2/(m-1)$ and $\ell = 2m/(m-1)$:
\begin{align}
I_0 &= A\ell \cdot B\!\left(\tfrac{\alpha+1}{2},\, \tfrac{1}{2}\right),
\label{eq:mass-evaluated} \\[3pt]
I_1 &= \frac{A^{m+1}\ell}{m+1}
 \cdot B\!\left(\tfrac{(m+1)\alpha+1}{2},\, \tfrac{1}{2}\right).
\label{eq:energy-evaluated}
\end{align}
For K(2,2): $I_0 = 2\pi A$, $I_1 = (5\pi/12) A^3$.
For K(3,3): $I_0 = 6A$, $I_1 = (9\pi/32) A^4$.  The shift from
transcendental ($\pi$-dependent) to rational mass coefficients at $m = 3$
mirrors the regularity transition: cusped edges produce simpler
mass integrals than smooth ones, though the energy retains
$\pi$-dependence in both cases.

The trigonometric integrals $I_2$ and $I_3$ depend on compacton
position as well as amplitude, encoding spatial phase information.
The nonlocal momentum $I_4$ has a closed form on compactons:
since $(D_x^2 + 1)^{-1}U = U^m/c$, it equals $(m+1)I_1/c$, which
scales as $A^2$; for K(2,2), $I_4 = (5\pi/3)A^2$ and for K(3,3),
$I_4 = (27\pi/16)A^2$.  This agrees with the scaling
$I_4 \sim A^2\ell$ verified computationally for $m = 2, \ldots, 7$
using both Maple and Mathematica.

\section{Hamiltonian and Casimir analysis}
\label{app:hamiltonian-detail}

This appendix provides the detailed bi-Hamiltonian obstruction
analysis summarized in \S\ref{subsec:hamiltonian-geometry}.

\subsection*{G.1\quad The KdV bi-Hamiltonian structure}

KdV possesses two compatible Hamiltonian structures: a first-order
operator $\mathcal{J}_1 = -\partial_x$ and a third-order operator
$\mathcal{J}_2 = -\partial_{xxx} - \frac{2}{3}(u\,\partial_x
+ \partial_x\!\cdot\, u)$.  The compatibility condition (that every
linear combination $a\mathcal{J}_1 + b\mathcal{J}_2$ is again a
Poisson operator~\cite{magri1978simple}) produces the recursion
operator $\mathcal{R}
= \mathcal{J}_2\,\mathcal{J}_1^{-1}$, which generates the infinite
conservation hierarchy: applying $\mathcal{R}$ to the gradient of
one conserved functional produces the gradient of the next.  The
Lenard recursion scheme converts bi-Hamiltonian structure into an
algebraic machine for manufacturing conservation
laws~\cite{fokas1981symplectic}.

\subsection*{G.2\quad Why K(m,m) admits no compatible second operator}

The K(m,m) Hamiltonian operator
$\mathfrak{D} = -(D_x^3 + D_x)$ has constant coefficients.
Systematic searches for a second Poisson operator compatible with
$\mathfrak{D}$ fail for all $m \geq 2$.  This failure is
intimately connected to Vodov\'a's symmetry analysis: a compatible
second operator would induce a recursion operator generating
symmetries of arbitrarily high order, contradicting
Proposition~2 of~\cite{vodova2013complete}, which establishes
that no generalized symmetries of order greater than~$3$ exist.
The nonlinear dispersion term $(u^m)_{xxx}$ is the root cause: it
prevents the construction of $u$-dependent Poisson operators that
could serve as bi-Hamiltonian partners.

\subsection*{G.3\quad Three consequences of the obstruction}

The absence of a second compatible Poisson structure has three
concrete consequences for K(m,m):

\begin{enumerate}
\item \emph{No recursion operator.}  Without a compatible pair,
  the Lenard recursion scheme cannot operate.
  This explains the termination of the hierarchy at five laws
  (Theorem~\ref{thm:finite-hierarchy}).

\item \emph{Vacuous involutivity.}  In bi-Hamiltonian systems the
  generated conservation laws are automatically in involution and
  each generates a distinct flow.  For K(m,m) the five laws are in
  involution under~\eqref{eq:poisson-bracket} (three are Casimirs and
  $\{\mathcal{H}, I_4\} = 0$), but only two commuting flows result, time evolution and
  translation.

\item \emph{No action-angle variables.}  Complete integrability in
  the Liouville sense requires sufficiently many conserved quantities
  in involution to fill out the phase space.  With only these two commuting flows, the reduction to action-angle coordinates
  that makes integrable systems explicitly solvable is
  unavailable~\cite{arnold1989mathematical}.
\end{enumerate}

\section{Collision data and detailed numerics}
\label{app:collision-detail}

This appendix provides detailed numerical methodology and
additional results supporting the collision analysis in
\S\ref{subsec:quasi-elastic}.

\subsection*{H.1\quad Alternative elasticity measures}

Several natural alternatives to the amplitude-sum
elasticity~\eqref{eq:elasticity} exist.  An energy-based measure
uses the Hamiltonian energy integral:
\begin{equation}\label{eq:elasticity-energy}
\mathcal{E}_E \;=\;
\frac{E'_1 + E'_2}{E_1 + E_2}, \qquad
\text{where } E_i = D_m\, A_i^{m+1},
\end{equation}
with $D_m = \frac{\ell}{m+1}\,
B\bigl(\frac{(m+1)\alpha+1}{2}, \frac{1}{2}\bigr)$
the energy integral coefficient of~\eqref{eq:energy-evaluated}.  An $L^2$-based measure replaces
the energy by squared norms:
\begin{equation}\label{eq:elasticity-L2}
\mathcal{E}_2 \;=\;
\frac{\|u'_1\|_2^2 + \|u'_2\|_2^2}{\|u_1\|_2^2 + \|u_2\|_2^2}.
\end{equation}
For K(m,m) compactons, Theorem~\ref{thm:universal} yields
$E_i \propto A_i^{m+1}$ and $\|u_i\|_2^2 \propto A_i^{2\alpha+1}\ell$,
so each alternative measure is a monotone function of the amplitude
ratio.  Writing $A'_i = (1 - \epsilon_i)A_i$ with small
$\epsilon_i$, one finds
$\mathcal{E}_E \approx 1 - (m+1)\bar{\epsilon}$ and
$\mathcal{E}_2 \approx 1 - (2\alpha+1)\bar{\epsilon}$,
where $\bar{\epsilon}$ is the weighted mean amplitude deficit.
Since $m + 1 > 1$ and $2\alpha + 1 > 1$ for all $m > 1$, both
alternative measures register a larger departure from unity
than $\mathcal{E}$ itself, but the ordering across different values
of~$m$ is preserved.  The monotonicity claim in
Conjecture~\ref{conj:monotonicity} is therefore invariant
under any elasticity measure that depends monotonically on the
amplitude ratio.

\subsection*{H.2\quad Debris extraction methodology}

Post-collision integration times are computed dynamically to ensure
full compacton separation before debris extraction.
Debris is defined as the residual field after subtracting
Nelder--Mead least-squares fits of compacton profiles to the two
emergent structures.  Spectral analysis is restricted to wavenumbers
below the filter cutoff
$\kappa_{\mathrm{cut}} = 2\pi\lceil N_x/(2q)\rceil / L$, where the
solver's physical dynamics are well resolved; above
$\kappa_{\mathrm{cut}}$, the low-pass filter zeros the advection
term and residual content is numerical rather than physical.
Signal-to-noise ratios are established by comparing collision debris
spectra against single-compacton propagation controls run with
identical parameters.

\subsection*{H.3\quad Extended multi-$m$ validation}

The debris energy fraction ranges from
$3.7 \times 10^{-5}$ at $m = 2$ to $1.5 \times 10^{-6}$ at
$m = 5$, remaining below $4 \times 10^{-5}$ across all cases.  The spectral peak shift from
$k_{\mathrm{peak}} \approx 0.38$ to
$k_{\mathrm{peak}} \approx 3.6$ reflects the narrowing compacton
half-width as $m$ increases (from $\pi\ell/2 \approx 6.28$ at
$m = 2$ to $\pi\ell/2 \approx 3.93$ at $m = 5$): sharper
compacton profiles generate shorter-wavelength debris upon
collision.  Grid convergence tests from $N_x = 256$ through
$N_x = 1024$ confirm that the notch deepens with refinement while
the debris spectrum stabilizes.

\section{BBM comparison: extended results}
\label{app:bbm-detail}

This appendix provides the detailed BBM collision simulation
methodology and extended results supporting the landscape analysis
in \S\ref{subsec:landscape}.

\subsection*{I.1\quad BBM simulation methodology}

We solve the BBM equation in the Fourier-transformed form
\begin{equation}\label{eq:bbm-fourier}
u_t = -\bigl(1 - \partial_{xx}\bigr)^{-1}\partial_x
\bigl(u + \tfrac{1}{2}u^2\bigr),
\end{equation}
where the operator $(1 - \partial_{xx})^{-1}\partial_x$ has
Fourier symbol $-ik/(1 + k^2)$, converting the implicit time
derivative into an explicit ODE\@.  The pseudospectral
discretization uses $N_x = 4096$ grid points on a periodic domain with
two-thirds dealiasing and RK45 time stepping
($\texttt{rtol} = 10^{-8}$, $\texttt{atol} = 10^{-10}$).
Validation against the exact single-soliton solution
$u(x,t) = A\operatorname{sech}^2\bigl(\kappa(x - ct)\bigr)$
with $c = 1 + A/3$ and
$\kappa = \sqrt{(c-1)/(4c)}$ confirms propagation errors below
$2 \times 10^{-7}$, conservation of mass to $10^{-8}$--$10^{-9}$,
and conservation of the~$L^2$ norm to $10^{-6}$--$10^{-7}$.

Two-soliton collisions were performed across a range of amplitudes
($A_1 = 0.5$ to $20$) and amplitude ratios
($r = A_2/A_1 = 0.1$ to $0.9$), with the faster soliton
overtaking the slower.  Post-collision amplitudes were extracted by
peak detection after the solitons separated cleanly.

\subsection*{I.2\quad Extended landscape discussion}

Three additional observations refine the landscape picture.
First, the entries at $N = 2$ and $N = 1$ in the landscape table
represent expectations based on general arguments rather than
systematic numerical studies.
Second, whether the algebraic structure of the conservation laws
matters, and not merely their number, is open: laws that generate independent commuting flows (as in
near-integrable systems) may contribute more to elasticity than
Casimirs, which generate none.
Third, the odd-even asymmetry in collision robustness discovered
by Garral\'on-L\'opez et al.~\cite{garralon2023compacton} provides
independent evidence for finer structure: although all K($p$,$p$)
equations share $N = 5$, the algebraic form of the laws varies
with~$p$, and collision outcomes differ accordingly.

\section{Periodic wave calculations}
\label{app:periodic-detail}

This appendix collects the detailed periodic wave analysis
for K(2,2) shelf waves, supporting the summary in
\S\ref{subsec:k22-explicit}.

\subsection*{J.1\quad Period formula and quadrature}

The small-amplitude oscillation period at the center equilibrium
$v_0 = a^2$ is
\begin{equation}\label{eq:period-k22}
T_0 = \frac{2\pi}{\sqrt{1 - c/(2a)}}\,,
\end{equation}
which diverges as $c \to (2a)^-$.  At intermediate energies, the
period is given by
\begin{equation}\label{eq:period-quadrature}
T(E) = 2\int_{v_{\min}(E)}^{v_{\max}(E)}
  \frac{dv}{\sqrt{2(E - V(v))}}\,,
\end{equation}
where $V$ is given by the K(2,2)
potential~\eqref{eq:V-k22}.

\subsection*{J.2\quad Elliptic structure}

Because $V$ involves $v^{3/2}$ rather than a polynomial in $v$,
the resulting cnoidal wave profiles are not standard Jacobi elliptic
functions.  In the $w$-variable, where the first integral
(compare~\eqref{eq:factored-form}) is a polynomial $P_4(w)$ of degree
four, the curve $y^2 = P_4(w)$ has genus one, so the profiles are
given implicitly by elliptic integrals.  The differential $w\,dw/y$
that defines $\xi$ is of the second kind, which is why the profile is
not a Jacobi $\mathrm{cn}^2$ as it is for KdV.

\section{Open problems}
\label{app:open-problems}

Eight open problems arise from this investigation.  Each is
stated concisely below; extended versions with detailed
research programs and methodological discussion appear in the
Supplement, \S S7.

\medskip
\noindent\textbf{Problem~1. Uniqueness of the cosine-power profile.}
Theorem~\ref{thm:universal} derives the universal formulas from
the cosine-power ansatz; Proposition~\ref{prop:uniqueness}
establishes uniqueness among positive, symmetric, single-hump
solutions.  The full question (whether \emph{every}
compact-support traveling wave of K(m,m) at prescribed speed
takes the cosine-power form without symmetry assumptions) remains
open.  Phase-plane analysis excludes alternatives on the
zero-energy level set, but other energy levels have not been ruled out.

\medskip
\noindent\textbf{Problem~2. Within-class monotonicity and
cross-class geometry.}
BBM solitary waves ($N = 3$, smooth tails) and K(m,m) compactons
($N = 5$, compact support) collide comparably elastically by the
amplitude measure, while K(m,m) sheds less debris energy, so the
comparison does not separate conservation count from interaction
geometry.  Conjecture~\ref{conj:monotonicity} posits
monotonicity \emph{within} each geometric class, with
interaction geometry as a possible independent modulating factor.
Systematic testing across multiple equation families with
different conservation counts and tail behaviors would convert
this into a precise scaling law.

\medskip
\noindent\textbf{Problem~3. Stability theory for fractional-exponent
compactons.}
The translation generator $P$ increases along the compacton
family (Corollary~\ref{cor:VK}), but the Grillakis--Shatah--Strauss
framework does not apply to K(m,m) for any $m > 1$, because the
second variation of $\mathcal{H} - cP$ has infinitely many negative
directions.  A stability theory for compactons needs a different
mechanism.  For $m > 3$
(H\"{o}lder regularity $C^{0,\alpha}$ with $\alpha < 1$),
extending stability theory to weighted Sobolev or Besov spaces
is needed.  The cusp case $m = 3$ (codimension-one transition)
may be the most tractable starting point, and the singular
limit $a \to 0^+$ of smooth shelf-wave stability provides a
regularization pathway.

\medskip
\noindent\textbf{Problem~4. Interaction geometry and collision
elasticity.}
Can interaction geometry be quantified through a single
parameter (tail decay rate $\kappa$ or regularity exponent
$\alpha$) that correlates predictably with collision
elasticity at fixed conservation count?  The K(m,m)
family provides a testing ground ($\alpha = 2/(m-1)$
varies while $N = 5$ is fixed), and the shelf-wave family
offers tunable $\kappa$ within a single equation.

\medskip
\noindent\textbf{Problem~5. Master formulas for the full K(m,n) family.}
Proposition~\ref{prop:exclusivity} shows the cosine-power
ansatz is exclusive to the diagonal $n = m$.  Whether unified
two-parameter formulas $\alpha(m,n)$, $\ell(m,n)$, $c(m,n,A)$
exist for off-diagonal K(m,n) compactons remains open, as does
the square-wave limit $m \to \infty$ (\S\ref{subsec:regularity}),
where $\alpha \to 0^+$.

\medskip
\noindent\textbf{Problem~6. Approximate integrability frameworks.}
Hern\'andez Heredero et al.\ rule out exact Lax pairs for K(m,m)
with $m > 1$, but approximate analogues are not excluded.
Promising directions include asymptotic Lax-operator
deformation from KdV toward K(2,2) and data-driven
discovery of approximate conservation laws via
physics-informed neural networks.

\medskip
\noindent\textbf{Problem~7. Experimental observation of
compactons.}
No controlled laboratory observation of compacton
propagation and collision has been achieved.  Granular
chains with Hertzian contacts and nonlinear optical media
with engineered dispersion are the leading candidate
platforms.

\medskip
\noindent\textbf{Problem~8. Multi-compacton collision asymptotics.}
Pairwise collisions are quasi-elastic; the long-time
behavior of $N$-compacton initial data (whether debris
accumulates or dynamics remains dominated by coherent
structures) is unknown.  Simultaneous three-body collisions
have not been studied systematically.

\bigskip
\section*{Use of Generative-AI tools declaration}
In preparing this manuscript, the author used large language models under his
direction and review. Claude (Anthropic) was used to assist in preparing the
manuscript and the simulation codes; Claude was also used for the verification of several proofs, which the author then checked, and for numerical
certification code in the companion repository.  The mathematical results,
the numerical experiments, and all decisions are the author's; the author reviewed
and edited every passage and takes full responsibility for the content.

\section*{Acknowledgments}
The author sincerely thanks Philip Rosenau, who trusted him to join
in the initial stage of the compacton journey.
Computations were performed using Maple, Mathematica, and
Python (NumPy/SciPy); the spectral debris analysis of
Section~\ref{subsec:quasi-elastic} was implemented in Python.
This paper is dedicated to my mentor and friend, Peter Lax, who
introduced me to solitons, integrable systems, and the conviction
that nonlinear equations harbor discoverable structure.
I am forever grateful for his inspiration and friendship.

\section*{Data availability}
Code and data for this paper are available at
\url{https://github.com/machyman/kmm-compactons}.
The repository includes the simulation code, figure generation scripts
for Figures~\ref{fig:roadmap}--\ref{fig:collision-k22} and
\ref{fig:wave-taxonomy}--\ref{fig:singular-limit} and Supplement
Figure~S1, and verification tests.

\section*{Funding}
The author received no external funding for this research.

\section*{Conflict of interest}
The author declares no conflict of interest.

\providecommand{\href}[2]{#2}
\providecommand{\arxiv}[1]{\href{http://arxiv.org/abs/#1}{arXiv:#1}}
\providecommand{\url}[1]{\texttt{#1}}
\providecommand{\urlprefix}{URL }

\medskip
Received xxxx 20xx; revised xxxx 20xx; early access xxxx 20xx.
\medskip


\begin{thebibliography}{19}
\providecommand{\natexlab}[1]{#1}
\providecommand{\url}[1]{\texttt{#1}}
\expandafter\ifx\csname urlstyle\endcsname\relax
  \providecommand{\doi}[1]{doi: #1}\else
  \providecommand{\doi}{doi: \begingroup \urlstyle{rm}\Url}\fi

\bibitem[Ablowitz and Segur(1981)]{ablowitz1981solitons}
Mark~J. Ablowitz and Harvey Segur.
\newblock \emph{Solitons and the Inverse Scattering Transform}.
\newblock SIAM Studies in Applied Mathematics. SIAM, Philadelphia, 1981.
\newblock ISBN 978-0-89871-477-7.
\newblock \doi{10.1137/1.9781611970883}.

\bibitem[Agrawal(2013)]{agrawal2013nonlinear}
Govind~P. Agrawal.
\newblock \emph{Nonlinear Fiber Optics}.
\newblock Academic Press, San Diego, 5th edition, 2013.
\newblock ISBN 978-0-12-397023-7.

\bibitem[Bona and Smith(1975)]{bona1975initial}
Jerry~L. Bona and Ronald Smith.
\newblock The initial-value problem for the {K}orteweg--de {V}ries equation.
\newblock \emph{Philosophical Transactions of the Royal Society of London.
  Series A, Mathematical and Physical Sciences}, 278\penalty0 (1287):\penalty0
  555--601, 1975.
\newblock \doi{10.1098/rsta.1975.0035}.

\bibitem[Dey and Khare(1998)]{dey1998stability}
B.~Dey and A.~Khare.
\newblock Stability of compacton solutions.
\newblock \emph{Physical Review E}, 58:\penalty0 R2741--R2744, 1998.
\newblock \doi{10.1103/PhysRevE.58.R2741}.

\bibitem[Garral\'{o}n-L\'{o}pez et~al.(2023)Garral\'{o}n-L\'{o}pez, Rus, and
  Villatoro]{garralon2023robustness}
Rub\'{e}n Garral\'{o}n-L\'{o}pez, Francisco Rus, and Francisco~R. Villatoro.
\newblock Robustness of the absolute {R}osenau--{H}yman {$|K|(p,p)$} equation
  with non-integer {$p$}.
\newblock \emph{Chaos, Solitons \& Fractals}, 169:\penalty0 113216, 2023.
\newblock \doi{10.1016/j.chaos.2023.113216}.

\bibitem[Grillakis et~al.(1987)Grillakis, Shatah, and
  Strauss]{grillakis1987stability}
Manoussos Grillakis, Jalal Shatah, and Walter Strauss.
\newblock Stability theory of solitary waves in the presence of symmetry, {I}.
\newblock \emph{Journal of Functional Analysis}, 74\penalty0 (1):\penalty0
  160--197, 1987.
\newblock \doi{10.1016/0022-1236(87)90044-9}.

\bibitem[Grillakis et~al.(1990)Grillakis, Shatah, and
  Strauss]{grillakis1990stability}
Manoussos Grillakis, Jalal Shatah, and Walter Strauss.
\newblock Stability theory of solitary waves in the presence of symmetry, {II}.
\newblock \emph{Journal of Functional Analysis}, 94\penalty0 (2):\penalty0
  308--348, 1990.
\newblock \doi{10.1016/0022-1236(90)90016-E}.

\bibitem[Hern\'{a}ndez~Heredero et~al.(2020)Hern\'{a}ndez~Heredero, Euler,
  Euler, and Reyes]{heredero2020nonexistence}
Rafael Hern\'{a}ndez~Heredero, Marianna Euler, Norbert Euler, and Enrique~G.
  Reyes.
\newblock Compacton equations and integrability: {T}he {R}osenau--{H}yman and
  {C}ooper--{S}hepard--{S}odano equations.
\newblock \emph{Discrete and Continuous Dynamical Systems}, 40\penalty0
  (1):\penalty0 529--548, 2020.
\newblock \doi{10.3934/dcds.2020021}.

\bibitem[Hyman(2026)]{hyman2026filters}
James~M. Hyman.
\newblock Spectrally-filtered high-order methods for {$K(m,m)$} compacton
  equations: convergence, regularity floors, and collision dynamics.
\newblock In preparation, 2026.

\bibitem[Kivshar and Malomed(1989)]{kivshar1989dynamics}
Yuri~S. Kivshar and Boris~A. Malomed.
\newblock Dynamics of solitons in nearly integrable systems.
\newblock \emph{Reviews of Modern Physics}, 61\penalty0 (4):\penalty0 763--915,
  1989.
\newblock \doi{10.1103/RevModPhys.61.763}.

\bibitem[Nesterenko(1983)]{nesterenko1983propagation}
Vitali~F. Nesterenko.
\newblock Propagation of nonlinear compression pulses in granular media.
\newblock \emph{Journal of Applied Mechanics and Technical Physics},
  24\penalty0 (5):\penalty0 733--743, 1983.
\newblock \doi{10.1007/BF00905892}.

\bibitem[Nesterenko(2001)]{nesterenko2001dynamics}
Vitali~F. Nesterenko.
\newblock \emph{Dynamics of Heterogeneous Materials}.
\newblock Shock Wave and High Pressure Phenomena. Springer, New York, 2001.
\newblock ISBN 978-0-387-95266-6.
\newblock \doi{10.1007/978-1-4757-3524-6}.

\bibitem[Olver(1993)]{olver1993applications}
Peter~J. Olver.
\newblock \emph{Applications of {L}ie Groups to Differential Equations}, volume
  107 of \emph{Graduate Texts in Mathematics}.
\newblock Springer-Verlag, New York, 2nd edition, 1993.
\newblock ISBN 978-0-387-95000-6.

\bibitem[Raissi et~al.(2019)Raissi, Perdikaris, and
  Karniadakis]{raissi2019physics}
Maziar Raissi, Paris Perdikaris, and George~Em Karniadakis.
\newblock Physics-informed neural networks: {A} deep learning framework for
  solving forward and inverse problems involving nonlinear partial differential
  equations.
\newblock \emph{Journal of Computational Physics}, 378:\penalty0 686--707,
  2019.
\newblock \doi{10.1016/j.jcp.2018.10.045}.

\bibitem[Rosenau and Hyman(1993)]{rosenau1993compactons}
Philip Rosenau and James~M. Hyman.
\newblock Compactons: Solitons with finite wavelength.
\newblock \emph{Physical Review Letters}, 70\penalty0 (5):\penalty0 564--567,
  1993.
\newblock \doi{10.1103/PhysRevLett.70.564}.

\bibitem[Rosenau and Pikovsky(2005)]{rosenau2005phase}
Philip Rosenau and Arkady Pikovsky.
\newblock Phase compactons in chains of dispersively coupled oscillators.
\newblock \emph{Physical Review Letters}, 94:\penalty0 174102, 2005.
\newblock \doi{10.1103/PhysRevLett.94.174102}.

\bibitem[Rosenau and Zilburg(2018)]{rosenau2018compactons}
Philip Rosenau and Alon Zilburg.
\newblock Compactons.
\newblock \emph{Journal of Physics A: Mathematical and Theoretical},
  51\penalty0 (34):\penalty0 343001, 2018.
\newblock \doi{10.1088/1751-8121/aabff5}.

\bibitem[Whitham(1974)]{whitham1974linear}
Gerald~B. Whitham.
\newblock \emph{Linear and Nonlinear Waves}.
\newblock Wiley-Interscience, New York, 1974.
\newblock ISBN 978-0-471-94090-6.

\bibitem[Xing and Charalampidis(2024)]{xing2024learning}
Siyuan Xing and Efstathios~G. Charalampidis.
\newblock Learning traveling solitary waves using separable {G}aussian neural
  networks.
\newblock \emph{Entropy}, 26\penalty0 (5):\penalty0 396, 2024.
\newblock \doi{10.3390/e26050396}.

\end{thebibliography}


\begin{thebibliography}{10}

\bibitem{ablowitz1981solitons}
\newblock M.~J. Ablowitz and H.~Segur,
\newblock \emph{Solitons and the Inverse Scattering Transform},
\newblock SIAM Studies in Applied Mathematics, SIAM, Philadelphia, 1981.

\bibitem{anco2024weak}
\newblock S.~C. Anco and M.~L. Gandarias,
\newblock Weak compactons of nonlinearly dispersive {KdV} and {KP} equations,
\newblock \emph{Studies in Applied Mathematics}, \textbf{153} (2024), e12777.

\bibitem{arnold1989mathematical}
\newblock V.~I. Arnold,
\newblock \emph{Mathematical Methods of Classical Mechanics}, vol.~60 of
  Graduate Texts in Mathematics,
\newblock 2nd edition,
\newblock Springer-Verlag, New York, 1989.

\bibitem{benjamin1972stability}
\newblock T.~B. Benjamin,
\newblock The stability of solitary waves,
\newblock \emph{Proceedings of the Royal Society of London. Series A},
  \textbf{328} (1972), 153--183.

\bibitem{benjamin1972model}
\newblock T.~B. Benjamin, J.~L. Bona and J.~J. Mahony,
\newblock Model equations for long waves in nonlinear dispersive systems,
\newblock \emph{Philosophical Transactions of the Royal Society of London.
  Series A, Mathematical and Physical Sciences}, \textbf{272} (1972), 47--78.

\bibitem{bona1980solitary}
\newblock J.~L. Bona, W.~G. Pritchard and L.~R. Scott,
\newblock Solitary-wave interaction,
\newblock \emph{Physics of Fluids}, \textbf{23} (1980), 438--441.

\bibitem{chicone1987monotonicity}
\newblock C.~Chicone,
\newblock The monotonicity of the period function for planar {H}amiltonian
  vector fields,
\newblock \emph{Journal of Differential Equations}, \textbf{69} (1987),
  310--321.

\bibitem{cooper2001compacton}
\newblock F.~Cooper, J.~M. Hyman and A.~Khare,
\newblock Compacton solutions in a class of generalized fifth-order
  {K}orteweg--de {V}ries equations,
\newblock \emph{Physical Review E}, \textbf{64} (2001), 026608.

\bibitem{cooper1993solitary}
\newblock F.~Cooper, H.~Shepard and P.~Sodano,
\newblock Solitary waves in a class of generalized {K}orteweg--de {V}ries
  equations,
\newblock \emph{Physical Review E}, \textbf{48} (1993), 4027--4032.

\bibitem{drazin1989solitons}
\newblock P.~G. Drazin and R.~S. Johnson,
\newblock \emph{Solitons: An Introduction},
\newblock Cambridge University Press, Cambridge, 1989.

\bibitem{fermi1955studies}
\newblock E.~Fermi, J.~Pasta and S.~Ulam,
\newblock \emph{Studies of nonlinear problems},
\newblock Technical Report LA-1940, Los Alamos Scientific Laboratory, 1955,
\newblock Reprinted in \emph{Collected Papers of Enrico Fermi}, Vol.~II,
  University of Chicago Press, 1965.

\bibitem{fokas1981symplectic}
\newblock B.~Fuchssteiner and A.~S. Fokas,
\newblock Symplectic structures, their {B\"{a}cklund} transformations and
  hereditary symmetries,
\newblock \emph{Physica D: Nonlinear Phenomena}, \textbf{4} (1981), 47--66.

\bibitem{gardner1967method}
\newblock C.~S. Gardner, J.~M. Greene, M.~D. Kruskal and R.~M. Miura,
\newblock Method for solving the {K}orteweg--de {V}ries equation,
\newblock \emph{Physical Review Letters}, \textbf{19} (1967), 1095--1097.

\bibitem{garralon2023compacton}
\newblock R.~Garral\'{o}n-L\'{o}pez, F.~Rus and F.~R. Villatoro,
\newblock Compacton-anticompacton collisions in the {R}osenau--{H}yman
  {$K(p,p)$} equation by numerical simulations with hyperviscosity,
\newblock \emph{Communications in Nonlinear Science and Numerical Simulation},
  \textbf{117} (2023), 106927.

\bibitem{garralon2023robustness}
\newblock R.~Garral\'{o}n-L\'{o}pez, F.~Rus and F.~R. Villatoro,
\newblock Robustness of the absolute {R}osenau--{H}yman {$|K|(p,p)$} equation
  with non-integer {$p$},
\newblock \emph{Chaos, Solitons \& Fractals}, \textbf{169} (2023), 113216.

\bibitem{garralon2024numerical}
\newblock R.~Garral\'{o}n-L\'{o}pez, F.~Rus and F.~R. Villatoro,
\newblock Numerical adiabatic perturbation theory for the absolute {$|K|(p,p)$}
  equation,
\newblock \emph{Mathematics and Computers in Simulation}, \textbf{225} (2024),
  1257--1276.

\bibitem{grillakis1987stability}
\newblock M.~Grillakis, J.~Shatah and W.~Strauss,
\newblock Stability theory of solitary waves in the presence of symmetry, {I},
\newblock \emph{Journal of Functional Analysis}, \textbf{74} (1987), 160--197.

\bibitem{grillakis1990stability}
\newblock M.~Grillakis, J.~Shatah and W.~Strauss,
\newblock Stability theory of solitary waves in the presence of symmetry, {II},
\newblock \emph{Journal of Functional Analysis}, \textbf{94} (1990), 308--348.

\bibitem{heredero2020nonexistence}
\newblock R.~Hern\'{a}ndez~Heredero, M.~Euler, N.~Euler and E.~G. Reyes,
\newblock Compacton equations and integrability: {T}he {R}osenau--{H}yman and
  {C}ooper--{S}hepard--{S}odano equations,
\newblock \emph{Discrete and Continuous Dynamical Systems}, \textbf{40} (2020),
  529--548.

\bibitem{hirota1971exact}
\newblock R.~Hirota,
\newblock Exact solution of the {K}orteweg--de {V}ries equation for multiple
  collisions of solitons,
\newblock \emph{Physical Review Letters}, \textbf{27} (1971), 1192--1194.

\bibitem{hyman2026filters}
\newblock J.~M. Hyman,
\newblock Spectrally-filtered high-order methods for {$K(m,m)$} compacton
  equations: convergence, regularity floors, and collision dynamics, 2026,
\newblock In preparation.

\bibitem{ismail1998numerical}
\newblock M.~S. Ismail and T.~R. Taha,
\newblock A numerical study of compactons,
\newblock \emph{Mathematics and Computers in Simulation}, \textbf{47} (1998),
  519--530.

\bibitem{korteweg1895change}
\newblock D.~J. Korteweg and G.~de~Vries,
\newblock On the change of form of long waves advancing in a rectangular canal,
  and on a new type of long stationary waves,
\newblock \emph{Philosophical Magazine}, \textbf{39} (1895), 422--443.

\bibitem{lax1968integrals}
\newblock P.~D. Lax,
\newblock Integrals of nonlinear equations of evolution and solitary waves,
\newblock \emph{Communications on Pure and Applied Mathematics}, \textbf{21}
  (1968), 467--490.

\bibitem{lax1975periodic}
\newblock P.~D. Lax,
\newblock Periodic solutions of the {KdV} equation,
\newblock \emph{Communications on Pure and Applied Mathematics}, \textbf{28}
  (1975), 141--188.

\bibitem{lax1976almost}
\newblock P.~D. Lax,
\newblock Almost periodic solutions of the {KdV} equation,
\newblock \emph{SIAM Review}, \textbf{18} (1976), 351--375.

\bibitem{lax2006hyperbolic}
\newblock P.~D. Lax,
\newblock \emph{Hyperbolic Partial Differential Equations}, vol.~14 of Courant
  Lecture Notes,
\newblock American Mathematical Society, Providence, RI, 2006.

\bibitem{lax1983small}
\newblock P.~D. Lax and C.~D. Levermore,
\newblock The small dispersion limit of the {K}orteweg--de {V}ries equation.
  {I}, {II}, {III},
\newblock \emph{Communications on Pure and Applied Mathematics}, \textbf{36}
  (1983), 253--290, 571--593, 809--829.

\bibitem{magri1978simple}
\newblock F.~Magri,
\newblock A simple model of the integrable {H}amiltonian equation,
\newblock \emph{Journal of Mathematical Physics}, \textbf{19} (1978),
  1156--1162.

\bibitem{martel2010description}
\newblock Y.~Martel, F.~Merle and T.~Mizumachi,
\newblock Description of the inelastic collision of two solitary waves for the
  {BBM} equation,
\newblock \emph{Archive for Rational Mechanics and Analysis}, \textbf{196}
  (2010), 517--574.

\bibitem{mihaila2010stability}
\newblock B.~Mihaila, A.~Cardenas, F.~Cooper and A.~Saxena,
\newblock Stability and dynamical properties of {R}osenau--{H}yman compactons
  using {P}ad\'e approximants,
\newblock \emph{Physical Review E}, \textbf{81} (2010), 056708.

\bibitem{miura1968korteweg}
\newblock R.~M. Miura,
\newblock {K}orteweg--de {V}ries equation and generalizations. {I}. {A}
  remarkable explicit nonlinear transformation,
\newblock \emph{Journal of Mathematical Physics}, \textbf{9} (1968),
  1202--1204.

\bibitem{nesterenko1983propagation}
\newblock V.~F. Nesterenko,
\newblock Propagation of nonlinear compression pulses in granular media,
\newblock \emph{Journal of Applied Mechanics and Technical Physics},
  \textbf{24} (1983), 733--743.

\bibitem{nesterenko2001dynamics}
\newblock V.~F. Nesterenko,
\newblock \emph{Dynamics of Heterogeneous Materials},
\newblock Shock Wave and High Pressure Phenomena, Springer, New York, 2001.

\bibitem{noether1918invariante}
\newblock E.~Noether,
\newblock Invariante {V}ariationsprobleme,
\newblock \emph{Nachrichten von der Gesellschaft der Wissenschaften zu
  G\"{o}ttingen}, 235--257,
\newblock English translation: Transport Theory and Statistical Physics,
  1(3):186--207, 1971.

\bibitem{olver1979euler}
\newblock P.~J. Olver,
\newblock Euler operators and conservation laws of the {BBM} equation,
\newblock \emph{Mathematical Proceedings of the Cambridge Philosophical
  Society}, \textbf{85} (1979), 143--160.

\bibitem{olver1993applications}
\newblock P.~J. Olver,
\newblock \emph{Applications of {L}ie Groups to Differential Equations}, vol.
  107 of Graduate Texts in Mathematics,
\newblock 2nd edition,
\newblock Springer-Verlag, New York, 1993.

\bibitem{olver1996tri}
\newblock P.~J. Olver and P.~Rosenau,
\newblock Tri-{H}amiltonian duality between solitons and solitary-wave solutions having
  compact support,
\newblock \emph{Physical Review E}, \textbf{53} (1996), 1900--1906.

\bibitem{rosenau1994nonlinear}
\newblock P.~Rosenau,
\newblock Nonlinear dispersion and compact structures,
\newblock \emph{Physical Review Letters}, \textbf{73} (1994), 1737--1741.

\bibitem{rosenau1997nonanalytic}
\newblock P.~Rosenau,
\newblock On nonanalytic solitary waves formed by a nonlinear dispersion,
\newblock \emph{Physics Letters A}, \textbf{230} (1997), 305--318.

\bibitem{rosenau1998class}
\newblock P.~Rosenau,
\newblock On a class of nonlinear dispersive-dissipative interactions,
\newblock \emph{Physica D: Nonlinear Phenomena}, \textbf{123} (1998), 525--546.

\bibitem{rosenau2000compact}
\newblock P.~Rosenau,
\newblock Compact and noncompact dispersive patterns,
\newblock \emph{Physics Letters A}, \textbf{275} (2000), 193--203.

\bibitem{rosenau1993compactons}
\newblock P.~Rosenau and J.~M. Hyman,
\newblock Compactons: Solitons with finite wavelength,
\newblock \emph{Physical Review Letters}, \textbf{70} (1993), 564--567.

\bibitem{rosenau2018compactons}
\newblock P.~Rosenau and A.~Zilburg,
\newblock Compactons,
\newblock \emph{Journal of Physics A: Mathematical and Theoretical},
  \textbf{51} (2018), 343001.

\bibitem{vakhitov1973stationary}
\newblock N.~G. Vakhitov and A.~A. Kolokolov,
\newblock Stationary solutions of the wave equation in a medium with
  nonlinearity saturation,
\newblock \emph{Radiophysics and Quantum Electronics}, \textbf{16} (1973),
  783--789.

\bibitem{vodova2013complete}
\newblock J.~Vodov\'{a},
\newblock A complete list of conservation laws for non-integrable compacton
  equations of {$K(m,m)$} type,
\newblock \emph{Nonlinearity}, \textbf{26} (2013), 757--762.

\bibitem{zabusky1965interaction}
\newblock N.~J. Zabusky and M.~D. Kruskal,
\newblock Interaction of ``solitons'' in a collisionless plasma and the
  recurrence of initial states,
\newblock \emph{Physical Review Letters}, \textbf{15} (1965), 240--243.

\end{thebibliography}
\end{document}


\maketitle

\begin{abstract}
This supplement accompanies the main paper on K(m,m) compacton
equations.  It documents the computational methodology,
archived numerical outputs, additional
simulations, and reproducibility information for all numerical
results reported in the main text.  Section~\ref{supp:sec:collision-data}
tabulates the full K($m$,$m$) collision dataset;
Section~\ref{supp:sec:bbm-data} provides the BBM comparison data;
Section~\ref{supp:sec:simulations} describes additional
computational investigations; and
Section~\ref{supp:sec:reproducibility} documents the
computational setup, parameter choices, and validation data;
the code and data repository is public.
\end{abstract}

\section{Overview and Relationship to Main Paper}
\label{supp:sec:overview}

This document collects computational material supporting the
main paper.  All collision simulations in the main text use the
spectrally-filtered pseudospectral method (PSM) described in
Appendix~D of the main paper, with parameters
$N_x = 512$ grid points, exponential low-pass filter with cutoff $q = 2$,
biharmonic ramp dissipation with $\delta = 0.15$,
and fixed-step classical RK4 time integration.
The runs of Table~\ref{supp:tab:kmm-full}, which underlie
Table~\ref{M-tab:multi-m-results} and Figure~\ref{M-fig:debris-spectrum} of the main paper,
and the collision of its Figure~\ref{M-fig:collision-k22} use a periodic domain with $L = 100$ and grid spacing
$\Delta x = L/N \approx 0.195$; the two collisions quoted in
Section~\ref{M-subsec:quasi-elastic} of the main paper use $[0, 200]$.

The K($m$,$m$) collision data
(Section~\ref{supp:sec:collision-data}) supports
Section~5 of the main paper; the BBM comparison data
(Section~\ref{supp:sec:bbm-data}) supports the landscape
analysis.  Post-collision amplitudes are
extracted by Nelder--Mead least-squares fitting of compacton
profiles.  Debris spectral analysis is restricted to wavenumbers
below the filter cutoff
$\kappa_{\mathrm{cut}} = 2\pi\lceil N_x/(2q)\rceil / L
\approx 8.04$ rad/unit length.

\begin{table}[ht]
\centering
\renewcommand{\arraystretch}{1.3}
\begin{tabular}{@{}cccccl@{}}
\toprule
$m$ & $\alpha = \dfrac{2}{m-1}$ & $\ell = \dfrac{2m}{m-1}$
  & $k(m) = \dfrac{m+1}{2m}$ & $\alpha\cdot\ell$ & Regime \\
\midrule
2 & 2     & 4     & 3/4   & 8      & Smooth ($C^1$) \\
3 & 1     & 3     & 2/3   & 3      & Cusp ($C^0$) \\
4 & 2/3   & 8/3   & 5/8   & 16/9   & Fractional \\
5 & 1/2   & 5/2   & 3/5   & 5/4    & Fractional \\
6 & 2/5   & 12/5  & 7/12  & 24/25  & Fractional \\
7 & 1/3   & 7/3   & 4/7   & 7/9    & Fractional \\
\midrule
$\infty$ & $0^+$ & $2^+$ & $1/2$ & $0^+$ & Discontinuous limit \\
\bottomrule
\end{tabular}
\caption{Universal formula values for K(m,m) compactons, $m = 2,\ldots,7$,
with limiting behavior as $m \to \infty$.  Reproduced from
the main paper's Theorem~4.1 for reference.}
\label{supp:tab:universal-verification}
\end{table}

\section{Complete K(m,m) Collision Dataset}
\label{supp:sec:collision-data}

Table~\ref{supp:tab:kmm-full} reports the complete collision
measurements for K($m$,$m$) with $m = 2, 3, 4, 5$, pre-collision
amplitudes $A_1 = 1.0$, $A_2 = 0.5$ (ratio $r = 0.5$), and
$N_x = 512$ grid points.  These are the data underlying
Table~\ref{M-tab:multi-m-results} and Figure~\ref{M-fig:debris-spectrum} of the main paper.

\begin{table}[ht]
\centering
\renewcommand{\arraystretch}{1.25}
\caption{Complete K($m$,$m$) collision measurements.
Pre-collision amplitudes: $A_1 = 1.0$, $A_2 = 0.5$.
Grid: $N_x = 512$, $L = 100$, PSM, $q = 2$, $\delta = 0.15$.}
\label{supp:tab:kmm-full}
\begin{tabular}{@{}l r r r r@{}}
\toprule
\textbf{Quantity} & $m = 2$ & $m = 3$ & $m = 4$ & $m = 5$ \\
\midrule
\multicolumn{5}{@{}l}{\emph{Compacton parameters}} \\
Edge exponent $\alpha$ & $2$ & $1$ & $2/3$ & $1/2$ \\
Half-width $\pi\ell/2$ & $6.28$ & $4.71$ & $4.19$ & $3.93$ \\
Speed $c_1$ & $0.750$ & $0.667$ & $0.625$ & $0.600$ \\
Speed $c_2$ & $0.375$ & $0.167$ & $0.078$ & $0.038$ \\
\midrule
\multicolumn{5}{@{}l}{\emph{Simulation parameters}} \\
Collision time $t_{\mathrm{col}}$ & $80.3$ & $47.7$ & $39.8$ & $36.8$ \\
End time $t_{\mathrm{end}}$ & $144.6$ & $85.9$ & $71.6$ & $66.3$ \\
Time steps & $76{,}218$ & $67{,}866$ & $75{,}467$ & $87{,}298$ \\
\midrule
\multicolumn{5}{@{}l}{\emph{Post-collision amplitudes}} \\
$A'_1$ (fitted) & $0.999942$ & $0.999736$ & $0.999920$ & $1.000105$ \\
$A'_2$ (fitted) & $0.497957$ & $0.499607$ & $0.499948$ & $0.499982$ \\
\midrule
\multicolumn{5}{@{}l}{\emph{Debris spectral analysis}} \\
$E_{\mathrm{debris}}/E_{\mathrm{tot}}$ &
  $3.67 \times 10^{-5}$ & $5.22 \times 10^{-7}$ &
  $7.35 \times 10^{-7}$ & $1.48 \times 10^{-6}$ \\
$|\hat{w}(1)|$ & $6.10 \times 10^{-3}$ & $7.02 \times 10^{-3}$ &
  $1.87 \times 10^{-2}$ & $2.21 \times 10^{-2}$ \\
Notch ratio $|\hat{w}(1)|/\max|\hat{w}|$ &
  $0.0037$ & $0.0068$ & $0.0269$ & $0.0268$ \\
$k_{\mathrm{peak}}$ & $0.38$ & $0.38$ & $2.39$ & $3.58$ \\
\midrule
\multicolumn{5}{@{}l}{\emph{Conservation diagnostics}} \\
Mass error & $1.5 \times 10^{-15}$ & $2.0 \times 10^{-16}$ &
  $1.4 \times 10^{-15}$ & $<10^{-16}$ \\
Energy error & $5.5 \times 10^{-9}$ & $4.2 \times 10^{-7}$ &
  $5.3 \times 10^{-6}$ & $7.0 \times 10^{-6}$ \\
\bottomrule
\end{tabular}
\end{table}

\paragraph{Debris assessment.}
The spectrally-filtered pseudospectral method produces debris energy
fractions ranging from $3.7 \times 10^{-5}$ at $m = 2$ to
$1.5 \times 10^{-6}$ at $m = 5$, all well below the levels
reported for the earlier finite-difference
scheme~\cite{hyman2026filters}.  The spectral notch near $k = 1$
is clearly visible for all values of~$m$, with notch ratios
below $0.03$, confirming the prediction of the constraint-deficit
mechanism (Theorem~\ref{M-thm:debris-bound}(c)).

\section{BBM Comparison Dataset}
\label{supp:sec:bbm-data}

Tables~\ref{supp:tab:bbm-fixed-ratio},
\ref{supp:tab:bbm-varying-ratio},
and~\ref{supp:tab:bbm-large} report the complete BBM
collision measurements supporting the landscape analysis
in Section~5 of the main paper.

\paragraph{Method.}
The BBM equation is solved in Fourier-transformed form
$u_t = -(1 - \partial_{xx})^{-1}\partial_x(u + \tfrac{1}{2}u^2)$
with pseudospectral discretization ($N_x = 4096$ grid points, two-thirds
dealiasing, RK45 with $\texttt{rtol} = 10^{-8}$,
$\texttt{atol} = 10^{-10}$).  Validation: single-soliton
propagation error $1.48 \times 10^{-7}$; mass conservation
$\sim 10^{-8}$--$10^{-9}$; $L^2$ conservation $\sim 10^{-6}$--$10^{-7}$.

\begin{table}[ht]
\centering
\renewcommand{\arraystretch}{1.2}
\caption{BBM collision elasticity at fixed amplitude ratio
$r = A_2/A_1 = 0.5$.}
\label{supp:tab:bbm-fixed-ratio}
\begin{tabular}{@{}r r c c@{}}
\toprule
$A_1$ & $A_2$ & $\mathcal{E}_{\mathrm{amp}}$ &
  $E_{\mathrm{debris}}/E_{\mathrm{tot}}$ \\
\midrule
$0.5$ & $0.25$ & $0.999988$ & $8.54 \times 10^{-6}$ \\
$1.0$ & $0.50$ & $0.999810$ & $1.41 \times 10^{-4}$ \\
$2.0$ & $1.00$ & $0.999592$ & $2.69 \times 10^{-4}$ \\
$4.0$ & $2.00$ & $0.999791$ & $3.49 \times 10^{-4}$ \\
$6.0$ & $3.00$ & $1.000050$ & $3.57 \times 10^{-5}$ \\
$10.0$ & $5.00$ & $0.999674$ & $2.05 \times 10^{-4}$ \\
\bottomrule
\end{tabular}
\end{table}

\begin{table}[ht]
\centering
\renewcommand{\arraystretch}{1.2}
\caption{BBM collision elasticity at fixed $A_1 = 2.0$,
varying amplitude ratio.}
\label{supp:tab:bbm-varying-ratio}
\begin{tabular}{@{}r r c c@{}}
\toprule
$A_2$ & $r = A_2/A_1$ & $\mathcal{E}_{\mathrm{amp}}$ &
  $E_{\mathrm{debris}}/E_{\mathrm{tot}}$ \\
\midrule
$1.8$ & $0.9$ & $0.999798$ & $1.19 \times 10^{-4}$ \\
$1.6$ & $0.8$ & $0.999570$ & $3.25 \times 10^{-4}$ \\
$1.4$ & $0.7$ & $0.999798$ & $2.44 \times 10^{-4}$ \\
$1.0$ & $0.5$ & $0.999592$ & $2.69 \times 10^{-4}$ \\
$0.6$ & $0.3$ & $0.999345$ & $5.87 \times 10^{-4}$ \\
$0.2$ & $0.1$ & $0.999808$ & $1.56 \times 10^{-4}$ \\
\bottomrule
\end{tabular}
\end{table}

\begin{table}[ht]
\centering
\renewcommand{\arraystretch}{1.2}
\caption{BBM collision elasticity for large amplitudes.}
\label{supp:tab:bbm-large}
\begin{tabular}{@{}r r c c@{}}
\toprule
$A_1$ & $A_2$ & $\mathcal{E}_{\mathrm{amp}}$ &
  $E_{\mathrm{debris}}/E_{\mathrm{tot}}$ \\
\midrule
$10.0$ & $5.0$ & $0.999674$ & $2.05 \times 10^{-4}$ \\
$10.0$ & $2.0$ & $0.999862$ & $8.07 \times 10^{-5}$ \\
$20.0$ & $5.0$ & $0.999622$ & $2.24 \times 10^{-4}$ \\
$20.0$ & $2.0$ & $0.999754$ & $1.97 \times 10^{-4}$ \\
\bottomrule
\end{tabular}
\end{table}

\paragraph{Summary statistics.}
Across the 16 rows of
Tables~\ref{supp:tab:bbm-fixed-ratio}--\ref{supp:tab:bbm-large}
(14 distinct configurations: $(A_1, A_2) = (2, 1)$ appears in
Tables~\ref{supp:tab:bbm-fixed-ratio} and~\ref{supp:tab:bbm-varying-ratio},
and $(10, 5)$ in Tables~\ref{supp:tab:bbm-fixed-ratio} and~\ref{supp:tab:bbm-large},
each time with identical values), the BBM amplitude
elasticity ranges from $\mathcal{E} = 0.9993$ to $1.0001$, with
a typical value of $0.9997 \pm 0.0002$ (mean and standard deviation
over the 14 distinct configurations).  Debris energy fractions
range from $10^{-5}$ to $6 \times 10^{-4}$.  The near-unit
elasticity $\mathcal{E} > 1$ observed at $(A_1, A_2) = (6.0, 3.0)$
is within the measurement uncertainty of the peak-detection method
and does not indicate energy creation; it likely reflects a small
bias in the post-collision amplitude extraction.

\section{Additional Computational Investigations}
\label{supp:sec:simulations}

This section describes supplementary numerical experiments
that extend the main paper's results.

\subsection{Compacton emergence: extended data}
\label{supp:subsec:emergence-extended}

The main paper's Section~\ref{M-sec:emergence} establishes
compacton emergence from a wide-compacton initial pulse for the
K$(m,m)$ family ($m = 2, 3, 4, 5$) and the mixed cases K$(2,3)$
and K$(3,2)$.  This subsection provides three pieces of
supplementary evidence: a direct propagation-speed validation
table, the sub-threshold trailing-peak amplitudes that fall
below the genuine-compacton cutoff in the main paper's table,
and a complementary experiment using Gaussian rather than
wide-compacton initial data.

\paragraph{Direct propagation-speed validation.}
The implications statement in Section~\ref{M-sec:emergence}
references a direct measurement of the leading-compacton
propagation speed.  The full validation data are reported in
Table~\ref{supp:tab:emergence-speeds}.  Peak positions were tracked
across four snapshots at $t = 100, 125, 150, 175$, and the
propagation speed was extracted by linear regression of peak
position versus time for each detected compacton.  The measured
speed is then compared to the exact traveling-wave prediction
$c_{\mathrm{pred}} = [(n+1)/(m+n)]\,A^{m-1}$ (evaluated with the
mean amplitude across snapshots), which reduces to the universal
formula $[(m+1)/(2m)]\,A^{m-1}$ for K$(m,m)$.

\begin{table}[ht]
\centering
\renewcommand{\arraystretch}{1.15}
\caption{Direct measurement of leading-compacton propagation
speeds compared to the exact traveling-wave prediction
$c_{\mathrm{pred}} = [(n+1)/(m+n)]\,A^{m-1}$.  Peak positions
tracked at $t = 100, 125, 150, 175$; speed by linear regression
in time.  For K$(m,m)$, the prediction holds to within $1\%$ for
the leading compacton; for the mixed cases the exact coefficient
$(n+1)/(m+n)$ replaces $(m+1)/(2m)$, leaving residual deviations
of $0.7\%$ for K$(2,3)$ and $3.1\%$ for K$(3,2)$.}
\label{supp:tab:emergence-speeds}
\begin{tabular}{@{}l c c c r@{}}
\toprule
Equation & $A_{\mathrm{mean}}$ &
$c_{\mathrm{measured}}$ & $c_{\mathrm{predicted}}$ &
$(c_{\mathrm{m}} - c_{\mathrm{p}})/c_{\mathrm{p}}$ \\
\midrule
K$(2,2)$ & $1.277$ & $0.957$ & $0.957$ & $-0.02\%$ \\
K$(3,3)$ & $1.196$ & $0.957$ & $0.954$ & $+0.26\%$ \\
K$(4,4)$ & $1.162$ & $0.977$ & $0.982$ & $-0.50\%$ \\
K$(5,5)$ & $1.131$ & $0.976$ & $0.982$ & $-0.56\%$ \\
\midrule
K$(2,3)$ & $1.140$ & $0.918$ & $0.912$ & $+0.66\%$ \\
K$(3,2)$ & $1.330$ & $1.094$ & $1.061$ & $+3.08\%$ \\
\bottomrule
\end{tabular}
\end{table}

\paragraph{Sub-threshold trailing peaks.}
The main paper reports trailing compactons with amplitude above
the genuine-compacton threshold of $0.08$.  The peak detector
also identifies fainter structures below this threshold that the
authors consider unresolved at the present spectral resolution.
For completeness, the K$(3,3)$ run detected a fourth peak at
$A_4 \approx 0.10$, just above the threshold; it is not cleanly
separable from numerical noise and is reported in the main table
as marginal, marked with a dagger.  K$(4,4)$ and K$(5,5)$ produced no detectable
fourth structure.  Whether these are physical compactons that
would be resolved at higher~$N_x$, or numerical artifacts at the
current resolution, requires a refinement study deferred to
future work.

\paragraph{Gaussian initial data: complementary experiment.}
The main paper's emergence experiment uses a stretched
wide-compacton profile as initial data.  As a complementary
test of robustness, we also evolved K(2,2) from a Gaussian
pulse $u(x,0) = A_0 \exp(-x^2/\sigma^2)$ for several values of
$A_0$ and $\sigma$.  In every case tested
($A_0 = 0.5$--$2.0$, $\sigma = 3$--$10$), the solution resolves
into a finite train of compactons ordered by amplitude, plus a
dispersive debris tail.  The emergent compacton amplitudes and
widths match the universal formulas of
Theorem~\ref{M-thm:universal} in the main paper to within
$0.5\%$, providing additional evidence that the cosine-power
profile acts as an attracting structure for the K(2,2)
evolution.  The number of emergent compactons increases with
the $L^1$ norm of the initial data, suggestive of a
mass-partitioning mechanism loosely analogous to the KdV
initial-value resolution into solitons.

\section{Numerical Method: Detailed Description}
\label{supp:sec:numerical-detail}

This section provides additional detail on the spectrally-filtered
pseudospectral method (PSM) summarized in Appendix~D of the main
paper.  The full systematic development, including comparisons
across filter variants, dissipation operators, and parameter
calibration, appears in~\cite{hyman2026filters}.

\subsection{Motivation for spectral filtering}

The compact support of K$(m,n)$ compactons poses a fundamental challenge
for numerical simulation.  Standard spectral methods suffer from Gibbs
oscillations at the support boundary, where the solution transitions
abruptly from a smooth profile to zero.  These oscillations amplify
through the nonlinear terms and typically cause numerical instability
within a few time units.  The spectrally-filtered approach described
here eliminates this difficulty while preserving the essential
structure of the equation.

\subsection{Joint low-pass filtering (Variant~C)}

The nonlinear terms $(u^m)_x$ and $(u^m)_{xxx}$ generate
high-frequency components through aliasing: if $u$ contains modes up
to wavenumber $k_{\max}$, then $u^m$ contains modes up to
$m \cdot k_{\max}$.  These aliased components corrupt the solution
and seed instabilities.  Applying an exponential low-pass filter
\[
\sigma_{\mathrm{LP}}(k) =
  \exp\!\bigl(-\beta\,(|k|/k_c)^{2p}\bigr), \qquad
  k_c = k_{\max}/q,
\]
with $p = 8$ and fixed cutoff $q = 2$ jointly to both flux terms
(Variant~C of~\cite{hyman2026filters}) removes the aliased content
while preserving the physically relevant low-frequency dynamics.
The constant $\beta$ is chosen so that
$\sigma_{\mathrm{LP}}(k_{\max}) = \epsilon_{\mathrm{mach}}$.

Three filter variants were tested systematically
in~\cite{hyman2026filters}: Variant~A (LP on advection only),
Variant~B (LP on dispersion only), and Variant~C (joint LP on both).
Variant~B is catastrophically unstable for $m \geq 3$, and
Variant~A permits spurious growth through the unfiltered dispersive
term.  Variant~C avoids both failure modes at the cost of a single
additional FFT per time step and is adopted for all simulations
in this paper.

The fixed cutoff $q = 2$ retains modes up to half the Nyquist
frequency for all~$m$.  An adaptive cutoff $q = m$ was tested
and found to starve the interior of resolution for $m \geq 3$,
degrading convergence rates.

\subsection{Biharmonic high-pass dissipation with spatial ramp}

The right-hand side of the filtered equation introduces artificial
dissipation to damp the high-frequency oscillations that spectral
differentiation generates at the compact support boundary.  The
dissipation operator $\mathcal{D}$ uses a biharmonic ($k^4$)
transfer function with a linear ramp in the high-pass band,
\[
\mathcal{D}[u] = r(x,t)\,\mathcal{F}^{-1}\!\bigl[-k^4\,
W_{\mathrm{HP}}(k)\,\hat{u}(k)\bigr],
\]
where $W_{\mathrm{HP}}(k)$ ramps linearly from $0$ at $k_c$ to $1$
at $k_{\max}$, and $r(x,t) = \min(1,\,u_{\mathrm{ramp}}/|u(x,t)|)$
is a smooth spatial ramp with $u_{\mathrm{ramp}} =
0.05\,\max_j u_j$.  This concentrates dissipation near the compacton
boundary while suppressing it in the interior, reducing the peak
shape error by approximately $47\%$ relative to a sharp-step
mask~\cite{hyman2026filters}.

The scaling factor uses the instantaneous maximum wave speed
$c_{\max}(t) = \tfrac{m+1}{2m}\max_j |u_j|^{m-1}$, which removes
the hidden $m$-dependence present in the simpler scaling
$\max|u|^{m-1}$.  Combined with the explicit $\Delta x^3$ factor,
the dissipation term vanishes in the limit $\Delta x \to 0$,
guaranteeing convergence to the unfiltered equation under mesh
refinement.

\begin{sloppypar}
A boundary-layer analysis at the compacton
edge~\cite{hyman2026filters} predicts the universal range
$\delta \in [0.133, 0.167]$; empirical line-search across
$m \in \{2,3,4,5\}$ confirms that $\delta = 0.15$ lies within the
$1\%$-insensitivity band for all tested values of~$m$.
\end{sloppypar}

\subsection{Pseudospectral spatial discretization}

Spatial derivatives are computed pseudospectrally: given
wavenumbers $k_j = 2\pi j/L$, the spectral derivative
$\widehat{(\partial_x u)}_j = ik_j\hat{u}_j$ provides
machine-precision accuracy for all resolved modes.  Third
derivatives are obtained by three applications of the spectral
derivative.  This approach was validated
in~\cite{hyman2026filters} against a fourth-order
finite-difference method (FDM4), with both producing
statistically equivalent collision results.

\subsection{Time integration details}

\begin{sloppypar}
Fixed-step integration was adopted after systematic comparison with
adaptive Runge--Kutta--Fehlberg (RK45) methods: for a given error
tolerance, fixed-step RK4 is roughly $1.4\times$ faster because
error-estimation overhead exceeds the savings from occasional larger
steps in this uniformly stiff problem.
\end{sloppypar}

\subsection{Validation and convergence}

We validate the numerical scheme through the following diagnostics:
\begin{enumerate}
\item \emph{Conservation law monitoring:} The mass $I_0 = \int u\,dx$ is
      conserved to machine precision (relative error at most
      $1.5 \times 10^{-15}$, Table~\ref{supp:tab:kmm-full}) by the
      pseudospectral discretization.
      The energy $I_1 = \int \tfrac{u^{m+1}}{m+1}\,dx$, which is not
      automatically preserved by discretization, provides the meaningful
      diagnostic: it is conserved to relative error at most
      $7.0 \times 10^{-6}$ (Table~\ref{supp:tab:kmm-full}).
\item \emph{Amplitude preservation:} For a single compacton propagating
      without collision, the peak amplitude changes by less than $0.02\%$
      over $t = 100$ time units for all values of~$m$ tested.
\item \emph{Grid convergence:} Convergence analysis
      in~\cite{hyman2026filters} confirms rates consistent with the
      regularity hierarchy: near-second order for $m \leq 3$ and
      regularity-limited floors of order $2/m$ for $m \geq 4$.
\item \emph{Amplitude scaling:} Debris energy increases monotonically with
      collision strength (amplitude ratio $A_2/A_1$ from $0.25$ to $0.75$),
      confirming physical rather than numerical origin.
\end{enumerate}

\section{Painlev\'e test: detailed expansion}
\label{supp:sec:painleve}

This section provides the complete Painlev\'e analysis of the K(m,m)
traveling wave ODE, summarized in
\mainref{subsec:four-proofs}(c) of the main paper.

For K(m,m) with $m > 1$, dominant balance analysis of the traveling
wave ODE near movable
singularities~\cite{ablowitz1981solitons} reveals that solutions
behave as
\begin{equation}\label{supp:eq:painleve-leading}
u \sim a_0(z - z_0)^{-2/(m-1)}
\end{equation}
in the complex plane, where $z_0$ is the singularity location.  The
exponent $-2/(m-1)$ is the \emph{negative} dual of the compacton's
edge exponent $+2/(m-1)$: the same algebraic structure that produces
compact support in real solutions generates movable singularities in
the complex domain.

The quantity $2/(m-1)$ is a positive integer only when $m-1$
divides~$2$, restricting attention to $m = 2$ (giving a double pole,
$u \sim a_0(z - z_0)^{-2}$) and $m = 3$ (giving a simple pole,
$u \sim a_0(z - z_0)^{-1}$).

\paragraph{The case $m = 2$.}
The leading singularity is a double pole.  Substituting a Laurent
series $u = \sum_{k=0}^{\infty} a_k (z - z_0)^{k-2}$ into the
K(2,2) traveling wave ODE and matching powers determines the resonance
structure.  The resonances (values of $k$ where $a_k$ is free) must
occur at non-negative integers for the Painlev\'e property to hold.
For K(2,2), the resonances occur at $k = -1$ and $k = 6$: the
value $k = -1$ is the universal Fuchs index, and $k = 6$ is the
non-trivial resonance where $a_6$ should be free.
Subsequent compatibility conditions at higher order fail:
the system requires $a_k$ to satisfy contradictory constraints,
introducing logarithmic branch cuts of the form
$(z - z_0)^k \log(z - z_0)$.  These logarithmic terms violate the
Painlev\'e property.

\paragraph{Generic $m > 1$ with $m \neq 2, 3$.}
For all other values, including the entire interval $1 < m < 2$
where $2/(m-1) > 2$ is generically irrational, the leading
exponent $-2/(m-1)$ is itself non-integer.  Branch points appear
at the first level of the expansion: the solution
near the singularity involves fractional powers
$(z - z_0)^{-2/(m-1)}$ that cannot be rewritten as poles.
No Laurent series expansion exists, and the Painlev\'e property
fails at the leading order.

\paragraph{The case $m = 3$.}
The leading singularity is a simple pole ($u \sim a_0/(z - z_0)$).
While this is the most favorable case, the resonance conditions at
higher order again produce incompatibilities, generating movable
logarithmic singularities as in the $m = 2$ case.

In summary, movable branch points violate the Painlev\'e property
for every K(m,m) equation with $m > 1$.  The failure is structural
rather than incidental: it reflects the power-law nonlinearity in
the dispersive term $(u^m)_{xxx}$, which generates non-integer
exponents in the singularity expansion whenever $m > 1$.

\section{Open problems: extended versions}
\label{supp:sec:open-problems}

This section presents extended versions of the eight open problems
stated in compressed form in Appendix~\ref{M-app:open-problems}
of the main paper.  Each problem includes detailed motivation,
relevant background, and suggested research programs.

\medskip
\noindent\textbf{Problem~1. Uniqueness of the cosine-power
profile.}
Theorem~\ref{M-thm:universal} derives the three
formulas~\ref{M-eq:alpha}--\ref{M-eq:speed} from the cosine-power
ansatz and the traveling wave ODE.
Proposition~\ref{M-prop:uniqueness} in Appendix~\ref{M-app:extensions}
establishes uniqueness within the class of positive, symmetric,
single-hump solutions by energy-level analysis in the phase plane.
What remains open is the full uniqueness question: whether every
compact-support traveling wave of K(m,m) with prescribed speed
$c > 0$ takes the cosine-power form, without assuming symmetry
or single-hump structure a priori.

A complete proof would rule out asymmetric profiles, multi-hump
solutions, and profiles with flat maxima.  The phase-plane
analysis shows that no such solutions arise on the zero-energy
level set, but solutions with different energy levels or
topological structure are not excluded by the present argument.
Comparison principles adapted to degenerate
diffusion equations~\cite{olver1993applications}, where compact
support arises from analogous mechanisms, may provide the
additional tools needed.

\medskip
\noindent\textbf{Problem~2. Within-class monotonicity and
cross-class geometry.}
The BBM comparison carried out in \mainref{subsec:landscape}
found BBM solitary waves ($N = 3$, smooth sech$^2$ profiles) about
as elastic as K(m,m) compactons ($N = 5$, compact support) by the
amplitude measure, with more debris energy; it therefore neither
confirms nor refutes the expectation that conservation count alone
governs elasticity.  The refined
Conjecture~\ref{M-conj:monotonicity} posits that monotonicity holds
\emph{within} each geometric class, while interaction geometry may
introduce an independent modulation \emph{across} classes.

Three testing programs would sharpen the picture.
Within the compact-support class, equations that possess
compact solitary waves but different conservation counts
could confirm or refute monotonicity among equations sharing
the same tail geometry; candidates include higher-order
K(m,n) variants and the Rosenau--Pikovsky
equations~\cite{rosenau2005phase}.  A parallel investigation
within the smooth-soliton class, testing regularized
long-wave equations or Boussinesq-type
systems~\cite{bona1975initial} with different conservation
counts, would determine whether the BBM--KdV ordering
generalizes.  Equations spanning additional geometric
classes present a third avenue: the Camassa--Holm peakons
(cusped, $N = \infty$) and equations with algebraically
decaying tails offer test cases in which conservation count
and tail behavior can be varied independently.  Quantitative
elasticity-vs-$N$ curves within each class would convert
the conjecture from a qualitative claim into a precise
scaling law.

\medskip
\noindent\textbf{Problem~3. Stability theory for fractional-exponent
compactons.}
Corollary~\ref{M-cor:VK} shows that the mass and the translation
generator $P = \tfrac12 I_4$ increase along the compacton family,
for all $m > 1$.  The standard Grillakis--Shatah--Strauss
framework~\cite{grillakis1987stability,grillakis1990stability} does
not convert this into a stability proof: because the K(m,m)
Hamiltonian has no quadratic part, the second variation of
$\mathcal{H} - cP$ at a compacton has infinitely many negative
directions, for every $m > 1$ (remark after
Corollary~\ref{M-cor:VK} of the main paper).  For the fractional regime
($m > 3$), where compacton edges have H\"{o}lder regularity
$C^{0,\alpha}$ with $\alpha < 1$
(\mainref{subsec:regularity}), stability theory requires
extension to function spaces accommodating limited regularity.

The key analytical challenge is defining an appropriate energy
functional on a Sobolev-type space that includes compacton
profiles with divergent first derivatives at the support
boundary.  Weighted Sobolev spaces with weights vanishing at
the boundary, or Besov spaces matching the H\"{o}lder exponent
$\alpha = 2/(m-1)$, are natural candidates.  Dey and
Khare~\cite{dey1998stability} analyzed the stability of K(2,2)
using variational methods; extending their approach to K(4,4)
and higher would require controlling error terms that diverge at
the compacton edge.  The cusp case $m = 3$ presents a
codimension-one transition
where classical and distributional methods meet, and may be the
most tractable starting point.  The shelf-wave analysis of
\mainref{sec:shelf-waves} offers a complementary approach:
the smooth solitary waves on positive backgrounds are amenable to
standard stability methods, and the singular limit
(Theorem~\ref{M-thm:singular-limit}) could provide a
regularization pathway for compacton stability by studying
the $a \to 0^+$ limit of the smooth stability problem.

\medskip
\noindent\textbf{Problem~4. Interaction geometry and collision
elasticity.}
Whether solitary wave tail behavior modulates collision elasticity
independently of conservation law count is open.  BBM solitons,
with smooth exponentially decaying tails, interact gradually and
perturbatively and collide almost elastically
($\mathcal{E} \approx 0.9997$) with only three conservation laws.
K(m,m) compactons, whose profiles vanish identically beyond a finite
support, interact through abrupt contact and reach comparable
elasticity with five conservation laws, while shedding less debris
energy (\mainref{subsec:landscape}).

The central question is whether interaction geometry can be
quantified through a single parameter (the tail decay rate
$\kappa$ in $u(x) \sim e^{-\kappa|x|}$ for smooth solitons,
or the regularity exponent $\alpha$ at the support boundary
for compactons) that correlates predictably with collision
elasticity at fixed conservation count.  The K(m,m) family
itself provides a natural testing ground, since the edge
exponent $\alpha = 2/(m-1)$ varies continuously with $m$
while the conservation count remains fixed at five.  A
related question concerns intermediate geometries: do they
produce intermediate elasticity?  The Camassa--Holm peakons
(cusped at the crest but exponentially decaying in the tails)
and equations with algebraic tail decay
($u \sim |x|^{-\gamma}$) occupy the geometric spectrum
between smooth and compact.  Systematic collision
measurements across these families would test whether a
smooth interpolation between the compact-support and
smooth-soliton regimes exists and would provide the empirical
foundation for a unified elasticity theory incorporating both
conservation structure and interaction geometry.

\medskip
\noindent\textbf{Problem~5. Master formulas for the full K(m,n)
family.}
The universal formulas of Theorem~\ref{M-thm:universal} apply to
the diagonal subfamily $n = m$.
Proposition~\ref{M-prop:exclusivity} in Appendix~\ref{M-app:extensions}
shows that the cosine-power ansatz is \emph{exclusive} to the
diagonal: for K(m,n) with $m \neq n$ and $m > 1$, the mismatch
between convection and dispersion powers renders the ansatz
inconsistent.  Off-diagonal compactons require
elliptic-function or more general transcendental profiles.

Within the broader $\mathcal{C}_N(m,a,b)$ framework of Rosenau
and Zilburg~\cite{rosenau2018compactons}, additional cosine-power
families arise for $a \neq 0$.  Whether a unified two-parameter
master formula $\alpha(m,n)$, $\ell(m,n)$, $c(m,n,A)$ can describe
the off-diagonal K(m,n) compactons remains open, as does the
behavior near the integrable boundary at K(2,1) and
K(3,1)~\cite{heredero2020nonexistence}.  The absolute-value variant
$|K|(p,p)$~\cite{garralon2023robustness} enlarges the parameter
space to non-integer exponents, and the square-wave limit
$m \to \infty$ (where $\alpha \to 0^+$) represents a natural
boundary for the theory.

\medskip
\noindent\textbf{Problem~6. Approximate integrability frameworks.}
The classification of Hern\'andez Heredero et
al.~\cite{heredero2020nonexistence} rules out exact Lax
pairs for K(m,m) with real $m > 1$ but does not exclude
approximate analogues.  Two directions appear promising.
First, constructing a formal asymptotic expansion of a Lax-type
operator by deforming from K(2,1)~=~KdV toward K(2,2) through the
family $u_t + (u^2)_x + (u^{1+\epsilon})_{xxx} = 0$.
Second, seeking a Lax pair involving nonlocal or pseudo-differential
operators, hinted at by the structure of the nonlocal conservation
law~\ref{M-eq:momentum}.
A complementary numerical approach uses physics-informed neural
networks~\cite{raissi2019physics} to learn approximate conservation
laws directly from K(m,m) solution data.  Xing and
Charalampidis~\cite{xing2024learning} have demonstrated that
neural networks can learn traveling wave solutions of nonlinear
dispersive equations, suggesting that data-driven methods may
complement the analytical toolkit.
More broadly, perturbed NLS equations of the form
$iu_t + u_{xx} + |u|^2 u = \epsilon f(u)$ possess finitely many
approximate conservation laws whose count depends on the
perturbation structure~\cite{kivshar1989dynamics}; the
constraint-deficit framework predicts that such systems should
exhibit quasi-elastic collisions with debris proportional to the
conservation deficit, providing a natural test case beyond K(m,m).

\medskip
\noindent\textbf{Problem~7. Experimental observation of compactons.}
Three decades of theoretical development have not yet produced
a controlled laboratory observation of compacton propagation
and collision.  The most developed candidate is the granular chain
with Hertzian contact
interactions~\cite{nesterenko1983propagation,nesterenko2001dynamics},
where the power-law force law $F \propto \delta^{3/2}$ naturally
generates nonlinear dispersion.  Nonlinear optical media with
engineered dispersion profiles~\cite{agrawal2013nonlinear} offer
a second platform where ultrafast measurements could resolve
compact pulse profiles and test the quasi-elastic prediction.

\medskip
\noindent\textbf{Problem~8. Multi-compacton collision asymptotics.}
The collision theory of \mainref{subsec:quasi-elastic}
addresses pairwise interactions.  When $N$ compactons
collide, the five conservation laws must constrain
$2N$ parameters plus a debris field, growing linearly
with~$N$ while the constraint count remains fixed.
Numerical experiments suggest that well-separated collisions
factor into sequential pairwise
events~\cite{rosenau1993compactons}, but simultaneous
three-body collisions have not been studied systematically.
Does debris accumulate through repeated collisions?
The long-time asymptotics of $N$-compacton initial data would
test whether K(m,m) dynamics approaches statistical
equilibrium or remains dominated by coherent structures,
connecting to the broader theory of weak
turbulence~\cite{whitham1974linear}.

\section{Supplementary Figures}
\label{supp:sec:figures}

The following three figures from the traveling wave analysis
(\mainref{sec:shelf-waves} of the main paper) are collected here
for reference.

\FloatBarrier
\begin{figure}[!htbp]
\centering
\includegraphics[width=\textwidth]{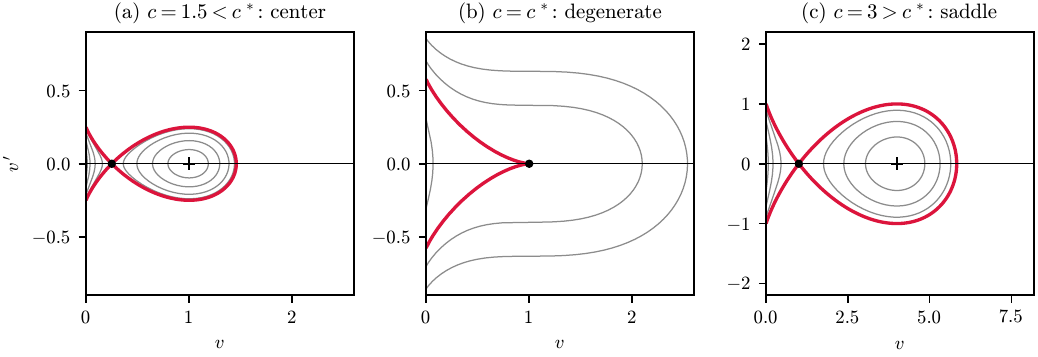}
\caption{Phase portraits for the K(2,2) traveling wave ODE in the
$(v,v')$ plane at shelf height $a = 1$, computed as level sets of the
energy $E = \tfrac12 (v')^2 + V(v)$ (Equation~(\ref{M-eq:shelf-energy})
of the main paper).  The critical speed $c^* = 2a = 2$ separates three
regimes; the thick curve is the level set through the equilibrium
marked by the dot, and $+$ marks a center.
(a)~$c = 1.5 < c^*$: nested periodic orbits around the center at
$v_0 = a^2$, bounded by the separatrix of a saddle at $(c-a)^2$.
(b)~$c = c^*$: the equilibrium at $v_0 = a^2$ is degenerate (an
inflection point of $V$).  No closed orbit exists, and the critical
level set has a cusp at $v_0$ with both branches running to $v = 0$
(Remark~\ref{M-rem:algebraic-soliton} of the main paper).
(c)~$c = 3 > c^*$: homoclinic orbit through the saddle at $v_0 = a^2$,
enclosing the center at $(c-a)^2$, with crest excursion beyond it
(see Equation~(\ref{M-eq:wpm}) of the main paper).}
\label{supp:fig:phase-portraits}
\end{figure}

\FloatBarrier
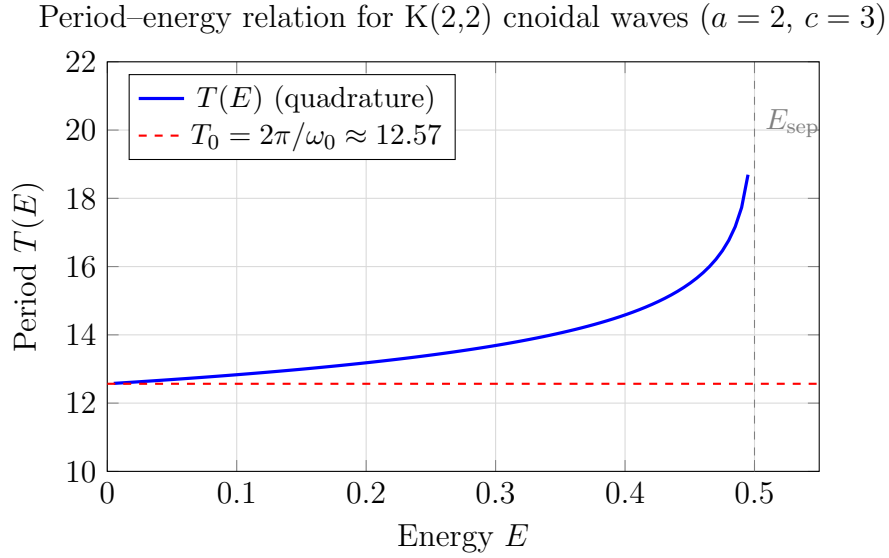
\begin{figure}[!htbp]
\centering
\pgfplotstableread[col sep=comma]{period_energy.csv}\perioddata
\begin{tikzpicture}
\begin{axis}[
  width=11cm, height=7cm,
  xlabel={Energy $E$},
  ylabel={Period $T(E)$},
  title={Period--energy relation for K(2,2) cnoidal waves
    ($a = 2$, $c = 3$)},
  xmin=0, xmax=0.55,
  ymin=10, ymax=22,
  grid=major,
  grid style={gray!30},
  legend style={at={(0.03,0.97)}, anchor=north west,
    font=\small},
]

\addplot[blue, very thick, mark=none] table[x=E, y=T] {\perioddata};
\addlegendentry{$T(E)$ (quadrature)}

\addplot[red, dashed, thick, domain=0:0.55]
  {12.5664};
\addlegendentry{$T_0 = 2\pi/\omega_0 \approx 12.57$}

\draw[gray, dashed] (axis cs:0.5,10) -- (axis cs:0.5,22)
  node[pos=0.85, right, font=\small] {$E_{\mathrm{sep}}$};

\end{axis}
\end{tikzpicture}
\caption{Period as a function of energy for the K(2,2)
one-parameter family of cnoidal waves on background $a = 2$
with speed $c = 3 < c^* = 4$.
The period approaches $T_0 = 2\pi/\omega_0$ (dashed line)
in the small-amplitude limit and diverges as the energy
approaches the separatrix value $E_{\mathrm{sep}}$,
where the periodic orbit merges with the homoclinic orbit.
This monotone $T(E)$ curve parametrizes the entire
one-parameter periodic family, in contrast with the
$N$-parameter tori of Lax's KdV theory.}
\label{supp:fig:period-energy}
\end{figure}

\FloatBarrier
\begin{figure}[!htbp]
\centering
\begin{tikzpicture}
\begin{axis}[
  width=10cm, height=7cm,
  xlabel={Shelf height $a$},
  ylabel={Speed $c$},
  title={Wave-type bifurcation diagram for K(2,2)},
  xmin=0, xmax=3,
  ymin=0, ymax=7,
  grid=major,
  grid style={gray!30},
  legend style={at={(0.03,0.97)}, anchor=north west,
    font=\small},
]

\addplot[red, very thick, domain=0:3, samples=100]
  {2*x};
\addlegendentry{$c^* = 2a$}

\node[font=\large] at (axis cs:2.2,2.0) {Periodic};
\node[font=\small,gray] at (axis cs:2.2,1.4)
  {(cnoidal waves)};

\node[font=\large] at (axis cs:0.8,5.0) {Solitary};
\node[font=\small,gray] at (axis cs:0.8,4.4)
  {(exponential tails)};

\draw[black, ultra thick] (axis cs:0,0) -- (axis cs:0,7);
\node[font=\small, rotate=90, anchor=north] at (axis cs:0.06,2.6)
  {Compactons ($a = 0$)};

\node[font=\small, rotate=33, anchor=south, red!70!black]
  at (axis cs:1.6,3.6) {solitary branch ends};

\end{axis}
\end{tikzpicture}
\caption{Bifurcation diagram for K(2,2) traveling waves in the
$(a, c)$ parameter plane.  The critical curve
$c^* = 2a$ (red) separates the solitary wave region
($c > c^*$, exponential tails) from the periodic wave region
($c < c^*$, cnoidal waves).  The left boundary $a = 0$
corresponds to compactons with compact support and finite
regularity $C^1$.  On the critical curve itself the solitary-wave
branch terminates: the amplitude vanishes as $c \to {c^*}^+$ and no
solitary wave exists at $c = c^*$ (Remark~\ref{M-rem:algebraic-soliton}
of the main paper).  For general
K(m,m), the critical curve is $c^* = m\,a^{m-1}$.}
\label{supp:fig:bifurcation}
\end{figure}
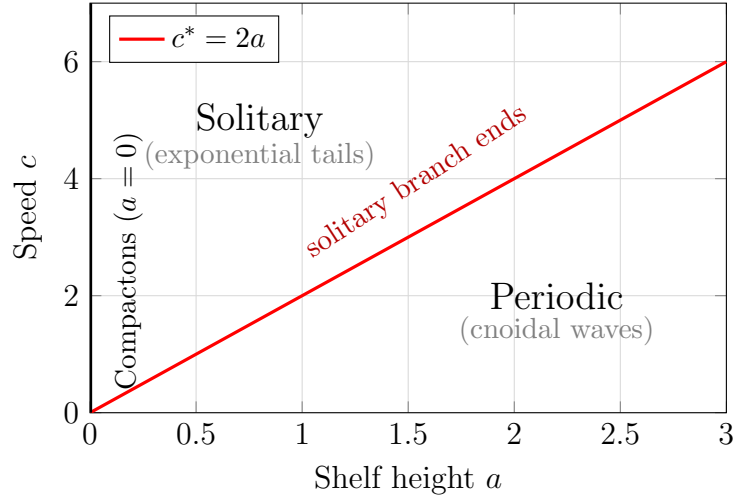

\section{Reproducibility}
\label{supp:sec:reproducibility}

\subsection{Software and environment}

All K($m$,$m$) simulations are implemented in Python~3.11 using
NumPy~1.24 and SciPy~1.11.  BBM simulations additionally use
SciPy's \texttt{solve\_ivp} with the RK45 integrator.  Symbolic
verification of conservation law identities and universal formulas
uses SymPy~1.13.  Figures are generated with Matplotlib~3.8.
All code has been tested on Python~3.10 (Google Colab)
and Python~3.11 (local, Ubuntu 22.04).  A \texttt{requirements.txt}
file is included in the repository for reproducible environment setup.

The computational notebook
\texttt{production\_multi\_m\_D3v4.ipynb} (included in the
repository) runs a finite-difference cross-check of the
K($m$,$m$) collisions ($N = 1024$, $q = 3$, $\delta = 0.5$) on Google Colab (Python~3.10,
typical runtime: 45 minutes on a T4 GPU instance, or 3 hours
on a CPU-only instance).  It is not the PSM configuration of
Table~\ref{supp:tab:kmm-full}, and its values differ from that table.

\subsection{Code repository}

Source code, data files, and Jupyter notebooks are
available at:
\begin{center}
\texttt{https://github.com/machyman/kmm-compactons}
\end{center}
A tagged release with a persistent DOI (via Zenodo) will be
created at the time of publication.
The repository contains:
\begin{itemize}
\item \texttt{papers/hyman2026structure/src/kmm\_compactons/}: the
  pseudospectral (PSM) K($m$,$m$) solver behind Section~\ref{M-subsec:quasi-elastic}
  and Figure~\ref{M-fig:collision-k22} of the main paper
\item \texttt{paper/}: figure sources for Figures~\ref{M-fig:roadmap}--\ref{M-fig:collision-k22}
  and~\ref{M-fig:wave-taxonomy}--\ref{M-fig:singular-limit} of the main paper and
  Figure~\ref{supp:fig:phase-portraits}, with their data in \texttt{paper/} and \texttt{paper/figdata/}
\item \texttt{src/kmm\_solver.py}: K($m$,$m$) regularized
  finite-difference solver with dual spectral filtering, used by the
  cross-check notebook
\item \texttt{bbm\_solver.py}: BBM pseudospectral solver
\item \texttt{collision\_analysis.py}: Post-collision amplitude
  extraction and debris spectral analysis
\item \texttt{shelf\_wave\_profiles.py}: Exact shelf-wave profile
  computation from quadrature formulas
\item \texttt{symbolic\_verification.py}: SymPy scripts for
  Appendix~C verification
\item \texttt{production\_multi\_m\_D3v4.ipynb}: Colab-compatible
  notebook running the finite-difference cross-check described above
\item \texttt{tests/}: Validation test suite (single-compacton
  propagation, conservation law verification)
\end{itemize}

\subsection{Data files}

The repository's data files:
\begin{itemize}
\item \texttt{measurements\_D3v4\_N1024.json}: the finite-difference
  cross-check run of the notebook above (all four $m$ values, 30+ measured
  quantities each), not the PSM data of Table~\ref{supp:tab:kmm-full}
\item \texttt{BBM\_Collision\_Results.json}: the BBM collision data of
  Tables~\ref{supp:tab:bbm-fixed-ratio}--\ref{supp:tab:bbm-large}
  (16 rows, 14 distinct configurations; $(2, 1)$ and $(10, 5)$ listed twice)
\end{itemize}

\bibliographystyle{plainnat}
\bibliography{Lax_Centennial_Paper_FINAL}